%% file: Rassoul-Agha.tex
\documentclass{amsproc}

\input{macros.tex}

\begin{document}

\setcounter{page}{95}

\title[Directed last-passage percolation]{Busemann functions, geodesics, and the competition interface for directed last-passage percolation}


\author{Firas Rassoul-Agha}
\address{University of Utah, Mathematics Department, 155 S 1400 E, Salt Lake City, UT 84109, USA}
\email{firas@math.utah.edu}
\thanks{The author was partially supported by NSF grant DMS-1407574.}



\date{\today}

\begin{abstract}
In this survey article we consider the directed last-passage percolation model on the planar square lattice with nearest-neighbor steps and general i.i.d.\ weights on the vertices, outside of the class of exactly solvable models. 
We show how stationary cocycles are constructed from queueing fixed points and how these cocycles characterize the limit shape, yield existence of Busemann functions in directions where the shape has some regularity,
describe the direction of the competition interface, and answer questions on existence, uniqueness, and coalescence of directional semi-infinite geodesics, and on nonexistence of doubly infinite geodesics.
\end{abstract}

\maketitle


\tableofcontents


\section{Introduction}

In 1965, Hammersley and Welsh \cite{Ham-Wel-65} introduced a model of a fluid flowing through a porous medium that they called
\index{first-passage percolation (FPP)}%
\index{FPP}%
first-passage percolation (FPP). Roughly speaking, the model consists of putting random positive numbers (called weights) on the nearest-neighbor edges of the square lattice $\Z^d$.
These numbers describe the time it takes to traverse the various edges. Only nearest-neighbor paths are allowed between lattice points.
The (random) distance between two  points is then given by the passage time -- the shortest time it takes to go between the  points. A path that realizes the passage time is naturally called a geodesic.
Immediate questions arise regarding the structure of the random metric, its balls, and its geodesics (finite, semi-infinite, and bi-infinite). 

As it is often the case in probability theory,  structure emerges from randomness when we look at the large scale behavior of the model. 
This means scaling the lattice down as we look at farther and farther sites. When done properly, the random metric on the scaled lattice converges to a deterministic one on $\R^d$.
As a result, random balls approach deterministic convex sets as their radii grow and geodesics from the origin to far away points approach straight lines. This raises the next layer of questions regarding 
the size and statistics of the fluctuations of the random objects from their deterministic limits.
See Section 1 of \cite{ch:Damron} for more details and precise formulations.

In a related model, called  directed last-passage percolation or just 
\index{last-passage percolation (LPP)}%
\index{LPP}%
last-passage percolation (LPP) paths are only allowed to take steps in $\{e_1,\dotsc,e_d\}$. One can think of  $e_1+\cdots+e_d$ as a time direction that cannot be reversed.
In fact, in LPP passage time is maximized rather than minimized and the random weights the path collects are usually put on the vertices instead of on the edges, but these are only minor differences from FPP. 
The crucial difference is that paths are now directed.
This creates some simplifications, but also introduces some complications. For example, in LPP it is clear that there exist maximizing paths 
\index{geodesic}%
(still called geodesics by analogy with LPP) from the origin to any point in the first quadrant. There are after all finitely many such paths, as opposed to the situation in FPP.
On the other hand, the passage time is no longer symmetric in LPP: there is no path back to the origin, starting from a point in the first quadrant.

Nevertheless, LPP and FPP are expected to a have similar qualitative behavior at large scales. One of the advantages the two-dimensional LPP has over its cousin FPP is a connection to other models, such as queues and particle systems. Furthermore, with certain specific weight distributions, LPP becomes 
what we call an exactly solvable model. This means that many computations can be carried out all the way down to exact formulas. In this case, one is able to extract very detailed information about the behavior of the model, which then drives or confirms predictions about general LPP and FPP (and related) models.

The focus of this article is on surveying recent developments in connection with the structure of geodesics in LPP. A main tool will be 
Busemann functions, originally used to study the large-scale geometry of geodesics in Riemannian manifolds.


The next section collects some notation for easier reference. Section \ref{LPP:intro} introduces the last-passage percolation model and Section \ref{sec:connection} discusses its connections to the corner growth model, queues in tandem, 
a system of interacting particles, and a competition model. In Section \ref{sec:shape} we describe the scaling limit of the passage time. Section \ref{sec:Bus} introduces the Busemann functions and studies their properties. To prove
existence of these functions we borrow tools from queuing theory. This is done in Section \ref{fixed-pt}. Busemann functions are then used in Section \ref{geodesics} to construct semi-infinite geodesics and describe their structure.
They are also used to study the competition interface, in Section \ref{cif:sec}. In the last two sections we give a brief account of the recent history of Busemann functions and geodesics in percolation and a preview of the
topic of the next two lecture notes \cite{ch:Seppalainen,ch:Corwin}: fluctuations.

\section{Notation}
$\Z$ denotes the integer numbers, $\Z_+$ the nonnegative integers, and $\N$ the positive integers. 
$\R$ denotes the real numbers and $\R_+$ the nonnegative real numbers. $\Q$ denotes the rational numbers.
$e_1$ and $e_2$ are the canonical basis vectors of $\R^2$. $\Uset=\{(t,1-t)=te_1+(1-t)e_2:0\le t\le1\}$ and its relative interior is 
$\riUset=\{(t,1-t):0<t<1\}$. We will write $\xi,\zeta,\eta$ for elements in $\Uset$.
$x_{i,j}=(x_i,\dotsc,x_j)$, for $-\infty\le i<j\le\infty$. $x\cdot y$ denotes the scalar product of vectors $x,y\in\R^2$.
$\abs{x}_1=\abs{x_1}+\abs{x_2}=\abs{x\cdot e_1}+\abs{x\cdot e_2}$. $\fl{a}$ is the largest integer $\le a$. If $a$ is a vector, then $\fl{a}$ is taken coordinatewise. For a 
real number $a$ we write $a^+$ for $\max(a,0)$ and $a^-$ for $-\min(a,0)$. Thus, $a=a^+-a^-$.
For $x,y\in\R^2$, $y\ge x$ and $y\le x$ mean the inequalities hold coordinatewise. 
We will say that a sequence $x_n$ is asymptotically directed into a subset $A\subset\Uset$
if all the limit points of $x_n/n$ are inside $A$.
A random variable $X$ has an exponential distribution with rate $\theta>0$ if $P(X>s)=e^{-\theta s}$ for all $s\ge0$.
The probability density function of such a variable equals $\theta e^{-\theta s}$ for $s\ge0$ and $0$ for $s<0$. 
The mean of $X$ then equals $E[X]=\theta \int_0^\infty s e^{-\theta s}\,ds=\theta^{-1}$. Its variance equals 
	\[E[(X^2-E[X])^2]=E[X^2]-E[X]^2=\theta \int_0^\infty s^2 e^{-\theta s}\,ds-\theta^{-2}=\theta^{-2}.\]
A random variable $X$ is said to have a continuous distribution if $P(X=s)=0$ for all $s\in\R$.

\section{Directed last-passage percolation (LPP)}\label{LPP:intro}
Consider the two-dimensional square lattice $\Z^2$. Put a random real number $\w_x$ at each site $x\in\Z^2$ and let these assignments be independent. 
In more precise terms,  
let $\Omega=\R^{\Z^2}$ and endow it with the product topology and the Borel $\sigma$-algebra. A generic element in $\Omega$ is denoted by $\w=\{\w_x:x\in\Z^2\}$ and called 
an 
\index{environment}%
{\sl environment} or a 
\index{configuration}%
{\sl configuration}. The numbers $\w_x$ are called 
\index{weights}%
{\sl weights}.
Let $\mu$ be a probability measure on $\R$ and let $\P$ be the product probability measure on $\Omega$ with all marginals equal to $\mu$: for a finite collection $x_1,\dotsc,x_n$
and measurable sets $A_1,\dotsc,A_n\subset\R$
	\[\P\big\{\w:\w_{x_i}\in A_i:1\le i\le n\big\}=\prod_{i=1}^n \mu(A_i).\]
(In particular $\P(a<\w_x\le b)=\mu((a,b])$ for all $x\in\Z^2$ and $-\infty\le a<b\le \infty$.) We will abbreviate the average weight value as $m_0=\E[\w_0]$.

When weights $\w_x$ are nonnegative they can be thought of as times spent at sites $x$.
Thus, the 
\index{passage time}%
{\sl passage time} of a path $x_0,\dotsc,x_n\in\Z^2$ is the time it takes to traverse the path and equals $\sum_{i=1}^n\w_{x_i}$. (We choose not to count the time spent at the very first point of the path.)
We will abbreviate $x_{i,j}$ for a sequence $x_i,x_{i+1},\dotsc,x_j$, with a similar notation for $x_{i,\infty}$, $x_{-\infty,j}$, and $x_{-\infty,\infty}$.

The model under consideration is directed, which means that we will only consider 
\index{up-right}%
{\sl up-right} paths, i.e.\ paths $x_{0,n}\in\Z^2$ with $x_{i+1}-x_i\in\{e_1,e_2\}$ for all $0\le i<n$.
We sometimes also call such paths 
\index{admissible}%
{\sl admissible}. If $x,y\in\Z^2$ are such that $x\le y$ coordinatewise, then 
the 
\index{last-passage time}%
{\sl last-passage time} between $x$ and $y$ is given by 
	\begin{align}\label{G:def}
	G_{x,y}=\max\Big\{\sum_{i=1}^n\w_{x_i}:x_{0,n}\text{ up-right, $x_0=x$, $x_n=y$, $n=\abs{y-x}_1$}\Big\}.
	\end{align}
Here, $\abs{\cdot}_1$ is the $\ell^1$ norm on $\R^2$. Let us immediately record an important inequality, called 
\index{superadditivity}%
{\sl superadditivity}: for $x\le y\le z$ coordinatewise
	\begin{align}\label{superadd}
	G_{x,y}+G_{y,z}\le G_{x,z}.
	\end{align}
Paths that maximize in \eqref{G:def} are called 
\index{geodesic}%
{\sl geodesics}. 
See Figure \ref{LPP:fig}. The above inequality comes simply from observing that concatenating a geodesic from $x$ to $y$ with one from $y$ to $z$ gives an up-right path from $x$ to $z$ (which may
or may not be a geodesic).

\begin{figure}
	\begin{center}
		\begin{tikzpicture}[ >=latex, scale=0.6]

			
			 \foreach \x in {0,...,9}{
              			 \foreach \y in {0,...,5}{
					\fill[color=mygray](\x,\y)circle(2mm);  
								}
							} 

			\draw[line width=3pt](9,5)--(6,5)--(6,4)--(5,4)--(5, 3)--(2, 3)--(2,1)--(1,1)--(1,0)--(0,0);
			\fill[color=black](9,5)circle(2mm);
			\fill[color=black](8,5)circle(2mm);
			\fill[color=black](7,5)circle(2mm);
			\fill[color=black](6,5)circle(2mm);
			\fill[color=black](6,4)circle(2mm);
			\fill[color=black](5,4)circle(2mm);
			\fill[color=black](5,3)circle(2mm);
			\fill[color=black](4,3)circle(2mm);
			\fill[color=black](3,3)circle(2mm);
			\fill[color=black](2,3)circle(2mm);
			\fill[color=black](2,2)circle(2mm);
			\fill[color=black](2,1)circle(2mm);
			\fill[color=black](1,1)circle(2mm);
			\fill[color=black](1,0)circle(2mm);
			\fill[color=black](0,0)circle(2mm);
			
			\draw (0,0)node[below left]{\Large$x$}; 
			\draw (9,5)node[above right]{\Large$y$}; 

		\end{tikzpicture}
	\end{center}
	\caption{\small  Illustration of a possible geodesic from $x$ to $y$.}
	\label{LPP:fig}
\end{figure}
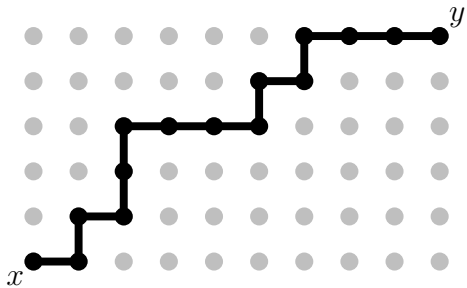

In what follows the math will make sense even when $\w_x$ can be negative and we continue to use the term 
\index{passage time}%
{\sl passage time} even with negative weights.

\begin{remark}
The above model is related to the standard  first\hyp{}passage site\hyp{}percolation model (FPP) as follows. Allow admissible paths to take steps in $\{\pm e_1,\pm e_2\}$, i.e.\ be nearest-neighbor paths. 
Then, paths from $x$ to $y$ no longer have to have a determined length and the last-passage time is defined by
	\[G_{x,y}=\max\Big\{\sum_{i=1}^n\w_{x_i}:x_{0,n}\text{ nearest-neighbor, $x_0=x$, $x_n=y$, $n\ge1$}\Big\}.\]
Replacing $\w$ with $-\w$ turns the $\max$ into a $\min$ and the model becomes the standard 
\index{first-passage percolation (FPP)}%
\index{FPP}%
first-passage percolation model. 
(Note that now weights must be taken nonnegative for otherwise $G_{x,y}$ will be infinite.)
In this case, the inequality in \eqref{superadd} is reversed and $G_{x,y}$ defines a (random) metric on $\Z^2$. It is by analogy with this situation that 
optimizing paths in the directed model are called geodesics.
\end{remark}

\section{Connections to other models}\label{sec:connection}
Say weights $\w_x$ are positive and for simplicity assume they have a continuous distribution, i.e.\ $\P(\w_x=s)=0$ for all $s\in\R$.
Then the last-passage percolation model we defined in Section \ref{LPP:intro} has several other equivalent descriptions. 

\subsection{Random corner growth model (CGM)}\label{cgm} 
\index{corner growth model (CGM)}%
\index{CGM}%

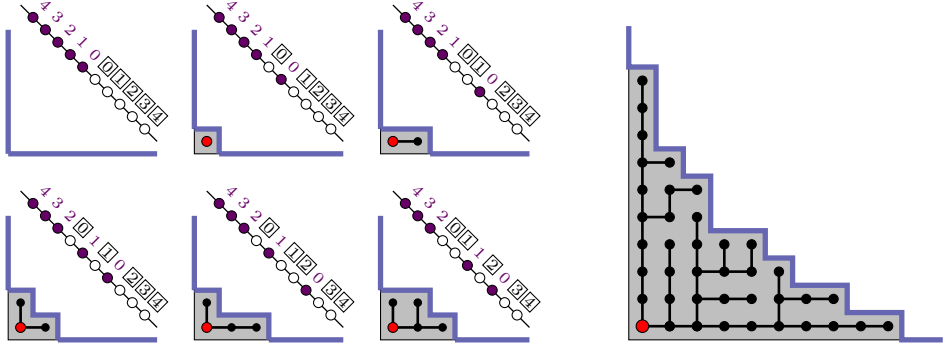
\begin{figure}
\centering
\input{CGM.tex}
\caption{\small 
A possible early evolution of the  corner growth model 
\index{corner growth model (CGM)}%
\index{CGM}%
on the first quadrant of the plane. Bullets mark infected lattice points $x$ with $G_{0,x}\le t$.  The origin is distinguished with the larger red bullet.
The gray region is the fattened set $\cB(t)+[-1/2,1/2]^2$.  
The thick purple down-right path is the height function that is the boundary of the fattened set $\cH(t)=(\Z^2_+\setminus\cB(t))+[-1/2,1/2]^2$.
The bold black edges  are the paths of minimal passage time from the origin. 
They are all directed.  
The antidiagonals illustrate the mapping of the 
\index{corner growth model (CGM)}%
\index{CGM}%
corner growth model to TASEP. 
Whenever a point is added to the growing infected cluster,  a particle (solid purple circle) switches places with the hole (open circle) to its right. For the queuing picture, boxed numbers indicate the service stations and numbers without a box are the customers.
When a customer and a station switch places, the customer has left that station and moved to the back of the line at the next station.}
\label{CGM:fig}
\end{figure}

For $x\in\Z^2$ set $G_{x,x}=0$. An infection starts at the origin and spreads into the first quadrant $\Z_+^2$. The set of infected sites at time $t \geq 0$ is
	\[\cB(t) = \{x \in \Z^2_+ : G_{0,x} \le t\}.\]
Let us see how $\cB(t)$ evolves. Since for a site $x=ke_i$, $k\in\N$ and $i\in\{1,2\}$, we have 
	\[G_{0,ke_i}=\sum_{j=1}^{k}\w_{je_i}>G_{0,(k-1)e_i};\] 
such a site cannot get infected before $(k-1)e_i$ is infected.
Similarly, for a site $x\in\N^2$ we have the induction
	\begin{align}\label{G:ind}
	G_{0,x}=\w_x+\max(G_{0,x-e_1},G_{0,x-e_2}).
	\end{align}
Thus, $G_{0,x}>G_{0,x-e_i}$ for both $i\in\{1,2\}$ and a site $x\in\N^2$ cannot be infected until after both $x-e_1$ and $x-e_2$ were infected.

There is a nice description of the evolution of $\cB(t)$ using the fattened set of heathy sites
	\[\cH(t)=(\Z^2_+\setminus\cB(t))+[-1/2,1/2]^2.\] 
See the sequence of  snapshots in Figure \ref{CGM:fig}.
Start with $\cB(0-)=\varnothing$ and 
	\begin{align}\label{H0-}
	\cH(0-)=\Z^2_++[-1/2,1/2]^2=\{x\in\R^2:x\ge-(e_1+e_2)/2\}.
	\end{align}
The boundary of $\cH(0-)$ is given by the path $\{h(s):s\in\R\}$ where 
	\begin{align}\label{h0-}
	h(s)=s^+ e_1+s^- e_2-(e_1+e_2)/2.
	\end{align}

At time $0$ the origin $x=0$ is infected, $\cB(0)=\{0\}$, $\cH(0)=\cH(0-)\setminus[-1/2,1/2)^2$, and the boundary of $\cH(0)$ is 
obtained from that of $\cH(0-)$ by flipping the south-west corner located at $-(e_1+e_2)/2$ to a north-east corner, creating two new south-west corners at $(e_1-e_2)/2$ and $(e_2-e_1)/2$.

For concreteness, say $\w_{e_1}<\w_{e_2}$. Then $e_1$ is the next site to become infected, exactly at time $t_1=\w_{e_1}>0$. 
We have $\cB(t)=\{0\}$ for $0\le t<t_1$ and $\cB(t_1)=\{0,e_1\}$. Region $\cH$ gets another square taken away: 
$\cH(t_1)=\cH(0)\setminus(e_1+[-1/2,1/2)^2)$. Its boundary changes by the south-west corner at $(e_1-e_2)/2$ getting flipped into a north-east corner. 

Site $(k,\ell)$ becomes infected at time $t=G_{0,(k,\ell)}$. Right before this time, the boundary of $\cH(t-)$ has a south-west corner at $ke_1+\ell e_2-(e_1+e_2)/2$ that then flips into a north-east corner.
Hence the name 
\index{corner growth model (CGM)}%
\index{CGM}%
``corner growth''. 

The evolution is particularly nice when weights $\w_x$ have an exponential distribution, i.e.\ when $\P(\w_x>s)=\mu((s,\infty))=e^{-\theta s}$ for some $\theta>0$ and all $s\in\R_+$. Parameter $\theta$ is called the rate
of the exponential random variable.

In this special case the evolution goes as follows.
Given set $\cB(t_0)$ at some point in time $t_0\ge0$ consider the south-west corners on the boundary of $\cH(t_0)$. (There are always finitely many such corners.)
Assign to these corners independent random variables, exponentially distributed with the same rate $\theta$ as weights $\w_x$.
Think of these variables as the time an ``alarm clock'' goes off at the corner the variable is assigned to. When the first of these clocks rings the corresponding south-west corner gets flipped to a north-east corner. At that point in time, we have 
a new $\cH$ and the procedure is repeated. The mathematics of this evolution is quite clear in the next equivalent description.

\subsection{Queues in tandem}\label{queues} The queueing interpretation of LPP in terms of tandem service stations goes as follows.  
Imagine a queueing system with customers labeled  by $\Z_+$ and service stations also labeled by $\Z_+$.   
The random weight $\w_{k,\ell}$ is the service time of customer $k$ at station $\ell$.  Right before time $0$ all customers  are  lined up at service station $0$ and customer $0$ is first in line and has just been served.  
At time $t=0$ customer $0$ is first in line at queue $1$ and the rest of the customers are still at queue $0$ with customer $1$ being first in line there, then customer $2$, and so on. Service of customers $0$ and $1$ begins.
Customers proceed through the system in order, obeying FIFO (first-in-first-out) discipline, 
and joining the queue at station $\ell+1$ as soon as service at station $\ell$ is complete.  Once customer $k\ge0$ is first in line at station $\ell\ge0$, it takes $\w_{k,\ell}$ time units to perform service. 
(The only exception is $k=\ell=0$ where customer $0$ advances immediately from queue $0$ to queue $1$.) 

For each $k\ge0$ and $\ell\ge0$,
$G_{0,(k,\ell)}$ is the time when customer $k$ departs station $\ell$ and joins the end of the queue at station $\ell+1$. To see this observe that on the one hand this is clear when $k=0$ or $\ell=0$ and, on the other hand, this departure time satisfies the same recurrence \eqref{G:ind}. 
Indeed, the departure time of customer $k$ from station $\ell$ is equal to the service time $\w_{k,\ell}$ plus the time $G_{0,(k-1,\ell)}$ when customer $k-1$ departs station $\ell$  or the time customer $k$ departs station $\ell-1$, whichever is larger.
(If the former time is larger, customer $k$ will have to wait for customer $k-1$ before service starts at station $\ell$. If instead the latter time is the larger of the two, then station $\ell$ will be empty when customer $k$ arrives and service starts immediately.)
In terms of the 
\index{corner growth model (CGM)}%
\index{CGM}%
corner growth model, this is exactly when site $(k,\ell)$ gets infected. See Figure \ref{CGM:fig}.
 Among the seminal references for these ideas are \cite{Gly-Whi-91,Mut-79}. 
 
 When weights are exponentially distributed, the description is again quite transparent. At every point in time there are only finitely many non-empty queues. Assign to the first customer in each of these queues an independent
 random variable (a clock) with exponential distribution having the same rate as weights $\w_x$. The customer whose clock rings first has been served and moves on to the end of the next station and then the procedure repeats.

\subsection{Totally asymmetric simple exclusion process (TASEP)}\label{tasep} 
In this model, a configuration is an assignment of $0$s and $1$s to the integers $\Z$. More precisely, it is a function $\eta:\Z\to\{0,1\}$. Think of $\eta_j=1$ as a particle occupying site $j\in\Z$ and
then $\eta_j=0$ means $j$ is empty (or a hole).  Given a configuration $\eta$ such that $\exists j_0$ with $\eta_j=0$ for all $j\ge j_0$ define a curve (known as a 
\index{height function}%
{\sl height function}) 
$h:\R\to\R^2$ by $h(j_0-1/2)=j_0 e_1-(e_1+e_2)/2$, $h(j+1/2)-h(j-1/2)=(1-\eta_j)e_1-\eta_j e_2$ for all $j\in\Z$, and linear interpolation on $\R\setminus(\Z+1/2)$. 
 (This is well defined, regardless of the choice of $j_0$.)

Right before time $t=0$ we start by placing a particle at every site $j\le-1$ and leaving sites $j\ge0$ empty.  In other words, $\eta_j(0-)=\one\{j<0\}$. The corresponding height function is given by \eqref{h0-}, i.e.\ it is the boundary of 
$\cH(0-)$ from \eqref{H0-}.  At time $t=0$ the particle at $j=-1$ jumps to the empty site $j=0$, leaving site $j=-1$ empty. Now, the corresponding height function is given by the boundary of $\cH(0)$.

As $t$ grows particles move around. At any point in time only one particle is allowed to make a move and it can only move one step to its right, if there is no other particle there already. 
Here is a more precise description of the particle dynamics.
Particles move only at times when $\cB(t)$ changes (i.e.\ when new sites are infected). 
Think of the boundary of $\cH(t)$ as a height function. Then, 
there is a one-to-one correspondence between south-west corners in the boundary of $\cH(t)$ and particle-hole pairs (the hole being immediately to the right of the particle).
When a south-west corner in the boundary of $\cH(t)$ is flipped, the corresponding particle jumps one step to its right, switching positions with the hole that was there. See Figure \ref{CGM:fig}.

Comparing this description to the queuing system we see that holes play the role of service stations and particles to the left of a hole play the role of customers in line at that service station.

Once again, when the weights are exponentially distributed the evolution can be described in a Markovian way using exponential clocks. Given a configuration at some time $t_0$ there are only finitely many particles that have a hole immediately to their right.
Each of these particles is given an independent exponential random variable with the same rate as weights $\w_x$ and the particle whose clock rings first moves
one position to the right, effectively switching places with the hole that was there. Then the process is repeated.  This is one of the most fundamental interacting particle systems. 
See \cite{Spi-70,Mac-Gib-Pip-68} for two of the earliest papers on the model.\medskip

Here is a table that summarizes the meaning of $G_{0,(k,\ell)}$ in the models in Sections \ref{cgm}-\ref{tasep}.

\renewcommand{\arraystretch}{1.2}
\begin{center}
    \begin{tabular}{|l||l|}
    \hline
    {\bf Model} & {$\bm{G_{0,(k,\ell)}}$} {\bf is the time when:} \\ \hline\hline
    CGM & site $(k,\ell)$ is infected\\ \hline
    Queues  & customer $k$ clears server $\ell$\\ \hline
    TASEP & particle $k$ exchanges places with hole $\ell$\\ \hline
    \end{tabular}
\end{center}

\subsection{The competition interface (CIF)}\label{cif-intro}
Start  two infections at $e_1$ and at $e_2$ and let sites get infected as before. Mark sites infected by $e_1$ purple and those infected by $e_2$ green. Once a site is infected by one of the two types,
it remains like that forever. This partitions the first quadrant $\Z_+^2\setminus\{0\}$ into two regions of infection. 

Mathematically, recall induction \eqref{G:ind}. Since we assumed weights to be independent and with a continuous distribution, equality $G_{0,x-e_1}=G_{0,x-e_2}$ happens with zero probability.
Thus, 
	\[G_{0,x}=\w_x+G_{0,x-e_i}\]
for exactly one of $i\in\{1,2\}$. This indicates who infected site $x$.

Another description of the spread of infection comes using geodesics. Again, because weights are independent and have a continuous distribution, there is a unique geodesic between any two distinct sites $x\le y$.
As such, the union of all geodesics from $0$ to sites $x\in\Z^2_+\setminus\{0\}$ forms a spanning tree of $\Z^2_+$ that represents the genealogy of the infection. 
The subtrees rooted at $e_1$ and $e_2$ are precisely the vertices infected by these two sites, respectively. They are separated by an up-right path on the dual lattice $\Z^2+(e_1+e_2)/2$ called the 
\index{competition interface}%
{\sl competition interface}. 
See Figure \ref{CIF:fig}.  Properties of this  interface, as well as references for further reading, are in Section \ref{cif:sec}.

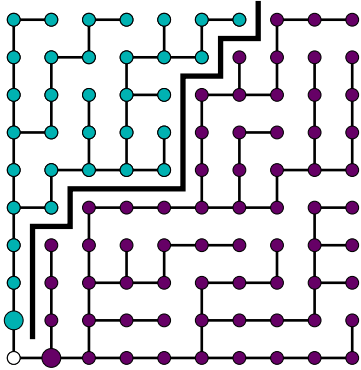
\begin{figure}
\begin{center}

\input{CIF.tex}

\end{center}

\caption{\small The full tree of infection in the 
\index{corner growth model (CGM)}%
\index{CGM}%
corner growth model.  The origin is the open circle at the bottom left.   The solid black line marks the 
\index{competition interface}%
competition interface  that separates the two competing infections that  
grow from points $e_1$ and $e_2$ marked with larger circles.  
}
\label{CIF:fig}
\end{figure}

\section{The shape function}\label{sec:shape}

One of the central questions in probability theory is to describe the order that emerges out of randomness as the number of random inputs into the system grows.
For instance, the law of large numbers says that if $\{X_n:n\in\N\}$ are independent random variables that are identically distributed
(i.e.\ random samples from the same population), then the sample or empirical mean $(X_1+\cdots+X_n)/n$ converges, with probability one, to the (population) mean $E[X_1]$ (provided this mean is well defined). 
Even though $G_{0,x}$ is not simply a sum
of independent random variables (it has a maximum), the law of large numbers raises the question of whether or not $G_{0,x}$ should grow at most linearly with $\abs{x}_1$. This is indeed the case.
For $a\in\R$ let $\fl{a}$ be the largest integer no greater than $a$. For $x\in\R^2$ let $\fl{x}$ act coordinatewise.  Recall that $m_0=\E[\w_0]$.

\index{shape theorem}%
\begin{theorem}\label{th:shape}
{\rm\cite{Mar-04}} Assume $\E[\abs{\w_0}]<\infty$. Then for any $\xi\in\R_+^2$ the limit
	\begin{align}\label{shape-sub}
	g(\xi)=\lim_{n\to\infty}\frac{G_{0,\fl{n\xi}}}n
	\end{align}
exists almost surely and {\rm(}if $g(\xi)<\infty${\rm)} in $L^1$. It is deterministic, $1$-homogenous,  concave, and satisfies $g(x_1,x_2)=g(x_2,x_1)$, $x_1,x_2\in\R_+$, 
$g(e_1)=g(e_2)=m_0$, and $g(\xi)\ge m_0\,\abs{\xi}_1$. 
If furthermore
	\begin{align}\label{w-bound}
	\int_0^\infty\!\!\!\sqrt{\P(\w_0>s)}\,ds<\infty
	\end{align} 
{\rm(}e.g.\ if $\E[\abs{\w_0}^{2+\e}]<\infty$ for some $\e>0${\rm)}, then $g$ is finite and continuous on all of $\R_+^2$ and
	\begin{align}\label{shape}
	\lim_{n\to\infty}\max_{x\in\Z_+^2:\abs{x}_1=n}\frac{\abs{G_{0,x}-g(x)}}n=0\quad\text{almost surely.}
	\end{align}
\end{theorem}

In particular, limit \eqref{shape} says that the fattened set $(\cB(t)+[-1/2,1/2])/t$ converges almost surely, as $t\to\infty$, to the set $\{x\in\R_+^2:g(x)\le1\}$. Thus, \eqref{shape} is called a 
\index{shape theorem}%
{\sl shape theorem} and
$g$ is the 
\index{shape function}%
{\sl shape function}. See Figure \ref{shape:fig}.

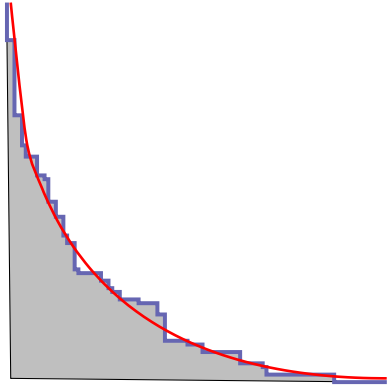
\begin{figure}
\centering
\input{shape.tex}
\caption{\footnotesize
LPP with exponentially distributed vertex weights with rate $1$. 
The gray region is a simulation of the scaled growing  set  $t^{-1}(\cB(t)+[-1/2,1/2]^2)$ at time $t=160$.
Its  boundary (the thick blue down-right path) approximates the smooth red limit curve 
$\{x\in\R_+^2:\sqrt{x\cdot e_1} +\sqrt{x\cdot e_2}=1\}$, 
as first proved by Rost in 1981.}
\label{shape:fig}
\end{figure}

\begin{proof}[Proof of Theorem \ref{th:shape}]
Let us start by considering the case $\xi\in\Z_+^2$. By superadditivity \eqref{superadd} we have for $m\le n$
	\[G_{0,m\xi}+G_{m\xi,n\xi}\le G_{0,n\xi}.\]
If we had additivity instead of superadditivity, then we could write
	\[G_{0,n\xi}=\sum_{i=0}^{n-1}G_{i\xi,(i+1)\xi}.\]
The summands are 
\index{i.i.d.}%
{\sl independent and identically distributed} (i.i.d.) and as such $n^{-1}\sum_{i=0}^{n-1}G_{i\xi,(i+1)\xi}$ is the sample mean of
random samples of $G_{0,\xi}$. A generalization of the law of large numbers, called the ergodic theorem, tells us then that this sample mean converges to the population mean 
$\E[G_{0,\xi}]$.

Unfortunately, additivity does not hold. However, it turns out that one can prove a stochastic version of Fekete's subadditive lemma and apply this 
\index{subadditive ergodic theorem}%
{\sl subadditive ergodic theorem} to 
$-G_{0,n\xi}$ to obtain the limit \eqref{shape-sub}. A version of this theorem is given in \cite{ch:Damron} as Theorem 3.2.  
A bit more work is needed for the general case $\xi\in[0,\infty)^2$ and even more work is needed to get more uniform control
and prove \eqref{shape}. The details are omitted as they are similar to the ones in the proof of the standard first-passage percolation shape theorem, given in \cite{ch:Damron}. See also Proposition 2.1(i) of \cite{Mar-04}.

Symmetry of $g$ follows from that of the lattice and the fact that  i.i.d.\ random variables are exchangeable (i.e.\ switching them around does not change the joint distribution).
Since $G_{0,ne_i}=\sum_{k=0}^{n-1}\w_{ke_i}$, $i\in\{1,2\}$, \eqref{shape-sub} and the law of large numbers give $g(e_1)=g(e_2)=\E[\w_0]=m_0$. 
Consider next the up-right path $x_{0,\fl{n\xi}}$ from $0$ to $\fl{n\xi}$ that first takes $\fl{n\xi_1}$ $e_1$-steps and then $\fl{n\xi_2}$ $e_2$-steps. We have
	\[n^{-1} G_{0,\fl{n\xi}}\ge n^{-1}\sum_{i=0}^{\fl{n\xi_1}-1}\w_{ie_1}+n^{-1}\sum_{i=0}^{\fl{n\xi_2}-1}\w_{\fl{n\xi_2}+ie_2}.\]
By \eqref{shape-sub} the left-hand side converges to $g(\xi)$ almost surely and hence also in probability.
By the weak law of large numbers, the two sums on the right-hand side converge in probability to $m_0\xi_1$ and $m_0\xi_2$, respectively. It follows that $g(\xi)\ge m_0\abs{\xi}_1$.

Finiteness of $g(\xi)$ comes easily if weights $\w_x$ are bounded above by some constant, i.e.\ if $\P(\w_0\le c)=1$ for some $c>0$ then clearly $G_{0,\fl{n\xi}}\le cn\abs{\xi}_1$ and thus $g(\xi)\le c\abs{\xi}_1$. 
More generally, finiteness would follow from the fact that $g$ is concave, homogenous, and continuous on $\R_+^2$ (all the way up to the boundary).

Let us now prove the regularity properties claimed in the theorem. Homogeneity of $g$ comes simply from
	\[g(c\xi)=\lim_{n\to\infty}\frac{G_{0,\fl{nc\xi}}}n=c\lim_{n\to\infty}\frac{G_{0,\fl{cn\xi}}}{cn}=cg(\xi),\quad\text{for }c>0.\]
Then concavity follows from homogeneity and superadditivity: for $\alpha\in(0,1)$
	\[\alpha g(\xi)+(1-\alpha)g(\zeta)=g(\alpha\xi)+g((1-\alpha)\zeta)\]
	and for $x,y\in\R_+^2$
	\begin{align*}
	&g(x)+g(y)\\ 
	&\quad=\lim_{n\to\infty}\frac{G_{0,\fl{nx}}}n+\lim_{n\to\infty}\frac{G_{0,\fl{ny}}}n&\text{(almost surely)}\\
	&\quad=\lim_{n\to\infty}\frac{G_{0,\fl{nx}}}n+\lim_{n\to\infty}\frac{G_{\fl{nx},\fl{ny}+\fl{nx}}}n&\text{(in probability, due to shift-invariance)}\\
	&\quad\le\lim_{n\to\infty}\frac{G_{0,\fl{ny}+\fl{nx}}}n&\text{(by supperadditivity \eqref{superadd})}\\
	&\quad=g(x+y).
	\end{align*}
In the second equality, we used the fact that if one shifts the picture, placing the origin where say $z$ used to be, then the two situations are statistically equivalent: for all $z\in\Z^2$ and $x\ge0$, $G_{z,z+x}$ has the same distribution as $G_{0,x}$.
It remains to prove continuity. We will do so under the assumption that weights are bounded in absolute value. This will capture the essence of the argument. The general case requires some extra technical work that we will avoid.  The interested reader can find the details in the proof of Proposition 2.2 of \cite{Mar-04}. 

Fix an $\e\in\Q\cap(0,1/2)$ and an integer $k$ such that $k\e\in\N$.
The above computation gives us	
\begin{align}\label{cont1}
g(e_2+\e e_1)-g(e_2)\ge g(\e e_1)=\e g(e_1)=\e m_0.
\end{align}
We next prove a similar upper bound.

Each up-right path from $0$ to $kn e_2+\e kn e_1$ consists of $kn$ $e_2$-steps and $\e k n$ $e_1$-steps. Thus, there are 
${nk(1+\e)\choose nk\e}$ such paths. By Stirling's approximation $N!\sim\sqrt{2\pi N}N^Ne^{-N}$ we have
that \[{nk(1+\e)\choose nk\e}\sim\sqrt{\frac{1+\e}{2\pi nk\e}}e^{nkh(\e)},\]
where $h(\e)=(1+\e)\log(1+\e)-\e\log\e$.
Next, fix $\delta>0$ and use a union bound (i.e.\ that $\P(\cup_j A_j)\le\sum_j\P(A_j)$) and $m_0=g(e_2)$ to write
 	\begin{align*}
	\P\bigl(G_{0,kne_2+\e kne_1}\ge nk g(e_2)+nk\delta\bigr)
	\le\sum_{\substack{x_{0,nk(1+\e)}\\ \text{ up-right}}}\P\Bigl(\sum_{i=1}^{nk(1+\e)}\w_{x_i}\ge nkg(e_2)+nk\delta\Bigr)&\\
	=\sum_{\substack{x_{0,nk(1+\e)}\\ \text{ up-right}}}\P\Bigl(\sum_{i=1}^{nk(1+\e)}(\w_{x_i}-m_0)\ge -nk\e m_0+nk\delta\Bigr).&
	\end{align*}
Now observe that once the up-right path $x_{0,nk(1+\e)}$ is fixed, weights $\w_{x_i}$ are i.i.d.\ with mean $m_0$. Thus, all the probabilities in the last sum have the same value and we can continue by writing
 	\begin{align}\label{inter}
	\begin{split}
	&\P\bigl(G_{0,kne_2+\e kne_1}\ge nk g(e_2)+nk\delta\bigr)\\
	&\qquad\qquad\le\sqrt{\frac{1+\e}{2\pi nk\e}}e^{nkh(\e)}P\Bigl(\sum_{i=1}^{nk(1+\e)}Z_i\ge nk(\delta-\e m_0)\Bigr),
	\end{split}
	\end{align}
where $Z_i$ are i.i.d.\ centered (and bounded) random variables with the same distribution as $\w_0-m_0$. 

If $m_0>0$, then pick $\e\in(0,\delta/m_0)$. Otherwise, just pick $\e>0$. Then the event in the last probability contradicts the law of large numbers which says that the sample mean $\frac1{nk(1+\e)}\sum_{i=1}^{nk(1+\e)}Z_i$ should be 
close to $0$. In fact, large deviation theory tells us that the probability of such an event decays exponentially fast. Indeed, fix a positive number $\lambda>0$ and use Chebyshev's exponential inequality to write
	\begin{align*}
	P\Bigl(\sum_{i=1}^{nk(1+\e)}Z_i\ge nk(\delta-\e m_0)\Bigr)
	&\le E\bigl[e^{\sum_{i=1}^{nk(1+\e)}\lambda Z_i-\lambda nk(\delta-\e m_0)}\bigr]\\
	&=e^{-\lambda nk(\delta-\e m_0)}E[e^{\lambda Z_0}]^{nk(1+\e)}\\
	&= \exp\Bigl\{-nk\Bigl(\lambda(\delta-\e m_0)-(1+\e)\log E[e^{\lambda Z_0}]\Bigr)\Bigr\}.
	\end{align*}
A little calculus exercise shows that $e^x\le1+x+x^2$ for $x$ close to $0$. Use this and $\log(1+x)\le x$ (valid for all $x$) to continue the above computation, remembering that $E[Z_0]=0$,
	\begin{align*}
	&P\Bigl(\sum_{i=1}^{nk(1+\e)}Z_i\ge nk(\delta-\e m_0)\Bigr)\\
	&\qquad\le \exp\Bigl\{-nk\Bigl(\lambda(\delta-\e m_0)-(1+\e)\log \bigl(1+\lambda^2 E[Z_0^2]\bigr)\Bigr)\Bigr\}\\
	&\qquad\le \exp\Bigl\{-nk\Bigl(\lambda(\delta-\e m_0)-(1+\e)\lambda^2 E[Z_0^2]\Bigr)\Bigr\}.
	\end{align*}
Take $\lambda=(\delta-\e m_0)/(2+2\e)$ and the above becomes
	\[P\Bigl(\sum_{i=1}^{nk(1+\e)}Z_i\ge nk(\delta-\e m_0)\Bigr)\le e^{-(\delta-\e m_0)^2nk/(4+4\e)}.\]
Going back to \eqref{inter} we have
	\begin{align*}
	&\P\bigl(G_{0,kne_2+\e kne_1}\ge nk g(e_2)+nk\delta\bigr)\\
	&\qquad\qquad\le \sqrt{\frac{1+\e}{2\pi nk\e}}\exp\Bigl\{-nk\Bigl(\frac{(\delta-\e m_0)^2}{4(1+\e)}-(1+\e)\log(1+\e)+\e\log\e\Bigr)\Bigr\}.
	\end{align*}
Taking $\e$ small enough makes the right-hand side converge exponentially fast to $0$ as $n\to\infty$. Combining this with the fact that $G_{0,kne_2+\e kne_2}/(nk)$ converges almost surely to $g(e_2+\e e_1)$ gives
	\[g(e_2+\e e_1)-g(e_2)\le\delta.\]
Together with \eqref{cont1} this proves continuity of $a\mapsto g(e_2+a e_1)$ at $a=0$. Symmetry gives the same result with $e_1$ and $e_2$ switched around.  This and the homogeneity of $g$ imply its continuity at the boundary of $\R^2_+$.
Continuity in the interior is a known fact about concave functions.
\end{proof}

When weights $\w_x$ are exponentially distributed (with rate $\theta>0$) one can get an explicit formula for the 
\index{shape function!explicit}%
shape:
	\begin{align}\label{shape:exp}
	g(\xi_1,\xi_2)=m_0(\xi_1+\xi_2)+2\sigma_0\sqrt{\xi_1\,\xi_2}\,,\quad\xi=(\xi_1,\xi_2)\in\R_+^2.
	\end{align}
Here, $m_0=\theta^{-1}$ is the mean of $\w_x$ and $\sigma_0=\theta^{-1}$ is its standard deviation. (The two quantities are equal for exponential random variables, but
this is not true in general.) 

A similar formula also holds when weights $\w_x$ are geometrically distributed, i.e.\ when $\mu$ is supported on $\N$ and for some $p\in(0,1)$ and all $j\in\N$, $\P(\w_x=j)=\mu(\{j\})=p^{j-1}(1-p)$.
In the exponential case this formula was first derived by Rost \cite{Ros-81} (who presented the model in its coupling with TASEP without the last-passage formulation)  while early derivations of the geometric case appeared in  \cite{Coh-Elk-Pro-96, Joc-Pro-Sho-98, Sep-98-mprf-1}.  
We will prove the formula at the end of Section \ref{fixed-pt} using equation \eqref{B-g} and the explicit knowledge of the distribution of Busemann functions.
See also  Theorem 3.4 in \cite{ch:Seppalainen}.

Other than the above two cases, no explicit formula is known for $g$. However, it is known that the above formula does hold in general near the boundary.

\begin{theorem}\label{th:Martin}
{\rm\cite{Mar-04}} Assume
	\[\int_0^\infty\!\!\!\sqrt{\P(\w_0>s)}\,ds<\infty\quad\text{and}\quad \int_{-\infty}^0\!\!\!\sqrt{\P(\w_0\le s)}\,ds<\infty.\]
Let $m_0=\E[\w_0]$ and $\sigma_0^2=\E[\w_0^2]-m_0^2$.
Then 
	\begin{align}\label{martin}
	g(1,\alpha)=m_0+2\sigma_0\sqrt{\alpha}+o(\sqrt\alpha)\quad\text{as }\alpha\searrow0.
	\end{align}
\end{theorem}

The fact that $g(1,\alpha)-m_0$ is of order $\sqrt{\alpha}$ for small $\alpha$ comes by a more careful look at the  continuity proof above. An even more careful look involving comparison to the case with exponential weights gives the more precise formula \eqref{martin}.

The above has a nice consequence regarding the limiting shape.

\begin{corollary}
The limiting shape $\{x\in\R_2^+:g(x)\le1\}$ is not a polygon {\rm(}with finitely many sides{\rm)}.
\end{corollary}


It is tempting to think that perhaps formula \eqref{shape:exp} holds in general. The following situation shows that this is not the case.

Assume that the LPP weights satisfy $\w_x\le 1$ and $p=\P\{\w_0=1\}>0$. The classical Durrett-Liggett flat edge result implies that if $p$ is large enough,  $g$ is linear on a whole cone. 
See Figure \ref{flat:fig}.

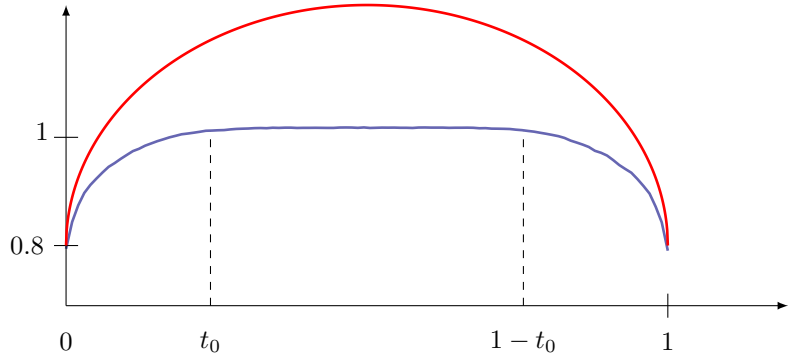
\begin{figure}
\centering
\input{flat.tex}
\caption{\footnotesize
LPP with Bernoulli distributed vertex weights: $\P(\w_0=1)=1-\P(\w_0=0)=0.8$. 
The bottom blue line is a simulation of the curve $t\mapsto g(te_1+(1-t)e_2)$, $t\in[0,1]$.
The top red curve is $t\mapsto 0.8+\sqrt{0.8\times0.2t(1-t)}$, according to formula \eqref{shape:exp}. 
The two curves are different: the top red one is strictly concave while the bottom blue one is flat on $[t_0,1-t_0]$. 
But their asymptotics match near $t=0$ and $t=1$.}
\label{flat:fig}
\end{figure}

\begin{theorem} \label{flat:thm}
{\rm\cite{Dur-Lig-81}}
Suppose $\E[\abs{\w_0}^{2+\e}]<\infty$ for some $\e>0$ and $\P\{\w_x\le 1\}=1$.
There exists a critical value $p_c\in(0,1)$ such that
if $p=\P\{\w_0=1\}>p_c$,  then there exists $t_0\in(0,1/2)$ such that $g(\xi)=\abs{\xi}_1=\xi\cdot(e_1+e_2)$ for all $\xi\in\R_+^2$ with $\xi\cdot e_1/\abs{\xi}_1\in[t_0,1-t_0]$.
\end{theorem}

Here is a heuristic argument to help understand why the above result holds. When  probability $p$ is large one can show that there is 
an infinite up-right path $x_{0,\infty}$ starting from the origin $x_0=0$ such that $\exists n_0$ with $\w_{x_n}=1$ for all $n\ge n_0$.
Furthermore, one can guarantee that $x_n\cdot e_1/n\to t_0$ for some $t_0\in(0,1/2)$ as $n\to\infty$. The shape theorem then implies that $g(t_0,1-t_0)=1$. By concavity  of $g$
we have that $g(t,1-t)=1$ for all $t\in[t_0,1-t_0]$ and the claim of the theorem follows from the homogeneity of $g$.\medskip

Before we close the section, it may be noteworthy that even though no closed formulas are known for $g$ in general, some variational characterizations of $g$ do exist in the general weights setting.
See Section 2 of \cite{ch:Seppalainen}.

\section{Busemann functions}\label{sec:Bus}

How does one prove \eqref{shape:exp}? Can one get any information on $g$, more than what is given by Theorem \ref{th:shape}?\medskip

One way to approach such questions is by taking a closer look at the shape function. In precise terms, abbreviate
	\[\Uset=\{(t,1-t):t\in[0,1]\}\quad\text{and}\quad\riUset=\{(t,1-t):0<t<1\}\]
and fix a $\xi\in\riUset$. Consider the collection of random variables
	\[\{G_{0,\fl{n\xi}}-G_{0,\fl{n\xi}+z}:\abs{z}_1\le M\},\]
for some $M>0$. In other words,  we want to examine the passage times to points in the vicinity of $n\xi$, relative to the passage time to $n\xi$ itself.
Presumably, as $n\to\infty$, (the distribution of) this vector of random variables converges weakly to some limit and this limit carries some useful information about the large scale behavior of the system ``in direction $\xi$''.

Since the point of reference is moving, there is really no hope of the convergence being almost sure. However, we can change our frame of reference and view things from $n\xi$ by considering
	\[\{G_{-\fl{n\xi},0}-G_{-\fl{n\xi},z}:\abs{z}_1\le M\},\]
or equivalently 
	\[\{G_{-\fl{n\xi},0}-G_{-\fl{n\xi},-z}:\abs{z}_1\le M\}.\]
Since $\{\w_{-x}:x\in\Z^2\}$ has the same distribution as $\{\w_x:x\in\Z^2\}$, the above has the same law as
	\begin{align}\label{(*)}
	\{G_{0,\fl{n\xi}}-G_{z,\fl{n\xi}}:\abs{z}_1\le M\}.
	\end{align}
The advantage now is that there is a chance this random vector does converge to an almost sure limit. This is indeed the case under some mild regularity assumption on the shape $g$.	

One little technicality: since when we defined $G_{x,y}$ in \eqref{G:def} we left out the weight $\w_x$, going through the above reflection argument leads us to slightly change our definition to become
	\begin{align}\label{G:def2}
	G_{x,y}=\max\Big\{\sum_{i=0}^{n-1}\w_{x_i}:x_{0,n}\text{ up-right, $x_0=x$, $x_n=y$, $n=\abs{y-x}_1$}\Big\},
	\end{align}
i.e.\ we now leave out the last weight $\w_y$ instead.  We will use this definition in the rest of the chapter. The reader should note, though, that since we are interested in the large scale behavior of the system,
it is really immaterial whether we leave out or include the first or last weights.

Recall that $g$ is a concave function. Then the set
	\[\cC=\{x\in\R^2_+:g(x)\ge1\}\]
is convex. It can then be decomposed into a union of closed faces (see Section 17 of \cite{Roc-70}). For a given $\xi\in\riUset$ the point $\xi/g(\xi)$ is on the relative boundary of $\cC$. Let 
	\[\Uset_\xi=\Big\{\frac{\zeta}{\abs{\zeta}_1}:\zeta\text{ belongs to the closed face containing $\xi/g(\xi)$}\Big\}.\]
If $g$ is differentiable, then $\Uset_\xi$ is simply the largest connected subset of $\riUset$ containing $\xi$ on which $g$ is affine. 
Set $\Uset_\xi$ cannot contain $e_1$ or $e_2$ because we chose
$\xi$ in the relative interior of $\Uset$ and Theorem \ref{th:Martin} prevents $g$ from being linear on a neighborhood of $e_1$ or $e_2$.

Let $\Uset_{e_1}=\{e_1\}$ and $\Uset_{e_2}=\{e_2\}$. Clearly $\Uset=\cup_{\xi\in\Uset}\Uset_\xi$. When $g$ is differentiable we have that $\forall\xi,\zeta\in\Uset$, either $\Uset_\xi=\Uset_\zeta$ or $\Uset_\xi\cap\Uset_\zeta=\varnothing$.

We will say that a sequence $x_n$ is asymptotically directed into a subset $A\subset\Uset$
if all the limit points of $x_n/n$ are inside $A$.

We are ready to state the theorem about the limits of the gradients in \eqref{(*)}. 
The proof requires a detour into queuing theory and is thus deferred to Section \ref{fixed-pt}.

\begin{theorem}\label{Bus:thm}
{\rm\cite{Geo-Ras-Sep-17-ptrf-1}}
Assume $\P\{\w_0\ge c\}=1$ for some $c\in\R$ and $\E[\abs{\w_0}^{2+\e}]<\infty$ for some $\e>0$. Assume $g$ is differentiable on $(0,\infty)^2$. Then, for each $\xi\in\riUset$ the {\rm(}random{\rm)} limit
	\begin{align}\label{B-limit}
	B^\xi(\w,x,y)=\lim_{n\to\infty}(G_{x,x_n}-G_{y,x_n})
	\end{align}
exists almost surely and in $L^1$ for all $x,y\in\Z^2$ and sequence $x_n\in\Z^2$ that is directed into $\Uset_\xi$. Furthermore,
	\begin{align}\label{B-g}
	\begin{split}
	&\text{for }i\in\{1,2\}\quad\E[B^\xi(0,e_i)]=e_i\cdot \nabla g(\xi)\\
	&\text{and}\quad g(\xi)=\E[B^\xi(0,e_1)]\,\xi\cdot e_1+\E[B^\xi(0,e_2)]\,\xi\cdot e_2.
	\end{split}
	\end{align}
If $\xi,\zeta\in\riUset$ are such that $\Uset_\xi=\Uset_\zeta$, then $B^\xi\equiv B^\zeta$ almost surely.
\end{theorem}

(It is customary in probability theory to drop the dependence on $\w$ from the notation of random variables, e.g.\ to write $B^\xi(x,y)$ instead of $B^\xi(\w,x,y)$.)

The distribution of $B^\xi(0,e_i)$, $i\in\{1,2\}$, is known in the solvable cases. Formula \eqref{shape:exp} then follows from the second equation in \eqref{B-g}. See the end of Section \ref{fixed-pt} for more details.

The condition that weights are bounded below by a deterministic constant is not really necessary. We assumed it in \cite{Geo-Ras-Sep-17-ptrf-1} because we used queuing theory for the proof, as we will see in Section \ref{fixed-pt},
and then weights $\w_x$ are service times and have to be nonnegative. The extension to weights that are bounded below is immediate. However, the math in the proof works just as well for general weights, even though then interpreting them as
service times would not make sense.

The differentiability assumption is more serious. Although still an open question, differentiability is believed to be generally true. It can be directly verified from the explicit formula, when the weights are either exponentially or geometrically distributed. 
In the case when a flat segment  occurs (see Theorem \ref{flat:thm}), it is known that $g$ is differentiable at the two edges of the segment. 
See \cite{Auf-Dam-13} for the standard first-passage percolation and  \cite{Geo-Ras-Sep-17-ptrf-1} for the directed LPP.
It is worthy to note that when $\{\w_x\}$ are only ergodic (a generalization of being i.i.d.), the limiting shape can  have corners and linear segments, and can even be a polygon with finitely many edges. See \cite{Hag-Mee-95}.\medskip

Limits $B^\xi$ are called  
\index{Busemann function}%
{\sl Busemann functions}. This name is borrowed from metric geometry due to a connection between Busemann functions and geodesics, which is revealed in Section \ref{geodesics}.

The first equation in \eqref{B-g} can be understood as follows: $B^\xi$ is a (microscopic) gradient of passage times to ``far away points in direction $\xi$''. The shape function at $\xi$ is the 
large scale (macroscopic) limit of passage times to these far away points. \eqref{B-g}  says that the mean of the microscopic gradient is exactly the macroscopic gradient. 

The second equation in \eqref{B-g} is simply a consequence of the first one since differentiating $g(t\xi)=tg(\xi)$ in $t$ gives
	\begin{align}\label{euler}
	g(\xi)=\xi\cdot \nabla g(\xi).
	\end{align}
	
Let us record right away a few important properties of processes $B^\xi$.\medskip 

\index{cocycle}%
{\bf 1. Cocycle:}
	\begin{align}\label{cocycle}
	B^\xi(x,y)+B^\xi(y,z)=B^\xi(x,y),\quad\text{almost surely and for all }x,y,z\in\Z^2.
	\end{align}
It follows  immediately from \eqref{B-limit}.\medskip 

\index{recovery}%
{\bf 2. Recovery:}
	\begin{align}\label{recovery}
	\w_x=\min(B^\xi(x,x+e_1),B^\xi(x,x+e_2))\quad\text{almost surely and for all }x\in\Z^2.
	\end{align}
To see this holds start with an induction equation similar to \eqref{G:ind} (but recall that passage times are now defined by \eqref{G:def2}):
	\[G_{x,y}=\w_x+\max(G_{x+e_1,y},G_{x+e_2,y})\quad\text{for all $x$ and $y\ge x+e_i$, $i\in\{1,2\}$}.\]
Rewrite this as
	\begin{align}\label{G:ind2}
	\w_x=\min(G_{x,y}-G_{x+e_1,y},G_{x,y}-G_{x+e_2,y}).
	\end{align}
Now set $y=\fl{n\xi}$ and take $n\to\infty$.\medskip

{\bf 3. Stationarity:} To express this, define the group action of $\Z^2$ on $\Omega=\R^{\Z^2}$ that consists of shifting the weights: for $z\in\Z^2$ and $\w\in\Omega$, $T_z\w\in\Omega$ is such that $(T_z\w)_x=\w_{x+z}$. In words, $T_z\w$ is simply the weight configuration obtained from $\w$ by placing the origin at $z$.
Then we have
	\begin{align}\label{shift}
	B^\xi(T_z\w,x,y)=B^\xi(\w,x+z,y+z)\quad\text{almost surely and for all }x,y,z\in\Z^2.
	\end{align}
This follows directly from \eqref{B-limit} and the fact that $G_{x,y}(T_z\w)=G_{x+z,y+z}(\w)$. (All these equations are just saying is that fixing the lattice and shifting the weight configuration is the same thing as fixing the weight configuration and shifting the lattice.)
\medskip

{\bf 4. Monotonicity:} We record the last property as a lemma.

\begin{lemma}\label{lm:monotone}
{\rm\cite{Geo-Ras-Sep-17-ptrf-1}}
Make the same assumptions as in Theorem \ref{Bus:thm}. Fix $\xi,\zeta\in\riUset$ with $\xi_1<\zeta_1$. Then we have almost surely and for all $x\in\Z^2$
	\begin{align}\label{monotone}
	B^\xi(x,x+e_1)\ge B^\zeta(x,x+e_1)\quad\text{and}\quad B^\xi(x,x+e_2)\le B^\zeta(x,x+e_2).
	\end{align}
\end{lemma}

\begin{proof}
The claim follows from a monotonicity of the passage times $G_{x,y}$ themselves: for all $x\in\Z^2$ and $y,z\in\Z^2_+$ such that $x+e_1+e_2\le y$, $x+e_1+e_2\le z$, $\abs{y}_1=\abs{z}_1$, and $y\cdot e_1<z\cdot e_1$ we have
	\begin{align}\label{crossing}
	G_{x,y}-G_{x+e_1,y}\ge G_{x,z}-G_{x+e_1,z}\quad\text{and}\quad G_{x,y}-G_{x+e_2,y}\le G_{x,z}-G_{x+e_2,z}.
	\end{align}
Indeed, once this is proved set $y=y_n$ and $z=z_n$ with $y_n,z_n\in\Z_+^2$ any two sequences with $\abs{y_n}_1=\abs{z_n}_1=n$, $y_n/n\to\xi$, and $z_n/n\to\zeta$, then send $n\to\infty$. 

Inequalities \eqref{crossing} are due to paths crossing. Indeed, a geodesic from $x$ to $z$ must cross a geodesic from $x+e_1$ to $y$. Let $u$ be the first point where the two paths cross, i.e.\ $u\cdot(e_1+e_2)$ is the smallest possible.
See Figure \ref{monotone:fig}. 

\begin{figure}
	\begin{center}
		\begin{tikzpicture}[ >=latex, scale=0.8]

			\draw[line width=1pt, sussexp](1.08,0)--(1.08,0.91)--(6,0.91)--(6,1.91)--(7,1.91)--(7,3);
			\draw[line width=1pt, sussexp](0.09,0)--(0.09,1.91)--(5,1.91)--(5,3)--(8,3);

			\draw[line width=1pt, sussexg](1,0)--(1,1)--(3,1)--(3,2)--(4,2)--(4,4)--(6,4)--(6,5);
			\draw[line width=1pt, sussexg](0,0)--(0,2)--(2,2)--(2,3)--(4,3);
			
			\draw[fill =black](0.04,0)circle(1.2mm);
			\draw(0.04,-0.4)node{$x$};
			\draw[fill =black](1.04,0)circle(1.2mm);
			\draw(1.4,-0.4)node{$x+e_1$};
			\draw[fill =black](8,3)circle(1.2mm);
			\draw(8.3,3.3)node{$z$};
			\draw[fill =black](6,5)circle(1.2mm);
			\draw(6.35,5.35)node{$y$};
			
			\draw[fill =red](3,1.95)circle(1.2mm);
			\draw(3,2.35)node{$u$};

		\end{tikzpicture}
	\end{center}
	\caption{\small  The paths crossing trick: a path from $x$ to $z$ must cross a path from $x+e_1$ to $y$.}
	\label{monotone:fig}
\end{figure}

Then superadditivity implies that
	\[G_{x,u}+G_{u,y} \le G_{x,y}\quad\text{and}\quad G_{x+e_1,u}+G_{u,z} \le G_{x+e_1,z}.\]
Add the two inequalities and rearrange to get
	\[G_{x,y}-G_{x+e_1,u}-G_{u,y}  \ge G_{x,u}+G_{u,z}-G_{x+e_1,z}.\]
Use the fact that $u$ is on geodesics to add the passage times and get the desired inequality:
	\[G_{x,y}-G_{x+e_1,y}=G_{x,y}-G_{x+e_1,u}-G_{u,y}  \ge G_{x,u}+G_{u,z}-G_{x+e_1,z}=G_{x,z} - G_{x+e_1,z}.\]
A similar proof works for the $e_2$-gradients and as we showed above, the claim of the theorem follows.
This ``crossing trick'' has been used profitably in planar percolation, and goes back at least to \cite{Alm-98, Alm-Wie-99}.
\end{proof}

\section{Queuing fixed points}\label{fixed-pt}

We now sketch how Theorem \ref{Bus:thm} is proved. By adding a constant to the weights, if necessary, we can assume they are nonnegative. Then we can use the queuing terminology. Let us note one small modification, though.
Due to our new definition \eqref{G:def2}, which replaced \eqref{G:def}, now $G_{0,(k,\ell)}$ is the time
when customer $k$ enters service at station $\ell$ and $G_{0,(k,\ell)}+\w_{k,\ell}$, is the time when
customer $k$ departs station $\ell$ and joins the end of the queue at station $\ell+1$.

We are looking for stationary versions of the queuing process described in Section \ref{queues}. For this, customers need to have been arriving for a long time. Thus, indexing them by $\N$ will not do and we need
instead to have customers indexed by $\Z$. Furthermore, one cannot track the customer's arrival times at the queues (since the process started in the far past) and the right thing to do is to consider 
\index{inter-arrival}%
{\sl inter-arrival} times.

Thus, the development begins with a processes $\{A_{n,0}:n\in\Z\}$ that records  the time between the arrival of customers number $n$ and $n+1$ to queue $0$. This process is supposed to be stationary,
i.e.\ the distribution of $\{A_{n+m,0}:n\in\Z\}$ does not depend on $m\in\Z$. We also assume it is 
\index{ergodic}%
{\sl ergodic}. (Stationary measures form a convex set whose extreme points are said to be ergodic. A process is ergodic 
if its distribution is ergodic.) 
One special case
is a constant inter-arrival process: $A_{n,0}=\alpha$ for all $n\in\Z$ and some $\alpha\in\R$.

We are also given service times $\{S_{n,k}:n\in\Z,k\in\Z_+\}$. They represent the time it takes to serve customer $n$ at station $k$, once the customer is first in line at that station. 
These service times have the same joint distribution $\P$ as the weights in our directed LPP model. In particular, they are i.i.d. 
Service times are also independent of the inter-arrival process.
In order for the system to be stable, we need to have customers served faster than they arrive.
Hence, we require that 
	\begin{align}\label{m<A}
	m_0=\E[S_{0,0}]<\E[A_{0,0}]<\infty.
	\end{align}

Given the inter-arrival and service times define waiting times at station $0$ by
	\begin{align}\label{W:def}
	W_{n,0}=\Big(\sup_{j\le n-1}\sum_{i=j}^{n-1}(S_{i,0}-A_{i,0})\Big)^+.
	\end{align}
$W_{n,0}$ is the time customer $n$ waits at station $0$ before their service starts. See further down for an explanation.

Process $\{S_{i,0}-A_{i,0}:i\in\Z\}$ is ergodic. This is because one can show that the product measure of an ergodic process with an i.i.d.\ one is ergodic.
By the ergodic theorem the sample mean $(n-j)^{-1}\sum_{i=j}^{n-1}(S_{i,0}-A_{i,0})$ converges to the population mean
$\E[S_{0,0}-A_{0,0}]$, which by \eqref{m<A} is negative. Therefore, almost surely, as $j\to-\infty$ the sum $\sum_{i=j}^{n-1}(S_{i,0}-A_{i,0})$ goes to $-\infty$ and $0\le W_{n,0}<\infty$ for all $n\in\Z$. It is immediate to check that these times satisfy Lindley's equation
	\begin{align}\label{ind1}
	W_{n+1,0}=(W_{n,0}+S_{n,0}-A_{n,0})^+.
	\end{align}
This now explains why $W_{n,0}$ is the time customer $n$ waits at station $0$ before their service starts. Indeed, if $W_{n,0}+S_{n,0}<A_{n,0}$ then customer $n$ will leave station $0$ before the next customer $n+1$ arrives.
As a result, customer $n+1$ does not wait and $W_{n+1,0}=0$. If, on the other hand, $W_{n,0}+S_{n,0}\ge A_{n,0}$, then customer $n+1$ waits time $W_{n+1,0}=W_{n,0}+S_{n,0}-A_{n,0}$ before  service begins.

Inter-departure times from queue $0$ or, equivalently, inter-arrival times at queue $1$ are given by
	\begin{align}\label{ind2}
	A_{n,1}=(W_{n,0}+S_{n,0}-A_{n,0})^-+S_{n+1,0}.
	\end{align}
Again, if $W_{n+1,0}>0$ then customer $n+1$ is already waiting and will start being serviced as soon as customer $n$ departs.
The time between the departure of customer $n$ and that of customer $n+1$, from station $0$, is then equal to $S_{n+1,0}$.
In the other case, when $W_{n+1,0}=0$, station $0$ is empty before customer $n+1$ gets there, for exactly time $A_{n,0}-S_{n,0}-W_{n,0}$. The time between the departures of customers
$n$ and $n+1$ from queue $0$ equals this idle time, plus the service time $S_{n+1,0}$.

From \eqref{ind1} and \eqref{ind2} we have 
	\[W_{n+1,0}+S_{n+1,0}+A_{n,0}=W_{n,0}+S_{n,0}+A_{n,1}\]
for all $n\in\Z$. This in turn gives
	\[\sum_{m=0}^{n-1} A_{m,0}+S_{n,0}-S_{0,0}+W_{n,0}-W_{0,0}=\sum_{m=0}^{n-1}A_{m,1}.\]
Process $\{A_{n,1}:n\in\Z\}$ is again ergodic. Hence, dividing by $n$ and taking it to $\infty$ the ergodic theorem tells us that $n^{-1}\sum_{m=0}^{n-1} A_{m,0}$ and $n^{-1}\sum_{m=0}^{n-1}A_{m,1}$ converge to $\E[A_{0,0}]$ and $\E[A_{0,1}]$, respectively.
Since $S_{n,0}$ is a stationary process, $S_{n,0}/n$ converges to $0$.   The next lemma shows that also $W_{n,0}/n\to0$ almost surely. Consequently, we have that $\E[A_{0,1}]=\E[A_{0,0}]$.

\begin{lemma}
We have $n^{-1}W_{n,0}\to0$ almost surely, as $n\to\infty$.
\end{lemma}

\begin{proof}
Let $U_n=S_{n,0}-A_{n,0}$. Fix $\e>0$, $a\ge0$, and let $u_0=\E[U_n]<0$. ($u_0$ does not depend on $n$ because $U_n$ is stationary.)
Set $W^\e_0(a)=a$ and define inductively $W_{n+1}^\e(a)=(W_n^\e(a)+U_n-u_0+\e)^+$, $n\ge0$. By induction
	\[W_n^\e(0)=\Bigl(\max_{0\le m<n}\sum_{k=m}^{n-1}(U_k-u_0+\e)\Big)^+.\]
Also, since the induction preserves monotonicity we have that 
	\[W_n^\e(a)\ge W^\e_n(0)\ge \sum_{k=0}^{n-1}(U_k-u_0+\e).\]
Since $\E[U_k-u_0+\e]=\e>0$, the ergodic theorem tells us the last sum grows to infinity (linearly in $n$). Consequently, there exists an $n_0$ such that $W_n^\e(a)>0$ for $n\ge n_0$.
But then if $n\ge n_0$ we have $W_n^\e(a)=W_{n-1}^\e(a)+U_{n-1}-u_0+\e$ and thus
	\[n^{-1}W_n^\e(a)=n^{-1}W_{n_0}^\e(a)+n^{-1}\sum_{k=n_0}^{n-1}(U_k-u_0+\e).\]
The ergodic theorem again tells us that the last term	converges to $\e$ as $n\to\infty$. As a result, we have shown that  with probability one for any $a\ge0$ we have $n^{-1}W_n^\e(a)\to0$.

Recall now that $W_{0,0}<\infty$ and observe that $W_0^\w(W_{0,0})=W_{0,0}$. Induction then shows that $W_{n,0}\le  W_n^\e(W_{0,0})$ for all $n\ge0$. Indeed,
	\begin{align*}
	W_{n+1,0}&=(W_{n,0}+U_n)^+\le(W_n^\e(W_{0,0})+U_n)^+\\
	&\le(W_n^\e(W_{0,0})+U_n-u_0+\e)^+=W_{n+1}^\e(W_{0,0}).
	\end{align*}
But then 
	\[0\le\varliminf_{n\to\infty}n^{-1}W_{n,0}\le\varlimsup_{n\to\infty}n^{-1}W_{n,0}\le\lim_{n\to\infty}n^{-1}W_n^\e(W_{0,0})=\e.\]
Taking $\e\to0$ completes the proof.
\end{proof}

Note that $A_{n,1}$ only used values $\{S_{m,0}:m\in\Z\}$ and is therefore independent of $\{S_{m,k}:m\in\Z,k\ge1\}$.
We can now repeat the above steps inductively: Say we already computed the (ergodic) inter-arrival process $\{A_{n,k}:n\in\Z\}$ at queue $k\ge0$ and that it is independent of the service times $\{S_{n,\ell}:n\in\Z,\,\ell\ge k\}$.
Say also that $\E[A_{0,k}]=\E[A_{0,0}]$. Then we define the waiting times
	\begin{align}\label{W:def_k}
	W_{n,k}=\Big(\sup_{j\le n-1}\sum_{i=j}^{n-1}(S_{i,k}-A_{i,k})\Big)^+,
	\end{align}
which satisfy 
	\begin{align}\label{lindley}
	W_{n+1,k}=(W_{n,k}+S_{n,k}-A_{n,k})^+.
	\end{align}
The inter-arrival process at queue $k+1$ is given by
	\begin{align}\label{A-ind}
	A_{n,k+1}=(W_{n,k}+S_{n,k}-A_{n,k})^-+S_{n+1,k}.
	\end{align}
It is ergodic and has the same mean as $A_{n,k}$ (and thus as $A_{n,0}$). It is also independent of $\{S_{n,\ell}:n\in\Z,\,\ell\ge k+1\}$ and we can continue the inductive process.

Using \eqref{lindley} and \eqref{A-ind} we can check the conservation law
	\begin{align}\label{q-coc}
	W_{n+1,k}+S_{n+1,k}+A_{n,k}=W_{n,k}+S_{n,k}+A_{n,k+1}
	\end{align}
and the equation
	\begin{align}\label{q-rec}
	S_{n+1,k}=\min(W_{n+1,k}+S_{n+1,k},A_{n,k+1}).
	\end{align}

Times $W_{n,k}+S_{n,k}$ are called 
\index{work load}%
{\sl work load} of station $k$ by customer $n$. Let us assign values to the edges of $\Z\times\Z_+$: on the horizontal edge $(n,k)-(n+1,k)$ put weight $A_{n,k}$ and on vertical 
edge $(n,k)-(n,k+1)$ put weight $W_{n,k}+S_{n,k}$. Also, put weights $S_{n,k}$ on vertices $(n,k+1)$. 
Recall now \eqref{cocycle} and \eqref{recovery}. Then \eqref{q-coc} can be seen as a 
\index{cocycle}%
cocycle property and \eqref{q-rec} is a 
\index{recovery}%
recovery property (but in the south-west direction 
instead of the north-east direction). See Figure \ref{q-coc:fig}. This is where the connection to Busemann functions lies.

\begin{figure}
	\begin{center}
		\begin{tikzpicture}[ >=latex, scale=1]

			\draw[line width=1pt](0,0)--(4,0)--(4,4)--(0,4)--(0,0);
			
			\draw[fill =black](0,0)circle(1.5mm);
			\draw[fill =black](4,0)circle(1.5mm);
			\draw[fill =black](0,4)circle(1.5mm);
			\draw[fill =black](4,4)circle(1.5mm);
			
			\draw(4.5,4.5) node{$S_{n+1,k}$};   
			\draw(4.7,-0.5) node{$S_{n+1,k-1}$};  
			\draw(-0.3,4.5) node{$S_{n,k}$};  
			\draw(0,-0.5) node{$S_{n,k-1}$}; 

			\draw(2,4.5) node{$A_{n,k+1}$};
			\draw(2,-0.5) node{$A_{n,k}$};
			\draw(5.6,2) node{$W_{n+1,k}+S_{n+1,k}$};
			\draw(-1.3,2) node{$W_{n,k}+S_{n,k}$};

		\end{tikzpicture}
	\end{center}
	\caption{\small  Assignment of work loads, inter-arrival times, and service times, to  edges and vertices.}
	\label{q-coc:fig}
\end{figure}
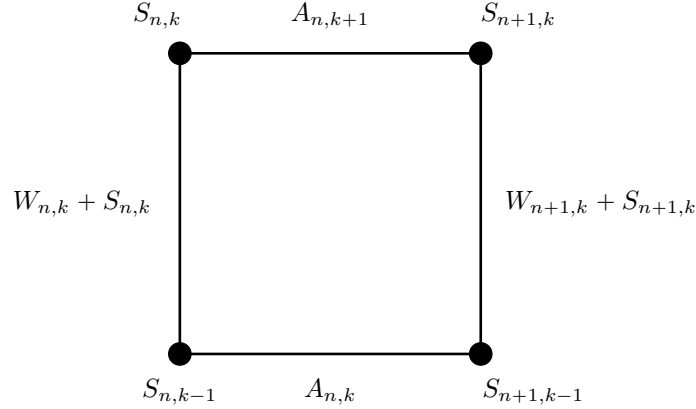

Observe that  system $\{A_{n,k},W_{n,k},S_{n,k}:n\in\Z,\,k\in\Z_+\}$ is invariant with respect to shifts in the first coordinate, i.e.\ 
the distribution of $\{A_{n+m,k},W_{n+m,k},S_{n+m,k}:n\in\Z,\,k\in\Z_+\}$ is the same for all $m\in\Z$. It is in fact also ergodic under shifts in the first coordinate.
However, it is not obvious at all (and in fact not true in general) that the system is  invariant (let alone ergodic) under shifts in the second coordinate. That is, we do not know a priori that
the distribution of $\{A_{n,k+\ell},W_{n,k+\ell},S_{n,k+\ell}:n\in\Z,\,k\in\Z_+\}$ is independent of $\ell\in\Z_+$.  For this to happen, clearly $\{A_{n,0}:n\in\Z\}$ needs to have some special distribution.

If we denote the distribution of $\{A_{n,0}:n\in\Z\}$ by $\nu$ then write $\Phi(\nu)$ for the distribution of $\{A_{n,1}:n\in\Z\}$.
This is the so-called  
\index{queuing operator}%
{\sl queuing operator}. It takes a ergodic probability measure on $\R^\Z$ and transforms it to another ergodic probability measure on $\R^\Z$, while preserving the value of the mean (recall that $\E[A_{n,1}]=\E[A_{n,0}]$).
Let $\Phi^k$ be the $k$-th iterate of $\Phi$, i.e.\ $\Phi^1=\Phi$ and $\Phi^{k+1}(\nu)=\Phi(\Phi^k(\nu))$.

For the  invariance under shifts in the second coordinate to hold, what  we need is  $\Phi(\nu)=\nu$, i.e.\ that $\nu$ be an ergodic 
\index{fixed point}%
{\sl fixed point} of the queuing operator.

Thus, the problem at hand is: Given $\alpha>m_0$ find ergodic measures $\nu_\alpha$ on $\R^\Z$ such that $\Phi(\nu_\alpha)=\nu_\alpha$. And it would be good if along the way we can also answer the question of uniqueness of such measures.

One way to produce fixed points with a prescribed mean $\alpha>m_0$ is to start for example with the measure $\delta_\alpha^\Z$, the distribution of the constant process $\{A_{n,0}=\alpha:n\in\Z\}$, and  
hope that $\Phi^k(\delta_\alpha^\Z)$ converges (weakly) to an ergodic fixed point that has mean $\alpha$, as $k\to\infty$. 

Mairesse and Prabhakar \cite{Mai-Pra-03} proved that for each $\alpha>m_0$ there exists a unique stationary fixed point $\nu_\alpha$ such that if one starts with any ergodic process $\nu$ with mean $\alpha$, 
then the C\'esaro mean $k^{-1}\sum_{\ell=1}^k\Phi^k(\nu)$ converges weakly to $\nu_\alpha$.
One consequence of this is that ergodic fixed points with a prescribed mean are unique.
It is an open question whether or not for each $\alpha>m_0$ probability measure $\nu_\alpha$ is ergodic. What is known, though, is that this is true if we have differentiability of the shape function $g$ for the LPP problem with weights distribution $\P$
(the distribution of the service times) \cite[Lemma 7.6(b)]{Geo-Ras-Sep-17-ptrf-1}.

For a fixed $\alpha>m_0$ let $\{A_{n,0}:n\in\Z\}$ have distribution $\nu_\alpha$ and let $\{S_{n,k}:n\in\Z,k\in\Z_+\}$ be i.i.d.\ with distribution $\P$, independent of $\{A_{n,0}\}$.
Construct the processes $\{W_{n,k},A_{n,k}:n\in\Z,k\in\Z_+\}$ by inductions \eqref{lindley} and \eqref{A-ind}. Then, because $\nu_\alpha$ is a fixed point, process $\{W_{0,k}:k\in\Z_+\}$ is stationary and can thus be extended
to a stationary process $\{W_{0,k}:k\in\Z\}$. (If this is not clear to the reader, it is an excellent exercise on a standard application of Kolmogorov's extension theorem.) Let
	\begin{align}\label{def:f}
	f(\alpha)=\E[W_{0,0}]+\E[S_{0,0}].
	\end{align}
Using the symmetry of the construction one can prove that the distribution of $\{W_{0,k}+S_{0,k}:k\in\Z\}$ is again a fixed point of the queuing operator and that it is ergodic if $\nu_\alpha$ is ergodic \cite[Lemma 7.7]{Geo-Ras-Sep-17-ptrf-1}.
In this case, this distribution is no other than $\nu_{f(\alpha)}$.  This tells us that not only the system $\{A_{n,k},W_{n,k},S_{n,k}:n\in\Z,k\in\Z_+\}$ is ergodic under shifting the $n$ coordinate, but it is also ergodic under shifts of the $k$ coordinate.

\begin{theorem}\label{f(alpha)}
Function $f$ takes values in $(m_0,\infty)$ and is a convex, continuous, and strictly decreasing involution {\rm(}i.e.\ $f(f(\alpha))=\alpha${\rm)}. Furthermore, $f(\alpha)$ converges to $\infty$ as $\alpha\searrow m_0$ at to $m_0$ as $\alpha\to\infty$.

Let $g$ be the shape function for the LPP problem with weight distribution $\P$ {\rm(}the distribution of the service times{\rm)}. Assume it is differentiable on $\riUset$.
Then for all $\xi\in\riUset$ we have
	\begin{align}\label{g-var}
	g(\xi)=\inf_{\alpha>m_0}\bigl(\alpha\xi_1+f(\alpha)\xi_2\bigr).
	\end{align}
The infimum is minimized at $\alpha=e_1\cdot\nabla g(\xi)$ and then $f(\alpha)=e_2\cdot\nabla g(\xi)$.
\end{theorem}

\begin{proof}
Let $G_{me_1,ne_2}$ denote the last passage time from $me_1$ to $ne_2$ using weights $\{S_{m,k}:m\in\Z,k\in\Z_+\}$. Equation \eqref{W:def} for $n=0$ can be rewritten as
	\[S_{0,0}+W_{0,0}=\sup_{j\le0}\Bigl(G_{je_1,e_2}-\sum_{i=j}^{-1}A_{i,0}\Big),\]
with the convention that an empty sum is zero.
We can extend this by induction to
	\[\sum_{k=0}^{n-1}(S_{0,k}+W_{0,k})=\sup_{j\le0}\Bigl(G_{je_1,ne_2}-\sum_{i=j}^{-1}A_{i,0}\Big).\]
Divide by $n$, take it to infinity, and use the ergodic theorem and the shape theorem (plus a bit of work similar to what we did in the proof of continuity of $g$) we get
	\[f(\alpha)=\sup_{s\ge0}\bigl(g(s,1)-s\alpha\bigr)\]
with a maximizer at $s$ such that $\alpha=e_1\cdot\nabla g(s,1)$.

The above formula shows that $f$ is nonincreasing and convex function. 
Martin's asymptotic formula \eqref{martin} (and homogeneity of $g$) implies that $m_0<f(\alpha)<\infty$ for $\alpha>m_0$. 
In particular, the convex function $f$ is continuous. The symmetry of $g$ implies that $f$ is an involution and is thus strictly decreasing. The limits claimed in the theorem follow. 

Inverting the above convex duality we get 
	\begin{align}\label{g-f}
	g(s,1)=\inf_{\alpha>m_0}\bigl(\alpha s+f(\alpha)\bigr)
	\end{align}
with a minimizer at $\alpha=e_1\cdot\nabla g(s,1)$. Variational formula \eqref{g-var} now comes by using homogeneity of $g$ to write $g(\xi)=\xi_2 g(\xi_1/\xi_2,1)$.
The minimizing $\alpha$ is given by $e_1\cdot g(\xi)$, as claimed, and then the fact that $f(\alpha)=e_2\cdot\nabla g(\xi)$ comes from \eqref{euler}.	
\end{proof}

The above construction of fixed points can be carried out simultaneously for any given countable set of parameters $\alpha>m_0$, thus coupling the fixed points $\nu_\alpha$.  A bit more precisely, fix a countable set $\cA_0\subset(m_0,\infty)$ 
and start with inter-arrival times $A^{(\alpha)}_{n,0}=\alpha$, $\alpha\in\cA_0$, $n\in\Z$, and service times $\{S_{n,k},n\in\Z,\,k\in\Z_+\}$ that are independent and have the same distribution as the weights of the LPP model.
(Note that the service times do not depend on $\alpha$.)
Use \eqref{W:def_k} and \eqref{A-ind} to define inductively $W^{(\alpha)}_{n,k}$ and $A^{(\alpha)}_{n,k}$ on all of $\Z\times\Z_+$.
Then, as $k\to\infty$ the C\'esaro mean of the distributions of $\{A^{(\alpha)}_{n,k}:n\in\Z,\,\alpha\in\cA_0\}$ converges weakly to a probability measure on $(\R^\Z)^{\cA_0}$
whose marginal for a fixed $\alpha\in\cA_0$ is exactly $\nu_\alpha$. \medskip

We can now go back to proving  existence of the limit \eqref{B-limit}. 

\begin{proof}[Proof of Theorem \ref{Bus:thm}]
The first step is to extract from the above queuing objects candidates for the limits $B^\xi$. Then we prove that the Busemann limits indeed exist and equal these candidates.

Recall that we assume $g$ is differentiable. 
For $\zeta\in\riUset$ define $\alpha(\zeta)=e_1\cdot\nabla g(\zeta)$. Then $\alpha$ is nonincreasing and continuous in $\zeta_1=\zeta\cdot e_1$. Also, it follows from \eqref{martin} that $\alpha(\zeta)\to m_0$ as $\zeta\to e_1$ and $\alpha(\zeta)\to\infty$ as $\zeta\to e_2$.
In particular,  $\alpha(\zeta)$ is always strictly bigger than $m_0=\E[\w_0]$.
Fix $\xi\in\riUset$ and let $\Uset_0=\{\xi\}\cup(\riUset\cap\Q^2)$. Let $\cA_0=\{\alpha(\zeta):\zeta\in\Uset_0\}$. This is a dense countable subset of $(m_0,\infty)$.

Let $\{A^{(\alpha)}_{n,0}:n\in\Z,\,\alpha\in\cA_0\}$ be distributed according to the above coupling of fixed points
and let $\{S_{n,k}:n\in\Z,\,k\in\Z_+\}$ be independent, with the same joint distribution $\P$ as the weights in the directed LPP model, and independent of the inter-arrival times.
Construct times $\{A_{n,k}^{(\alpha)},W_{n,k}^{(\alpha)}:n\in\Z,\,k\in\N,\,\alpha\in\cA_0\}$ using inductions \eqref{W:def_k} and \eqref{A-ind}.
Since it comes from fixed points, system $\{A^{(\alpha)}_{n,k},W^{(\alpha)}_{n,k},S_{n,k}:n\in\Z,\,k\in\Z_+,\alpha\in\cA_0\}$ is stationary under shifts in both coordinates. We can then extend it to 
a system $\{A^{(\alpha)}_{n,k},W^{(\alpha)}_{n,k},S_{n,k}:n\in\Z,\,k\in\Z,\,\alpha\in\cA_0\}$ on the whole lattice. 

For $\zeta\in\Uset_0$ define
	\begin{align*}
	&\w_{ne_1+ke_2}=S_{-n,-k-1},\quad B^\zeta(ne_1+ke_2,(n+1)e_1+ke_2)=A^{(\alpha(\zeta))}_{-n-1,-k},\quad \text{and}\\
	&B^\zeta(ne_1+ke_2,ne_1+(k+1)e_2)=W^{(\alpha(\zeta))}_{-n,-k-1}+S_{-n,-k-1}\,.
	\end{align*}
Then $\{\w_x:x\in\Z^2\}$ has distribution $\P$ and \eqref{q-rec} says that $B^\zeta$ satisfies the recovery property \eqref{recovery}. 
Equation \eqref{q-coc} says that $B^\zeta$ satisfies
	\begin{align}\label{B-cell}
	\begin{split}
	&B^\zeta(x,x+e_1)+B^\zeta(x+e_1,x+e_1+e_2)\\
	&\qquad\qquad\qquad=B^\zeta(x,x+e_2)+B^\zeta(x+e_2,x+e_1+e_2)
	\end{split}
	\end{align}
for all $x\in\Z^2$. Set $B^\zeta(x+e_i,x)=-B^\zeta(x,x+e_i)$, $i\in\{1,2\}$ and for $x,y\in\Z^2$ define
	\[B^\zeta(x,y)=\sum_{i=0}^{n-1}B^\zeta(x_i,x_{i+1}),\]
where $x_{0,n}$ is any nearest-neighbor path from $x_0=x$ to $x_n=y$. Thanks to \eqref{B-cell} this definition does not depend on the choice of the path $x_{0,n}$. 
Now, $B^\zeta$ is an $L^1$  cocycle and $\E[B^\zeta(0,e_1)]=\E[A^{(\alpha(\zeta))}_{0,0}]=\alpha(\zeta)$. 
This equality says that $B^\zeta$ satisfies the first equation in \eqref{B-g}, for the $e_1$ direction. The version for the $e_2$ direction comes from Theorem \ref{f(alpha)}: 
	\[\E[B^\zeta(0,e_2)]=\E[W^{\alpha(\zeta)}_{0,-1}+S_{0,-1}]=f(\alpha(\zeta))=e_2\cdot\nabla g(\zeta).\]
As was mentioned earlier, the second equation in \eqref{B-g} is simply a consequence of the first one and \eqref{euler}.

Observe next that if we start with two sequences of inter-arrival times $A_{n,0} \le A'_{n,0}$ for all $n\in\Z$, then \eqref{W:def} gives $W_{n,0}\ge W'_{n,0}$ for all $n\in\Z$. Then from \eqref{A-ind}  we get that 
$A_{n,1}\le A'_{n,1}$. This monotonicity of the queuing operator leads to a monotonicity in the coupling of the fixed points.
Combining this with the fact that if $\zeta,\eta\in\Uset_0$ are such that $\zeta_1<\eta_1$, then $\alpha(\zeta)= e_1\cdot\nabla g(\zeta)\ge e_1\cdot \nabla g(\eta)=\alpha(\eta)$, we get that the $B^\zeta$ cocycles 
we constructed satisfy monotonicity \eqref{monotone}.

The first equation in \eqref{B-g} and differentiability of $g$ imply that $\zeta\mapsto\E[B^\zeta(0,e_i)]$ is continuous, for $i\in\{1,2\}$.
Combined with the above monotonicity we get that with probability one 
	\begin{align}\label{B-cont}
	\lim_{\Uset_0\ni\zeta\to\xi} B^\zeta(x,x+e_i)=B^\xi(x,x+e_i),\quad i\in\{1,2\}.
	\end{align}
Note that the set of full $\P$-measure on which the above event holds depends on $\xi$. In fact, we will see in Corollary \ref{B-not-cont} that this continuity does not hold on all $\riUset$ simultaneously (i.e.\ with one null set thrown away).

Lastly,  by another monotonicity argument, not too different from the one we used in the proof of Lemma \ref{lm:monotone}, we can show that for $x\in\Z^2$, a sequence $x_n$ directed in $\Uset_\xi$,
directions $\zeta,\eta\in\Uset_0\setminus\Uset_\xi$ with $\zeta_1<\xi_1<\eta_1$, and a large integer $n$, we have the stochastic inequalities  
	\begin{align}\label{eq:comp}
	\begin{split}
	&B^\eta(x,x+e_1)\le G_{x,x_n}-G_{x+e_1,x_n}\le B^\zeta(x,x+e_1)\quad\text{and}\\
	&B^\eta(x,x+e_2)\ge G_{x,x_n}-G_{x+e_2,x_n}\ge B^\eta(x,x+e_2).
	\end{split}
	\end{align}
(A random variable $Y$ is said to be stochastically smaller than a random variable $Z$ if for all $a\in\R$ we have $P(Z\le a)\le P(Y\le a)$. If moreover one shows that $E[Z]=E[Y]$, then 
$Y$ and $Z$ have the same distribution.) Taking $n\to\infty$ then $\eta$ and $\zeta$ to $\xi$ and applying \eqref{B-cont} we get 
	\[B^\xi(x,x+e_i)=\lim_{n\to\infty}(G_{x,x_n}-G_{x+e_i,x_n}).\]
Then the cocycle property \eqref{B-cell} gives \eqref{B-limit}.

If $\Uset_\xi=\Uset_\zeta$ for some $\xi,\zeta\in\riUset$, then $\alpha(\xi)=e_1\cdot\nabla g(\xi)=e_1\cdot\nabla g(\zeta)=\alpha(\zeta)$ and thus $B^\xi=B^\zeta$.
The proof of Theorem \ref{Bus:thm} is complete.
\end{proof}

The next lemma develops the very first inequality in \eqref{eq:comp}, the others being similar. We focus on the case $x=0$, the general case coming from shift-invariance of $\P$.

\begin{lemma}
Fix $\xi\in\riUset$ and a {\rm(}possibly random{\rm)} sequence $x_n$ directed in $\Uset_\xi$.
Fix $\eta\in\riUset\setminus\Uset_\xi$ with $\xi_1<\eta_1$. Then almost surely for large $n$
	\[B^\eta(0,e_1)\le G_{0,x_n}-G_{e_1,x_n}.\]
\end{lemma}

\begin{proof}
Consider the two rectangle with common south-west corner at $0$ and with north-east corners at $x_n$ and $x_n+e_1+e_2$. 
Put weights $\w_x$ at all vertices $x\le x_n$. Let $\bar\w_x=\w_x$ for such vertices.
At sights $x=x_n+e_1+e_2-ke_1$, $k\in\N$, put weights $\bar\w_x=B^\eta(x,x+e_1)$. Similarly, at sights $x=x_n+e_1+e_2-ke_2$, $k\in\N$, put weights $\bar\w_x=B^\eta(x,x+e_2)$.
We will write $\Gbar_{y,x_n}$ for the passage time from $y$ to $x_n$ using weights $\w_x$, as defined in \eqref{G:def2}.
We also use the passage times from $y$ to $x_n+e_1+e_2$ that use weights the combination of weights $\w_x$, $x\le x_n$, and 
\end{proof}
	
To close this section let us describe the situation for solvable models. Here,  fixed points $\nu_\alpha$ can be described explicitly.
For example, in the case of exponentially distributed service times with mean $m_0>0$ (rate $1/m_0$) one can check directly that for any $\alpha>m_0$ inter-arrival times $\{A_{n,0}:n\in\Z\}$ 
that are i.i.d.\ exponentially distributed with mean $\alpha$ (rate $1/\alpha$) furnish an ergodic fixed point of the queuing operator. Because of the uniqueness of ergodic fixed points, this identifies $\nu_\alpha$.
Another direct computation verifies that $f(\alpha)=\E[W_{0,0}+S_{0,0}]=m_0\alpha/(\alpha-m_0)$ and the symmetry observed below \eqref{def:f} says that $\{W_{0,k}+S_{0,k}:k\in\Z\}$
are i.i.d.\ exponentially distributed with mean $f(\alpha)$. One more miracle occurs: It turns out that $\{A_{n,0}:n\in\Z_+\}$ and $\{W_{0,k}+S_{0,k}:k\in\Z_+\}$ are independent of each other. 
See Theorem 3.1 in \cite{ch:Seppalainen} for the proofs of all these distributional claims.
Once the explicit formula for $f$ is known solving the variational formula \eqref{g-f} leads to Rost's formula \eqref{shape:exp}. 

A consequence of the above is that  for $\xi\in\riUset$, $\{B^\xi(ne_1,(n+1)e_1):n\in\Z_+\}$  are independent exponentially distributed with rate $\frac{\sqrt{\xi_1}}{m_0(\sqrt{\xi_1}+\sqrt{\xi_2})}$,
$\{B^\xi(ne_2,(n+1)e_2):n\in\Z_+\}$ are independent exponentially distributed with rate $\frac{\sqrt{\xi_2}}{m_0(\sqrt{\xi_1}+\sqrt{\xi_2})}$, and the two sets of random variables are independent of each other.

Information  about the distribution of the Busemann functions  is powerful.   For example, it allows to get bounds on the coalescence time of geodesics \cite{Pim-16,Bas-Sar-Sly-17-}. We we will see in Section \ref{cif:sec} that 
it enables calculation of  the distribution of the asymptotic direction of the competition interface. It is also used in proving bounds on the fluctuations of passage times and geodesics, as is done
in Section 5 of \cite{ch:Seppalainen}.

\section{Geodesics}\label{geodesics}

In this section, let us assume  the conditions of Theorem \ref{Bus:thm} to be satisfied. In particular, the shape $g$ is differentiable on $(0,\infty)^2$.

One of the important questions in LPP concerns infinite geodesics: an infinite path is a geodesic if every finite segment of it is a geodesic between its endpoints.
The following existence result comes quite easily.

\begin{lemma}\label{exist-geo}
With probability one, for every $x\in\Z^2$ there is at least one infinite geodesics starting at that point. 
\end{lemma}

\begin{proof}
Fix $x\in\Z^2$.
Take any sequence $x_n\in\Z^2_+$ with $\abs{x_n}_1\to\infty$ and 
consider for each $n$ a geodesic from $x$ to $x_n$. Denote it by $x_{0,n}^{(n)}$. Fix $m\ge1$. We have only finitely many possible  
up-right paths of length $m$. Hence, there exists a subsequence along which $x_{0,n}^{(n)}$ all share the same initial $m$ steps.
Using the diagonal trick, we can find a subsequence $n_j$ such that for all $m\ge1$ there exists a $j_m$ such that , $\{x_{0,n_j}^{(n_j)}:j\ge j_m\}$ share the first $m$ steps.
This constructs an infinite path $x_{0,\infty}$ such that for all $m\ge1$, $x_{0,m}$ is the path shared by $\{x_{0,n_j}^{(n_j)}:j\ge j_m\}$. In particular, $x_{0,m}$ is a geodesic between
$x_0=x$ and $x_m$, for all $m\ge1$, and thus $x_{0,\infty}$ is an infinite geodesic starting at $x_0=x$.
\end{proof}

Now we know that infinite geodesics exist. But how many infinite geodesics starting at a given point are there? And what are their properties? Does every infinite geodesic $x_{0,\infty}$ 
have to have an asymptotic direction, i.e.\ is it necessary that $x_n/n\to\xi$ for some $\xi\in\Uset$? Can there be multiple geodesics that go in the same asymptotic direction $\xi\in\Uset$?
Do geodesics starting at different points and going in a given direction $\xi$ cross? Does there exist a bi-infinite geodesic, i.e.\ an up-right path $x_{-\infty,\infty}$ whose finite segments are all geodesics?\medskip

Busemann functions can help answer some (if not all) of the above questions. Take a look, for example, at Theorems \ref{direction}, \ref{direction2}, \ref{left}, \ref{coal:thm}, \ref{unique:thm}, \ref{double:thm} and Corollary \ref{nice-cor} below.\medskip

To see the connection between Busemann functions and geodesics start with formula \eqref{G:ind2}. Say, for simplicity, weights $\w_x$ have a continuous distribution. Then $G_{x+e_1,y}=G_{x+e_2,y}$ happens with zero probability and thus there is a unique $i\in\{1,2\}$ for which 
	\[\w_x=G_{x,y}-G_{x+e_i,y}.\]
The geodesic path from $x$ to $y$ will follow this increment and go from $x$ to $x+e_i$. Then, from there the procedure can be repeated, until the path reaches
the north-east boundary with corner $y$, i.e.\ until one gets to an $x\in\{y-ke_1:k\in\N\}\cup\{y-ke_2:k\in\N\}$. From there, the geodesic marches straight to $y$ using only
$e_1$ or only $e_2$ steps.
This description of geodesics motivates the following.

\begin{lemma}\label{B-geo}
{\rm\cite{Geo-Ras-Sep-17-ptrf-1}}
Let $B:\Z^2\times\Z^2\to\R$ satisfy the cocycle and recovery properties \eqref{cocycle}  and \eqref{recovery}. Let $x_{0,\infty}$ be a path such that
for every $i\ge0$ we have
	\[\w_{x_i}=B(x_i,x_{i+1}).\]
In other words, the path goes along the ``minimal gradient'' of $B$.  Then, $x_{0,\infty}$ is a geodesic.
\end{lemma}

\begin{proof}
Fix $n\ge1$. Consider an arbitrary up-right path $y_{0,n}$ with $y_0=x_0$ and $y_n=x_n$. Write
	\[\sum_{i=0}^{n-1} \w_{x_i}
	=\sum_{i=0}^{n-1}B(x_i,x_{i+1})
	=B(x_0,x_n)
	=\sum_{i=0}^{n-1}B(y_i,y_{i+1})
	\ge\sum_{i=0}^{n-1}\w_{y_i}.\]
(The first equality is from recovery, the second and third use the cocycle property, and the fourth uses recovery again.)
Take a maximum over all up-right paths $y_{0,n}$ between $x_0$ and $x_n$ to get
	\[\sum_{i=0}^{n-1} \w_{x_i}\ge G_{x_0,x_n},\]
which says that $x_{0,n}$ is a geodesic. Since $n$ was arbitrary, the lemma is proved. 
\end{proof}

As a bonus, we  get in the above proof that when $x_{0,\infty}$ follows the smallest gradient of a cocycle $B$ that recovers, we have for $0\le m\le n$
	\begin{align}\label{bonus}
	G_{x_m,x_n}=\sum_{i=m}^{n-1} \w_{x_i}=B(x_m,x_n).
	\end{align}

The above lemma says in particular that Busemann functions $B^\xi$ from \eqref{B-limit} provide us with a ``machine'' to produce infinite geodesics starting from any given point.
Given a starting point $u$, a direction $\xi\in\riUset$, and an integer $j\in\{1,2\}$, let $x_{0,\infty}^{u,\xi,j}$ be the path produced by the following inductive mechanism: $x_0=u$ and for  $k\ge0$, 
if $B^\xi(x_k,x_k+e_1)\ne B^\xi(x_k,x_k+e_2)$, then let $x_{k+1}=x_k+e_i$ for the unique $i\in\{1,2\}$ such that $\w_{x_k}=B^\xi(x_k,x_k+e_i)$. If, on the other hand,
$B^\xi(x_k,x_k+e_1)=B^\xi(x_k,x_k+e_2)$, then break the tie by letting $x_{k+1}=x_k+e_j$.\medskip

Now that we know how to produce geodesics we ask about whether or not these geodesics have an asymptotic direction. We have the following theorem.

\index{geodesic!direction}%
\begin{theorem}\label{direction}
{\rm\cite{Geo-Ras-Sep-17-ptrf-1}}
Make the same assumptions as in Theorem \ref{Bus:thm}. For each $\xi\in\riUset$ we have with probability one that for all $u\in\Z^2$ and $j\in\{1,2\}$ 
	\[\P\big\{\text{geodesic }x_{0,\infty}^{u,\xi,j}\text{ is asymptotically directed into }\Uset_\xi\big\}=1.\]
In particular, if the boundary of $\cC$ is strictly convex at $\xi$, then $x_{0,\infty}^{u,\xi,j}$ has asymptotic direction $\xi$: $n^{-1}x_n^{u,\xi,j}\to\xi$ as $n\to\infty$.
\end{theorem}

The proof of the above theorem needs a fact about stationary $L^1$ cocycles, i.e.\ measurable functions $B:\Omega\times\Z^2\times\Z^2\to\R$ that satisfy \eqref{cocycle} and \eqref{shift} and are such that 
for each $x,y\in\Z^2$ we have $\E[\abs{B(\w,x,y)}]<\infty$.

\index{cocycle!ergodic theorem}%
\begin{theorem}\label{B-shape:thm}
{\rm\cite{Geo-Ras-Sep-17-ptrf-1,Ras-Sep-Yil-13,Geo-etal-15}}
Let $B$ be a stationary $L^1$ cocycle. Define $\Bavg=\E[B(0,e_1)]e_1+\E[B(0,e_2)]e_2$.
Then for any $\xi\in\R_+^2$ we have $\P$-almost surely
	\begin{align}\label{B-ergod}
	\lim_{n\to\infty} \frac{B(0,\fl{n\xi})}n=\Bavg\cdot\xi.
	\end{align}
If furthermore $B$ recovers {\rm(}i.e.\ satisfies \eqref{recovery}{\rm)}, then we also have
	\begin{align}\label{B-shape}
	\lim_{n\to\infty}\max_{x\in\Z_+^2:\abs{x}_1=n}\frac{\abs{B(0,x)-\Bavg\cdot x}}n=0\quad\P\text{-almost surely.}
	\end{align}
\end{theorem}

\begin{proof}[Sketch of proof of Theorem \ref{B-shape:thm}]
As it was the case in Theorem \ref{th:shape}, we will show how the proof of \eqref{B-ergod} goes when $\xi\in\Z_+^2$. The case $\xi\in\R_+^2$ comes with some more work and \eqref{B-shape} comes with considerably more work.
Assume thus that $\xi\in\Z^2_+$. Then we can use the cocycle and stationarity properties to write
	\[B(\w,0,n\xi)=\sum_{i=0}^{n-1}B(\w,i\xi,(i+1)\xi)=\sum_{i=0}^{n-1}B(T_{i\xi}\w,0,\xi).\]
Terms $B(T_{i\xi}\w,0,\xi)$ are just shifted copies of the first term $B(\w,0,\xi)$. As such, $n^{-1}B(\w,0,n\xi)$ can be thought of as a sample mean of, albeit dependent, samples of $B(\w,0,\xi)$. 
A generalization of the law of large numbers, called the ergodic theorem, tells us then that this sample mean converges to the population mean $\E[B(\w,0,\xi)]$. In other words, for $\P$-almost every $\w$
	\begin{align}\label{B-ergod2}
	\lim_{n\to\infty} \frac{B(\w,n\xi)}n=\E[B(0,\xi)].
	\end{align}
Now use the cocycle property again to write
	\[B(0,\xi)=\sum_{i=0}^{\xi\cdot e_1-1}B(ie_1,(i+1)e_1)+\sum_{j=0}^{\xi\cdot e_2-1}B((\xi\cdot e_1)e_1+je_2,(\xi\cdot e_1)e_1+(j+1)e_2).\]
The summands in the first sum are shifted copies of $B(0,e_1)$ and the summands in the second sum are shifted copies of $B(0,e_2)$.
Hence, taking expectation we get $\E[B(0,\xi)]=\E[B(0,e_1)]\xi\cdot e_1+\E[B(0,e_2)]\xi\cdot e_2=\Bavg\cdot\xi$, which combined with \eqref{B-ergod2} proves the claim of the theorem.
\end{proof}

\begin{proof}[Proof of Theorem \ref{direction}]
Let us abbreviate and write $x_n$ for $x^{u,\xi,j}_n$. Let $\zeta\in\Uset$ be a (possibly random) limit point of $x_n/n$, i.e.\ there exists a (possibly random) subsequence $n_j$ such that $x_{n_j}/{n_j}\to\zeta$.

By Lemma \ref{B-geo} we know $x_{0,\infty}$ is a geodesic and by its definition it moves along the smallest gradient of $B^\xi$. By \eqref{bonus}
	\[G_{0,x_n}=B^\xi(x_0,x_n).\]
Divide by $n$ then apply this to $n=n_j$ and use \eqref{shape} and \eqref{B-shape} to deduce that
	\[g(\zeta)=\E[B^\xi(0,e_1)]\zeta\cdot e_1+\E[B^\xi(0,e_2)]\zeta\cdot e_2.\]
Apply the first equality in \eqref{B-g} to get
	\[g(\zeta)=\zeta\cdot\nabla g(\xi).\]
Combine with \eqref{euler} to get
	\[g(\zeta)-g(\xi)=(\zeta-\xi)\cdot\nabla g(\xi).\]
Then, function
	\[f(t)=g(t\xi+(1-t)\zeta)-g(\xi)-(t\xi+(1-t)\zeta-\xi)\cdot\nabla g(\xi),\quad t\in[0,1],\]
satisfies $f(0)=f(1)=0$ and
	\[\lim_{\e\searrow0}\frac{f(1)-f(1-\e)}\e=-\lim_{\e\searrow0}\frac{g(\xi+\e(\zeta-\xi))-g(\xi)}\e+(\zeta-\xi)\cdot\nabla g(\xi)=0.\]
Since $f$ is also concave it is identically $0$ and thus for all $t\in[0,1]$ 
	\[g(t\xi+(1-t)\zeta)-g(\xi)=(t\xi+(1-t)\zeta-\xi)\cdot\nabla g(\xi)=(1-t)(\zeta-\xi)\cdot\nabla g(\xi).\]
This says $g$ is affine on $\{t\xi+(1-t)\zeta:0\le t\le1\}$ and thus $\zeta\in\Uset_\xi$. The theorem is proved.
\end{proof}

Now that we know that geodesics generated using Busemann functions $B^\xi$ have an asymptotic direction we can prove the same thing about all geodesics.

\index{geodesic!direction}%
\begin{theorem}\label{direction2}
{\rm\cite{Geo-Ras-Sep-17-ptrf-1}}
Make the same assumptions as in Theorem \ref{Bus:thm}. With probability one, any geodesic is asymptotically directed into $\Uset_\xi$ for some $\xi\in\Uset$.
\end{theorem}

The proof will need one more fact about geodesics $x_{0,\infty}^{u,\xi,j}$.

\begin{lemma}\label{order}
{\rm\cite{Geo-Ras-Sep-17-ptrf-1}}
For all $n\ge m$, $x_{m,n}^{u,\xi,1}$ is the right-most geodesic between its two endpoints: if $y_{m,n}$ is a geodesic between $x_m^{u,\xi,1}$ and $x_n^{u,\xi,1}$, then we have
$y_k\cdot e_1\le x_k^{u,\xi,1}\cdot e_1$ for $m\le k\le n$. Similarly, $x_{m,n}^{u,\xi,2}$ is the left-most geodesic between its two endpoints.
\end{lemma}

\begin{proof}
We will prove the claim about $x_{m,n}^{u,\xi,1}$, the other one being symmetric. Abbreviate this path by writing $x_{m,n}$.
Take $y_{m,n}$ as in the claim. In particular, $y_m=x_m$. For starters we want to prove that $y_{m+1}\cdot e_1\le x_{m+1}\cdot e_1$.
For this, we only need to consider the case when $x_{m+1}=x_m+e_2$, for the inequality clearly holds in the other case.
Since we are using a superscript $j=1$ it cannot be that $B^\xi(x_m,x_m+e_1)=B^\xi(x_m,x_m+e_2)$, for otherwise the path would have taken an $e_1$-step out of $x_m$.
Since the path always takes a step along the smaller $B^\xi$ gradient and since $B^\xi$ recovers, we conclude that in the case at hand we have $\w_{x_m}=B^\xi(x_m,x_m+e_2)<B^\xi(x_m,x_m+e_1)$.

Now, recovery and the cocycle property imply that $G_{x,y}\le B^\xi(x,y)$ for any $x\le y$. Combine this with \eqref{bonus} and the cocycle property  again to get
	\begin{align*}
	\w_{x_m}+G_{x_m+e_1,x_n}
	&\le B^\xi(x_m,x_m+e_2)+B^\xi(x_m+e_1,x_n)\\
	&<B^\xi(x_m,x_m+e_1)+B^\xi(x_m+e_1,x_n)=B^\xi(x_m,x_n)=G_{x_m,x_n}.
	\end{align*}
Therefore, no geodesic from $x_m$ to $x_n$ can go through $x_m+e_1$ and we have $y_{m+1}=x_m+e_2=x_{m+1}$. 

Now repeat this argument every time $x_{m,n}$ and $y_{m,n}$ intersect to see that the latter never goes to the ``right'' of the former. The lemma is proved.
\end{proof}

One can squeeze the proof of the above lemma a little bit more to get the following interesting result.

\begin{theorem}\label{left}
{\rm\cite{Geo-Ras-Sep-17-ptrf-1}}
Make the same assumptions as in Theorem \ref{Bus:thm}. 
\begin{enumerate}[\ \ {\rm(}i{\rm)}]
\item\label{left-i} Fix $\xi\in\riUset$. With probability one and for all $u\in\Z^2$, $x_{0,\infty}^{u,\xi,1}$ is the right-most geodesic directed into $\Uset_\xi$ and
$x_{0,\infty}^{u,\xi,2}$ is the left-most geodesic directed into $\Uset_\xi$.  
\item\label{left-ii} With probability one and for any $u\in\Z^2$, every infinite geodesic out of $u$ stays between $x_{0,\infty}^{u,\xi,1}$ and $x_{0,\infty}^{u,\xi,2}$ for some $\xi\in\riUset$.
\end{enumerate}
\end{theorem}

%
%

\begin{proof}[Proof of Theorem \ref{direction2}]
First, observe that although we proved Theorem \ref{direction} and Lemma \ref{order} for a fixed $\xi\in\riUset$, they both hold simultaneously (i.e.\ with one null set thrown away) for all $\xi\in\riUset\cap\Q^2$, which is countable and dense in $\riUset$.

Assume that for some geodesic $x_{0,\infty}$, $x_n/n$ has  limit points in both $\Uset_\zeta$ and $\Uset_\eta$ with $\zeta,\eta\in\Uset$ and $\Uset_\zeta\ne\Uset_\eta$. We can assume
$\zeta\cdot e_1<\eta\cdot e_1$.
Since we have assumed $g$ to be differentiable, there must exist at least one (and in fact infinitely many) point(s) $\xi\in\riUset\cap\Q^2$ such that 
	\begin{align}\label{aux}
	\zeta\cdot e_1<\xi\cdot e_1<\eta\cdot e_1,\quad\Uset_\xi\not=\Uset_\zeta,\quad\text{and}\quad\Uset_\xi\not=\Uset_\eta.
	\end{align}
Let $x_0=u$ (the starting point of the geodesic under study).
Since we have shown that geodesic $x_{0,\infty}^{u,\xi,1}$ has asymptotic direction $\xi$, the ordering in \eqref{aux} implies that  $x_{0,\infty}$ passes infinitely often to the left of $x_{0,\infty}^{u,\xi,1}$.
But then Lemma \ref{order} implies that once the former goes strictly to the left of the latter, it has to remain (weakly) on that side forever. A similar argument shows that $x_{0,\infty}$ must also eventually
stay to the right of $x_{0,\infty}^{u,\xi,2}$. In other words, $x_{0,\infty}$ eventually stays between $x_{0,\infty}^{u,\xi,1}$ and $x_{0,\infty}^{u,\xi,2}$. But both these geodesics are directed into $\Uset_\xi$.
Hence, so is $x_{0,\infty}$, which contradicts the assumption that $x_n/n$ has limit points in $\Uset_\zeta$ and $\Uset_\eta$.
\end{proof}

Theorems \ref{direction} and \ref{direction2} have a very nice consequence when we know more about the regularity of $g$.

\begin{corollary}\label{nice-cor}
{\rm\cite{Geo-Ras-Sep-17-ptrf-1}}
Make the same assumptions as in Theorem \ref{Bus:thm}. Assume also that $g$ is strictly concave. Then 
\begin{enumerate}[\ \ {\rm(}i{\rm)}]
\item\label{cor1} For any given direction, with probability one, out of any given point,
 there exists an infinite geodesic  going in this direction:
	\[\forall\xi\in\Uset:\quad\P\big\{\forall u\in\Z^2\ \exists x_{0,\infty}\text{ geodesic }:x_n/n\to\xi\big\}=1;\]
\item\label{cor2} With probability one, every infinite geodesic has an asymptotic direction:
	\[\P\big\{\forall x_{0,\infty}\text{ geodesic }\ \exists\xi\in\Uset:x_n/n\to\xi\big\}=1.\]
\end{enumerate}
\end{corollary}

This simply follows from the fact that if $g$ is strictly concave, then $\Uset_\xi=\{\xi\}$ for all $\xi\in\Uset$.
Strict concavity is still an open question, but it is believed to hold in general, either when the maximum of $\w_0$ does not percolate or outside the flat segment that occurs when the maximum does percolate 
(see Theorem \ref{flat:thm}).\medskip

The claims in the above corollary appeared before  in Proposition 7 of \cite{Fer-Pim-05} for the solvable model where weights $\w_x$ are exponentially distributed. 
Note that in this case formula \eqref{shape:exp} gives an explicit expression for $g$ and we can  check directly that $g$ is indeed strictly concave. 

The approach used by \cite{Fer-Pim-05} follows the ideas of Licea and 
\index{Newman, Charles}%
Newman \cite{Lic-New-96} for nearest-neighbor 
\index{first-passage percolation (FPP)}%
\index{FPP}%
first-passage percolation (FPP).
In \cite{Lic-New-96} the authors assume a certain global curvature assumption on $g$ and use it to control how much infinite geodesics can wander, proving existence and directedness of infinite geodesics.
They also use a lack-of-space argument to prove 
\index{geodesic!coalescence}%
coalescence (i.e.\ merger) of geodesics with a given asymptotic direction. This is the only method known to date for proving coalescence of directional geodesics.
The same idea was adapted by \cite{Fer-Pim-05}  to the directed LPP model with exponential weights and then by \cite{Geo-Ras-Sep-17-ptrf-1}  to the general weights setting.

\begin{theorem}\label{coal:thm}
{\rm\cite{Geo-Ras-Sep-17-ptrf-1}}
Make the same assumptions as in Theorem \ref{Bus:thm}. Fix $\xi\in\riUset$. With probability one and for all $u,v\in\Z^2$ the right-most geodesics directed into $\Uset_\xi$ and starting at $u$ and at $v$ coalesce:
there exist $m,n\ge0$ such that $x_{m,\infty}^{u,\xi,1}=x_{n,\infty}^{v,\xi,1}$. The same claim holds for the left-most geodesics. 
\end{theorem}

Here is a very rough and high level sketch of how such a coalescence result is proved. Details can be found in Appendix A of \cite{Geo-Ras-Sep-15} (which is an extended version of \cite{Geo-Ras-Sep-17-ptrf-2}).
See also the proof of Theorem 2.3 in \cite{ch:Hanson}.
First observe that if $x_{0,\infty}^{u,\xi,1}$ and $x_{0,\infty}^{v,\xi,1}$ ever intersect, then from there on they follow the same evolution (smallest $B^\xi$ increment and increment $e_1$ in case of a tie).
Therefore, the task is really to prove that they eventually intersect.
By stationarity the assumption of two nonintersecting geodesics implies we can find at least three nonintersecting ones.  
A   local modification of the weights  turns  the middle geodesic of the triple   into a geodesic that stays disjoint from all geodesics that emanate from sufficiently far away.   By stationarity again  
at least $\delta L^2$ such disjoint geodesics   emanate from an $L\times L$ square.  This gives a contradiction because there are only $2L$ boundary points for these geodesics to exit through.  \medskip
 
\begin{corollary}\label{B-not-cont}
Make the same assumptions as in Theorem \ref{Bus:thm}. Then with probability one
	\[\{B^\xi(0,e_1):\xi\in\riUset\}\subset\{G_{0,z}-G_{e_1,z}:z\in e_1+\Z_+^2\}.\]
In particular the map $\xi\mapsto B^\xi(0,e_1)$ cannot be continuous on all of $\riUset$.
\end{corollary} 

\begin{proof}
For each $\xi\in\riUset\cap\Q^2$ consider geodesics $x_{0,\infty}^{0,\xi,1}$ and $x_{0,\infty}^{e_1,\xi,1}$. Theorem \ref{direction} says that with probability one these geodesics are asymptotically directed into $\Uset_\xi$.
Then Theorem \ref{Bus:thm} implies that almost surely and for all $\xi\in\riUset\cap\Q^2$
	\[B^\xi(0,e_1)=\lim_{n\to\infty}(G_{0,x_n^{0,\xi,1}}-G_{e_1,x_n^{0,\xi,1}}).\]
By Theorem \ref{coal:thm} the two geodesics $x_{0,\infty}^{0,\xi,1}$ and $x_{0,\infty}^{e_1,\xi,1}$ coalesce at say a point we denote by $z^\xi$. This leads to
	\[B^\xi(0,e_1)=G_{0,z^\xi}-G_{e_1,z^\xi}.\]
Therefore, we have almost surely
	\[\{B^\xi(0,e_1):\xi\in\riUset\cap\Q^2\}\subset\{G_{0,z}-G_{e_1,z}:z\in e_1+\Z_+^2\}.\]
The claim now follows from monotonicity \eqref{lm:monotone}. 
\end{proof}
 
Note that the above result does not contradict \eqref{B-cont}. Together, the two results say that there is zero probability that a given (fixed) $\xi\in\riUset$ happens to be a discontinuity point of $\xi\mapsto B^\xi(0,e_1)$.\medskip

When the weights have a continuous distribution one has a 
\index{geodesic!unique}%
{\sl unique} geodesic between any two given points. What about infinite directional geodesics?
The answer is also in the positive.

\begin{theorem}\label{unique:thm}
{\rm\cite{Geo-Ras-Sep-17-ptrf-1}}
Make the same assumptions as in Theorem \ref{Bus:thm}.  Assume also that $\w_0$ has a continuous distribution.
Fix $\xi\in\riUset$. Then with probability one, out of any $u\in\Z^2$, there exists a unique infinite geodesic directed into $\Uset_\xi$.
\end{theorem}

\begin{proof}
In view of Theorem \ref{left}\eqref{left-ii}, it is enough to show that $x_{0,\infty}^{u,\xi,1}=x_{0,\infty}^{u,\xi,2}$.
This in turn follows from the coalescence result. Indeed, assume the two geodesics do not match. We can assume that they separate right away, otherwise just consider the paths starting at the separation point.
Under this assumption, it must be the case that $\w_u=B^\xi(u,u+e_1)=B^\xi(u,u+e_2)$, for otherwise both geodesics would have followed the smaller $B^\xi$-gradient and thus stayed together.
Now, the above coalescence result implies that $x_{0,\infty}^{u,\xi,1}$ and $x_{0,\infty}^{u+e_2,\xi,1}$ will eventually coalesce, say at point $v=x_n^{u,\xi,1}=x_{n-1}^{u+e_2,\xi,1}$.
But then applying \eqref{bonus} we would have
	\[\sum_{i=0}^{n-1}\w(x_i^{u,\xi,1})=B^\xi(u,v)=B^\xi(u,u+e_2)+B^\xi(u+e_2,v)=\w_u+\sum_{i=0}^{n-2}\w(x_i^{u+e_2,\xi,1}),\]
which says that the weights add up to the same amount along two different paths. This happens with zero probability if weights have a continuous distribution. Hence the two geodesics can never separate and 
the theorem is proved. (We used $\w(x)$ to denote $\w_x$, for aesthetic reasons.)
\end{proof}

One of the things Theorem \ref{unique:thm} is really saying is that one should not need to worry about breaking ties among $B^\xi$ gradients. Let us spell this out as a separate result.

\begin{theorem}\label{tie:thm}
{\rm\cite{Geo-Ras-Sep-17-ptrf-1}}
Make the same assumptions as in Theorem \ref{Bus:thm}.  Assume also that $\w_0$ has a continuous distribution.
Fix $\xi\in\riUset$. Then $\P\{\exists u: B^\xi(u,u+e_1)=B^\xi(u,u+e_2)\}=0.$
\end{theorem}

\begin{proof}
The proof is quite simple: when a tie happens at $u$ geodesics
$x_{0,\infty}^{u,\xi,1}$ and $x_{0,\infty}^{u,\xi,2}$ separate right away. Since we just showed this cannot happen when weights have a continuous distribution, the theorem follows.
\end{proof}

As the last result of this section we address existence doubly-infinite geodesics (or rather lack thereof). 

\index{geodesic!doubly-infinite}%
\index{bigeodesic}%
\begin{theorem}\label{double:thm}
{\rm\cite{Geo-Ras-Sep-17-ptrf-1}}
Make the same assumptions as in Theorem \ref{Bus:thm}.  Assume also that $\w_0$ has a continuous distribution.
Fix $\xi\in\riUset$. Then
	\[\P\big\{\exists x_{-\infty,\infty}\text{ geodesic}:x_{0,\infty}\text{ is directed into }\Uset_\xi\big\}=0.\]
\end{theorem}

\begin{proof}[Sketch of the proof]
If such a doubly-infinite geodesic existed, with positive probability, then by stationarity we would have another (different) one $y_{-\infty,\infty}$, also directed into $\Uset_\xi$.
By the coalescence result we would have that $x_{0,\infty}$ and $y_{0,\infty}$ coalesce. Modulo renumbering the indices, we can assume that $x_{0,\infty}=y_{0,\infty}$
but $x_{-1}\ne y_{-1}$. Because weights have a continuous distribution, it cannot be that $x_{-n}=y_{-n}$ for any $n>1$ (otherwise the weights would add up to the same amount $G_{x_{-n},0}$ along
two different paths $x_{-n,0}$ and $y_{-n,0}$).   We thus have a bi-infinite three-armed ``fork'' embedded in $\Z^2$. But by stationarity this picture will repeat infinitely often, allowing us to embed a binary tree into $\Z^2$, with its vertices having 
a positive density.  This embedding is not possible and thus we have a contradiction. (Such a tree grows exponentially fast, while the boundary of a box in $\Z^2$ 
grows only linearly in the diameter of the box.)
See Figure \ref{bi-infinite:fig} for an illustration and \cite{Geo-Ras-Sep-17-ptrf-1} for the details.
\end{proof}

\begin{figure}
	\begin{center}
		\begin{tikzpicture}[ >=latex, scale=0.5]

			\draw[line width=1pt](-2,-0.5)--(-2,0)--(2,0)--(2,1)--(3,1)--(3,2)--(5,2)--(5,3)--(8,3)--(8,4)--(10,4);
			\draw[line width=1pt](-1,-1.5)--(-1,-1)--(0,-1)--(0,0);
			\draw[line width=1pt](8,0.5)--(8,1)--(9,1)--(9,4);
			\draw[line width=1pt](-0.5,-4)--(1,-4)--(1,-2)--(4,-2)--(4,-1)--(5,-1)--(5,1)--(6,1)--(6,3);
			\draw[line width=1pt](2.5,-5)--(3,-5)--(3,-2);
			
			\draw [line width=1pt,->] plot [smooth] coordinates {(10,4) (10.3,4.2) (10.5,5) (11.5,5) (12.5,6)};
			\draw [line width=1pt,->] plot [smooth] coordinates {(2.5,-5) (1.8,-5.1) (1.7,-5.8) (1,-6.5)};
			\draw [line width=1pt,->] plot [smooth] coordinates {(-0.5,-4) (-1.2,-4.1) (-1.3,-4.8) (-2,-5)};
			\draw [line width=1pt,->] plot [smooth] coordinates {(-1,-1.5) (-1.2,-2.5) (-1.9,-2.7)};
			\draw [line width=1pt,->] plot [smooth] coordinates {(-2,-0.5) (-2.2,-1.5) (-3.3,-1.7) (-3.5,-2.4) (-4.5,-2.5)};
			\draw [line width=1pt,->] plot [smooth] coordinates {(8,0.5) (7.9,-0.5) (6.9,-1.5)};
			
			\draw[fill =black](0,0)circle(2mm);
			\draw[fill =black](3,-2)circle(2mm);
			\draw[fill =black](6,3)circle(2mm);
			\draw[fill =black](6,3)circle(2mm);
			\draw[fill =black](9,4)circle(2mm);
			
			\draw(12,4.5) node[inner sep=0.6pt,sloped,rotate=30]{\small into $\Uset_\xi$};

		\end{tikzpicture}
	\end{center}
	\caption{\small  The tree of bi-infinite geodesics. Bullets mark the triple split points. They have a positive density.}
	\label{bi-infinite:fig}
\end{figure}
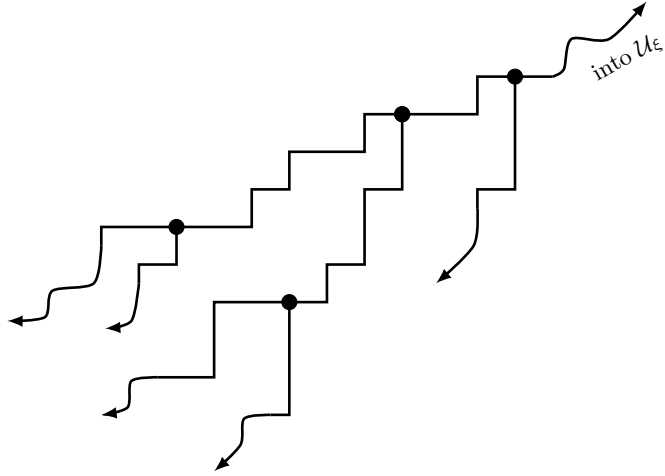


The proof of the above theorem used the coalescence result, which requires us to fix the direction of the geodesics. But we will see in 
the next section that there are random directions in which there are multiple geodesics out of say the origin. 
So are there then doubly-infinite geodesics in these random directions? The answer is still expected to be in the negative, but a proof remains
elusive.



\section{The competition interface}\label{cif:sec}
\index{competition interface}%
 
We continue to assume  the conditions of Theorem \ref{Bus:thm} to be satisfied. In particular, the shape $g$ is still assumed differentiable on $(0,\infty)^2$.
Let us also assume in this section that weights $\w_x$ have a continuous distribution. 

Recall our earlier definition of the competition interface separating the two geodesic trees rooted at $e_1$ and $e_2$ (see Figure \ref{CIF:fig}).
Denote this up-right path of sites in $\Z^2+(e_1+e_2)/2$ by $\varphi_n$. In particular, $\varphi_0=(e_1+e_2)/2$. \medskip

Does $\varphi_n$ have an asymptotic direction?  What can we say about this direction?  Can we describe $\varphi_n$ using Busemann functions, as we did for geodesics in the previous section? \medskip

By monotonicity \eqref{monotone} of the Busemann functions we have that
	\[B^\xi(0,e_1)-B^\xi(0,e_2)\]
is monotone in $\xi\in\riUset\cap\Q^2$. Namely, the above is nonincreasing as $\xi\cdot e_1$ increases in $(0,1)\cap\Q$. (The reason for only considering rational directions is that the limit in \eqref{B-limit} holds for configurations $\w$ outside 
a set of measure zero, but this null set depends on the direction $\xi$.  Thus, the limit can be claimed to hold almost surely for only countably many directions at once.)

By Theorem \ref{tie:thm} we have $B^\xi(0,e_1)\ne B^\xi(0,e_2)$ almost surely for all $\xi\in\riUset\cap\Q^2$.
If
	\[\P\big\{B^\xi(0,e_1)>B^\xi(0,e_2)\  \forall \xi\in\riUset\cap\Q^2\big\}>0,\] 
then, due to Theorem \ref{left}\eqref{left-ii}, for the configurations in the above event all infinite geodesics out $0$ must start with an $e_2$-step.
This can be contradicted by following an argument similar to the proof of Lemma \ref{exist-geo} to show that, with probability one, there exists at least one geodesic out of $0$ that takes a first step $e_1$.
A similar reasoning applies for the case where with positive probability $B^\xi(0,e_1)<B^\xi(0,e_2)$ for all $\xi\in\riUset$.

Then, with probability one there exists a unique $\cid\in\riUset$ such that for all $\xi\in\riUset\setminus\{\cid\}$
	\[B^\xi(0,e_1)>B^\xi(0,e_2)\text{ if }\xi\cdot e_1<\cid\cdot e_1\quad{and}\quad B^\xi(0,e_1)<B^\xi(0,e_2)\text{ if }\xi\cdot e_1>\cid\cdot e_1.\] 

One thing this says is that geodesics originating at $0$ that are directed into $\Uset_\xi$ with $\xi\cdot e_1<\cid\cdot e_1$ (i.e.\ $\xi$ to the left of $\cid$) must start with an $e_2$ step.
Similarly, geodesics originating at $0$ that are directed into $\Uset_\xi$ with $\xi\cdot e_1>\cid\cdot e_1$ (i.e.\ $\xi$ to the right of $\cid$) must start with an $e_1$ step. A slightly sharper version of this argument leads to the following.

\index{competition interface}%
\begin{theorem}\label{cif:thm}
{\rm\cite{Geo-Ras-Sep-17-ptrf-1}}
Make the same assumptions as in Theorem \ref{Bus:thm}.  Assume also that $\w_0$ has a continuous distribution.
\begin{enumerate}[\ \ {\rm(}i{\rm)}]
\item\label{cif-lln} With probability one, competition interface $\varphi_n$ has asymptotic direction $\cid$:
	\[\P\big\{\varphi_n/n\to\cid\big\}=1.\]
\item\label{2geo} With probability one, there exist two infinite geodesics out of $0$ with asymptotic direction $\cid$, one going through $e_1$ and the other through $e_2$:
	\[\P\big\{\exists x^1_{0,\infty},x^2_{0,\infty}\text{ geodesics }:x^1_1=e_1,\ x^2_1=e_2,\ x^1_n/n\to\cid,\ x^2_n/n\to\cid\big\}=1.\]
\item\label{cont-dist} $\cid$ is a genuine random variable that has a continuous distribution and is supported outside the linear segments of $g$ {\rm(}if any{\rm)}:
	\[\forall\xi\in\Uset:\quad\P\{\cid=\xi\}\le\P\{\cid\in\Uset_\xi\}=0.\]
\item\label{cid-supp} $\cid$ is supported on all of $\Uset$, take away the linear segments of $g$ {\rm(}if any{\rm)}: for any open interval (i.e.\ connected subset) $\mathcal V\subset\Uset$ such that $\Uset_\xi=\{\xi\}$ $\forall\xi\in\mathcal V$,
we have $\P\{\cid\in\mathcal V\}>0$.
\end{enumerate}
\end{theorem}

It is conjectured that $g$ is strictly concave when weights $\w_x$ have a continuous distribution. Hence, $\cid$ is expected to be supported on all of $\Uset$.

Since weights are assumed to be continuous, the two geodesics in Theorem \ref{cif:thm}\eqref{2geo} cannot coalesce. This does not contradict Theorem \ref{coal:thm} because direction $\cid$ is random while the coalescence result
was about fixed deterministic directions. In fact, this is one way to see why claim
\eqref{cont-dist} is true.\medskip

In the solvable case of exponentially distributed weights, one can compute the distribution of $\cid$ explicitly. 

\begin{lemma}
Assume $\w_0$ is exponentially distributed with rate $\theta>0$. Then for $a\in(0,1)$
	\begin{align}\label{cid-dist}
	\P\{\cid\cdot e_1>a\}=\frac{\sqrt{1-a}}{\sqrt{a}+\sqrt{1-a}}\,.
	\end{align}
If we define the angle $\cid$ makes with $e_1$ by $\theta^*=\tan^{-1}(\cid\cdot e_2/\cid\cdot e_1)$, then 
	\begin{align}\label{theta-dist}
	\P\{\theta^*\le t\}=\frac{\sqrt{\sin t}}{\sqrt{\sin t}+\sqrt{\cos t}}\,.
	\end{align}
\end{lemma}

\begin{proof}
Recall that $B^\xi(0,e_i)$ has an exponential distribution with rate 
	\[\lambda_i=\frac{\theta\sqrt{\xi\cdot e_i}}{\sqrt{\xi\cdot e_1}+\sqrt{\xi\cdot e_2}}.\]
Furthermore, $B^\xi(0,e_1)$ and $B^\xi(0,e_2)$ are independent. See Section \ref{fixed-pt} below. Now compute
	\begin{align*}
	\P\{\cid\cdot e_1>a\}
	&=\P\{B^{ae_1+(1-a)e_2}(0,e_1)>B^{ae_1+(1-a)e_2}(0,e_2)\}\\
	&=\int_0^\infty \lambda_2 e^{-\lambda_2 s} \P\{B^{ae_1+(1-a)e_2}(0,e_1)>s\}\,ds\\
	&=\int_0^\infty \lambda_2 e^{-\lambda_2 s} e^{-\lambda_1 s}\,ds=\frac{\lambda_2}{\lambda_1+\lambda_2}\,,
	\end{align*}
from which \eqref{cid-dist} follows. For the distribution of $\theta^*$ we have for $t\in(0,\pi/2)$
	\[\P\{\theta^*\le t\}=\P\{\cid\cdot e_2\le \cid\cdot e_1\tan t\}=\P\{\cid\cdot e_1\ge 1/(1+\tan t)\}=\frac{\sqrt{\tan t}}{1+\sqrt{\tan t}}\,,\]
which is \eqref{theta-dist}.
\end{proof}

The competition interface of the 
\index{corner growth model (CGM)!exponential}%
\index{CGM!exponential}%
exponential corner growth model maps to a certain object called the 
\index{second-class particle}%
{\sl second-class particle} in TASEP, so this object has been studied from both perspectives.  
In this case, a weak-limit version of Theorem \ref{cif:thm}\eqref{cif-lln} follows from translating a result of Ferrari and Kipnis  \cite{Fer-Kip-95} on the limit of the scaled location of the second-class particle in TASEP to LPP language.
Almost sure convergence was shown by Mountford and Guiol \cite{Mou-Gui-05} using  concentration inequalities and the TASEP variational formula of Sepp\"al\"ainen \cite{Sep-99-aop}.  
Concurrently, \cite{Fer-Pim-05}  gave a different proof of almost sure convergence of $\phi_n/n$ by applying the techniques of directed geodesics and then obtained the distribution \eqref{theta-dist} of the angle of the asymptotic direction $\cid$ 
from  the TASEP results of \cite{Fer-Kip-95}.  

Later, these results on the direction of the competition interface were extended from the quadrant  to larger classes of initial profiles in two rounds: first by 
\cite{Fer-Mar-Pim-09} still with TASEP and geodesic techniques,  and then by 
\cite{Cat-Pim-13}  using their earlier  results  on Busemann functions \cite{Cat-Pim-12}.  

Coupier \cite{Cou-11}  also relied on the TASEP connection   to sharpen the geodesics  results of \cite{Fer-Pim-05}.  
He showed that with probability one there are no triple geodesics (out of the origin) in any direction.

\section{History}

We now give a quick overview of the last twenty or so years of research  on Busemann functions and geodesics in percolation.  

As we mentioned above, Licea and 
\index{Newman, Charles}%
Newman  \cite{Lic-New-96} were the first to introduce a technique for 
proving existence, uniqueness, and coalescence of directional geodesics under a global curvature assumption on the limit shape to control how much geodesics deviate from a straight line, and then
as a consequence deducing the existence  of Busemann functions. See also the summary in Newman's ICM paper \cite{New-95}.
Although verifying the curvature assumption for percolation models with general weights remains an open problem,  it can be done in a number of special cases.
Thus, Licea and Newman's approach was applied to directed LPP with exponential weights \cite{Fer-Pim-05} 
and to several other (non-lattice) models built on homogeneous Poisson processes \cite{How-New-01,Wut-02,Cat-Pim-11,Bak-Cat-Kha-14}.   
In the case of the 
\index{corner growth model (CGM)!exponential}%
\index{CGM!exponential}%
exponential corner growth model,  another set of tools comes from its connection with TASEP, as explained at the end of Section \ref{cif:sec}.


The idea of deducing existence and uniqueness of stationary processes by studying  geodesic-like objects has also been used in random dynamical systems.
For example, this is how \cite{E-etal-00} and its extensions \cite{Hoa-Kha-03,Itu-Kha-03,Bak-07,Bak-Kha-10,Bak-13} show existence of invariant measures for the Burgers equation with random forcing. 
These works treated cases where space is compact or essentially compact.  To make progress in a non-compact case,  the approach of Newman et al.\ was adopted again in \cite{Bak-Cat-Kha-14,Bak-16,Bak-Li-16-}. 

The approach we presented in this chapter  is the one we took in \cite{Geo-Ras-Sep-17-ptrf-1,Geo-Ras-Sep-17-ptrf-2} and is the very opposite of the above.    
Using the connection to queues in tandem the Busemann limits are constructed a priori in the form of stationary cocycles
that come from certain 
\index{invariant measure}%
{\sl invariant} measures of the queuing system.  Using a certain monotonicity the cocycles are then compared to the gradients of passage times.    
The monotonicity,  ergodicity,  and differentiability (rather than curvature) of the limit shape  give the control that proves the Busemann limits.   After establishing existence of Busemann functions,
we use them to prove existence, directedness, coalescence, and uniqueness results about the geodesics.

A similar approach was carried out by Damron and Hanson \cite{Dam-Han-14,Dam-Han-17} for the standard first-passage percolation model.
They first construct (generalized) Busemann functions from  weak subsequential limits of first-passage time differences.  
These  weak Busemann  limits  can  be regarded as a  counterpart of our stationary cocycles. 
This then gives  access to  properties of geodesics,   while weakening  the need for the global curvature assumption.  

 An   independent line of work is that of  Hoffman \cite{Hof-05,Hof-08} on the standard first-passage percolation,  with  general  weights and without 
  any regularity assumptions on the limit shape.  Assuming all semi-infinite geodesics coalesce, \cite{Hof-05} constructed a Busemann function and used it to get a contradiction, concluding that 
  there are at least two semi-infinite geodesics.   (\cite{Gar-Mar-05} gave  an independent proof with a different method.)   \cite{Hof-08} extended this  to at least four geodesics.

\section{Next: fluctuations}

The results we presented can be thought of as analogues of the law of large numbers. A natural follow-up is to study the analogue of the central limit theorem, i.e.\ questions concerning the size of deviations of the passage times
$G_{0,\fl{n\xi}}$ from their asymptotic limit  $n g(\xi)$ and of the geodesics, both finite (i.e.\ from $0$ to $\fl{n\xi}$) and infinite (i.e.\ going in direction $\xi$), from a straight line. 

Due to the maximum in the definition of the passage times, $G_{0,\fl{n\xi}}-ng(\xi)$ should be tighter than in the case of just a sum of identically distributed independent random variables. That is,
its fluctuations should be smaller than order $n^{1/2}$. On the other hand, one would expect the geodesic paths to wander a great deal in search of favorable weights. For example,
if $x_{0,\infty}$ is a path that follows the smallest $B^\xi$ gradient, then it should be the case that $x_n-n\xi$ has fluctuations of order larger than $n^{1/2}$.   

We will see in articles \cite{ch:Seppalainen} and \cite{ch:Corwin} how the above can be answered quite precisely for the solvable models:  $G_{0,\fl{n\xi}}-ng(\xi)$ fluctuates on order $n^{1/3}$ (and we can even determine the limiting 
\index{Tracy-Widom distribution}%
{\sl Tracy-Widom distribution} of 
$n^{-1/3}(G_{0,\fl{n\xi}}-ng(\xi))$ and $x_n-n\xi$ fluctuates on the order $n^{2/3}$ (but here, a limiting distribution of $(x_n-n\xi)/n^{2/3}$ is only conjectured).


Just like it is the case for the central limit theorem, this behavior is believed to be universal, going beyond solvable models, but this is far from being proved. See, however, \cite{Alb-Kha-Qua-14-aop} and \cite{Kri-Qua-16-} for results 
towards this universality conjecture.


\bibliographystyle{amsplain}

\bibliography{Rassoul-Agha}

\end{document}

%% file: CGM.tex

\begin{tikzpicture}[>=latex,scale=0.33]

\begin{scope}[shift={(15,0)}, local bounding box=aa]

\draw[line width=\thick,color=sussexb](0,0)--(6,0);
\draw[line width=\thick,color=sussexb](0,0)--(0,5); 
\draw[line width=0.5pt](0.5,6)--(6,0.5);
\draw[ fill =sussexp](1,5.5)circle(2mm);
\draw (1,5.5)+(0.3,0.3) node[color=sussexp,inner sep=0.6pt,sloped,above,rotate=-45] {\tiny4};     
\draw[ fill =sussexp](1.5,5)circle(2mm);    
\draw (1.5,5)+(0.3,0.3) node[color=sussexp,inner sep=0.6pt,sloped,above,rotate=-45] {\tiny3};     
\draw[ fill =sussexp](2,4.5)circle(2mm);    
\draw (2,4.5)+(0.3,0.3) node[color=sussexp,inner sep=0.6pt,sloped,above,rotate=-45] {\tiny2};     
\draw[ fill =sussexp](2.5,4)circle(2mm);    
\draw (2.5,4)+(0.3,0.3) node[color=sussexp,inner sep=0.6pt,sloped,above,rotate=-45] {\tiny1};     
\draw[ fill =sussexp](3,3.5)circle(2mm);    
\draw (3,3.5)+(0.3,0.3) node[color=sussexp,inner sep=0.6pt,sloped,above,rotate=-45] {\tiny0};     
\draw[ fill =white](3.5,3)circle(2mm);
\draw (3.5,3)+(0.3,0.3) node[shape=rectangle,inner sep=0.8pt,minimum size=1pt,draw,sloped,above,rotate=-45] {\tiny0};     
\draw[ fill =white](4,2.5)circle(2mm);    
\draw (4,2.5)+(0.3,0.3) node[shape=rectangle,inner sep=0.8pt,minimum size=1pt,draw,sloped,above,rotate=-45] {\tiny1};     
\draw[ fill =white](4.5,2)circle(2mm);    
\draw (4.5,2)+(0.3,0.3) node[shape=rectangle,inner sep=0.8pt,minimum size=1pt,draw,sloped,above,rotate=-45] {\tiny2};     
\draw[ fill =white](5,1.5)circle(2mm);    
\draw (5,1.5)+(0.3,0.3) node[shape=rectangle,inner sep=0.8pt,minimum size=1pt,draw,sloped,above,rotate=-45] {\tiny3};     
\draw[ fill =white](5.5,1)circle(2mm);    
\draw (5.5,1)+(0.3,0.3) node[shape=rectangle,inner sep=0.8pt,minimum size=1pt,draw,sloped,above,rotate=-45] {\tiny4};     

\end{scope}

\begin{scope}[shift={(22.5,0)}, local bounding box=bb]

\draw[line width=\thick,color=sussexb](1,0)--(6,0);
\draw[line width=\thick,color=sussexb](0,1)--(0,5); 
\draw[fill=mygray](0,0)--(0,1)--(1,1)--(1,0)--cycle;
\draw[line width=\thick,color=sussexb](1,0)--(1,1)--(0,1); 
\draw[fill=red](0.5,0.5)circle(2mm);
\draw[line width=0.5pt](0.5,6)--(6,0.5);
\draw[ fill =sussexp](1,5.5)circle(2mm);
\draw (1,5.5)+(0.3,0.3) node[color=sussexp,inner sep=0.6pt,sloped,above,rotate=-45] {\tiny4};     
\draw[ fill =sussexp](1.5,5)circle(2mm);    
\draw (1.5,5)+(0.3,0.3) node[color=sussexp,inner sep=0.6pt,sloped,above,rotate=-45] {\tiny3};     
\draw[ fill =sussexp](2,4.5)circle(2mm);    
\draw (2,4.5)+(0.3,0.3) node[color=sussexp,inner sep=0.6pt,sloped,above,rotate=-45] {\tiny2};     
\draw[ fill =sussexp](2.5,4)circle(2mm);    
\draw (2.5,4)+(0.3,0.3) node[color=sussexp,inner sep=0.6pt,sloped,above,rotate=-45] {\tiny1};     
\draw[ fill =white](3,3.5)circle(2mm);    
\draw (3,3.5)+(0.3,0.3) node[shape=rectangle,inner sep=0.8pt,minimum size=1pt,draw,sloped,above,rotate=-45] {\tiny0};     
\draw[ fill =sussexp](3.5,3)circle(2mm);    
\draw (3.5,3)+(0.3,0.3) node[color=sussexp,inner sep=0.6pt,sloped,above,rotate=-45] {\tiny0};     
\draw[ fill =white](4,2.5)circle(2mm);    
\draw (4,2.5)+(0.3,0.3) node[shape=rectangle,inner sep=0.8pt,minimum size=1pt,draw,sloped,above,rotate=-45] {\tiny1};     
\draw[ fill =white](4.5,2)circle(2mm);    
\draw (4.5,2)+(0.3,0.3) node[shape=rectangle,inner sep=0.8pt,minimum size=1pt,draw,sloped,above,rotate=-45] {\tiny2};     
\draw[ fill =white](5,1.5)circle(2mm);    
\draw (5,1.5)+(0.3,0.3) node[shape=rectangle,inner sep=0.8pt,minimum size=1pt,draw,sloped,above,rotate=-45] {\tiny3};     
\draw[ fill =white](5.5,1)circle(2mm);    
\draw (5.5,1)+(0.3,0.3) node[shape=rectangle,inner sep=0.8pt,minimum size=1pt,draw,sloped,above,rotate=-45] {\tiny4};     

\end{scope}

\begin{scope}[shift={(30,0)}, local bounding box=cc]

\draw[line width=\thick,color=sussexb](2,0)--(6,0);
\draw[line width=\thick,color=sussexb](0,1)--(0,5); 
\draw[fill=mygray](0,0)--(2,0)--(2,1)--(0,1)--cycle;
\draw[line width=\thick,color=sussexb](0,1)--(2,1)--(2,0); 
\draw[line width=1pt](0.5,0.5)--(1.5,0.5);
\draw[fill=red](0.5,0.5)circle(2mm);
\draw[fill=black](1.5,0.5)circle(1.5mm);
\draw[line width=0.5pt](0.5,6)--(6,0.5);
\draw[ fill =sussexp](1,5.5)circle(2mm);
\draw (1,5.5)+(0.3,0.3) node[color=sussexp,inner sep=0.6pt,sloped,above,rotate=-45] {\tiny4};     
\draw[ fill =sussexp](1.5,5)circle(2mm);    
\draw (1.5,5)+(0.3,0.3) node[color=sussexp,inner sep=0.6pt,sloped,above,rotate=-45] {\tiny3};     
\draw[ fill =sussexp](2,4.5)circle(2mm);    
\draw (2,4.5)+(0.3,0.3) node[color=sussexp,inner sep=0.6pt,sloped,above,rotate=-45] {\tiny2};     
\draw[ fill =sussexp](2.5,4)circle(2mm);    
\draw (2.5,4)+(0.3,0.3) node[color=sussexp,inner sep=0.6pt,sloped,above,rotate=-45] {\tiny1};     
\draw[ fill =white](3,3.5)circle(2mm);    
\draw (3,3.5)+(0.3,0.3) node[shape=rectangle,inner sep=0.8pt,minimum size=1pt,draw,sloped,above,rotate=-45] {\tiny0};     
\draw[ fill =white](3.5,3)circle(2mm);  
\draw (3.5,3)+(0.3,0.3) node[shape=rectangle,inner sep=0.8pt,minimum size=1pt,draw,sloped,above,rotate=-45] {\tiny1};       
\draw[ fill =sussexp](4,2.5)circle(2mm);    
\draw (4,2.5)+(0.3,0.3) node[color=sussexp,inner sep=0.6pt,sloped,above,rotate=-45] {\tiny0};     
\draw[ fill =white](4.5,2)circle(2mm);    
\draw (4.5,2)+(0.3,0.3) node[shape=rectangle,inner sep=0.8pt,minimum size=1pt,draw,sloped,above,rotate=-45] {\tiny2};     
\draw[ fill =white](5,1.5)circle(2mm);    
\draw (5,1.5)+(0.3,0.3) node[shape=rectangle,inner sep=0.8pt,minimum size=1pt,draw,sloped,above,rotate=-45] {\tiny3};     
\draw[ fill =white](5.5,1)circle(2mm);    
\draw (5.5,1)+(0.3,0.3) node[shape=rectangle,inner sep=0.8pt,minimum size=1pt,draw,sloped,above,rotate=-45] {\tiny4};     

\end{scope}

\begin{scope}[shift={(15,-7.5)},local bounding box=dd]

\draw[line width=\thick,color=sussexb](2,0)--(6,0);
\draw[line width=\thick,color=sussexb](0,2)--(0,5); 
\draw[fill=mygray](0,0)--(0,2)--(1,2)--(1,1)--(2,1)--(2,0)--cycle;
\draw[line width=\thick,color=sussexb](0,2)--(1,2)--(1,1)--(2,1)--(2,0); 
\draw[line width=1pt](0.5,0.5)--(1.5,0.5);
\draw[line width=1pt](0.5,0.5)--(0.5,1.5);
\draw[fill=red](0.5,0.5)circle(2mm);
\draw[fill=black](1.5,0.5)circle(1.5mm);
\draw[fill=black](0.5,1.5)circle(1.5mm);
\draw[line width=0.5pt](0.5,6)--(6,0.5);
\draw[ fill =sussexp](1,5.5)circle(2mm);
\draw (1,5.5)+(0.3,0.3) node[color=sussexp,inner sep=0.6pt,sloped,above,rotate=-45] {\tiny4};     
\draw[ fill =sussexp](1.5,5)circle(2mm);    
\draw (1.5,5)+(0.3,0.3) node[color=sussexp,inner sep=0.6pt,sloped,above,rotate=-45] {\tiny3};     
\draw[ fill =sussexp](2,4.5)circle(2mm);    
\draw (2,4.5)+(0.3,0.3) node[color=sussexp,inner sep=0.6pt,sloped,above,rotate=-45] {\tiny2};     
\draw[ fill =white](2.5,4)circle(2mm);    
\draw (2.5,4)+(0.3,0.3) node[shape=rectangle,inner sep=0.8pt,minimum size=1pt,draw,sloped,above,rotate=-45] {\tiny0};     
\draw[ fill =sussexp](3,3.5)circle(2mm);    
\draw (3,3.5)+(0.3,0.3) node[color=sussexp,inner sep=0.6pt,sloped,above,rotate=-45] {\tiny1};     
\draw[ fill =white](3.5,3)circle(2mm);  
\draw (3.5,3)+(0.3,0.3) node[shape=rectangle,inner sep=0.8pt,minimum size=1pt,draw,sloped,above,rotate=-45] {\tiny1};       
\draw[ fill =sussexp](4,2.5)circle(2mm);    
\draw (4,2.5)+(0.3,0.3) node[color=sussexp,inner sep=0.6pt,sloped,above,rotate=-45] {\tiny0};     
\draw[ fill =white](4.5,2)circle(2mm);    
\draw (4.5,2)+(0.3,0.3) node[shape=rectangle,inner sep=0.8pt,minimum size=1pt,draw,sloped,above,rotate=-45] {\tiny2};     
\draw[ fill =white](5,1.5)circle(2mm);    
\draw (5,1.5)+(0.3,0.3) node[shape=rectangle,inner sep=0.8pt,minimum size=1pt,draw,sloped,above,rotate=-45] {\tiny3};     
\draw[ fill =white](5.5,1)circle(2mm);    
\draw (5.5,1)+(0.3,0.3) node[shape=rectangle,inner sep=0.8pt,minimum size=1pt,draw,sloped,above,rotate=-45] {\tiny4};     

\end{scope}

\begin{scope}[shift={(22.5,-7.5)},local bounding box=ee]

\draw[line width=\thick,color=sussexb](3,0)--(6,0);
\draw[line width=\thick,color=sussexb](0,2)--(0,5); 
\draw[fill=mygray](0,0)--(0,2)--(1,2)--(1,1)--(3,1)--(3,0)--cycle;
\draw[line width=\thick,color=sussexb](0,2)--(1,2)--(1,1)--(3,1)--(3,0); 
\draw[line width=1pt](0.5,0.5)--(2.5,0.5);
\draw[line width=1pt](0.5,0.5)--(0.5,1.5);
\draw[fill=red](0.5,0.5)circle(2mm);
\draw[fill=black](1.5,0.5)circle(1.5mm);
\draw[fill=black](0.5,1.5)circle(1.5mm);
\draw[fill=black](2.5,0.5)circle(1.5mm);
\draw[line width=0.5pt](0.5,6)--(6,0.5);
\draw[ fill =sussexp](1,5.5)circle(2mm);
\draw (1,5.5)+(0.3,0.3) node[color=sussexp,inner sep=0.6pt,sloped,above,rotate=-45] {\tiny4};     
\draw[ fill =sussexp](1.5,5)circle(2mm);    
\draw (1.5,5)+(0.3,0.3) node[color=sussexp,inner sep=0.6pt,sloped,above,rotate=-45] {\tiny3};     
\draw[ fill =sussexp](2,4.5)circle(2mm);    
\draw (2,4.5)+(0.3,0.3) node[color=sussexp,inner sep=0.6pt,sloped,above,rotate=-45] {\tiny2};     
\draw[ fill =white](2.5,4)circle(2mm);    
\draw (2.5,4)+(0.3,0.3) node[shape=rectangle,inner sep=0.8pt,minimum size=1pt,draw,sloped,above,rotate=-45] {\tiny0};     
\draw[ fill =sussexp](3,3.5)circle(2mm);    
\draw (3,3.5)+(0.3,0.3) node[color=sussexp,inner sep=0.6pt,sloped,above,rotate=-45] {\tiny1};     
\draw[ fill =white](3.5,3)circle(2mm);  
\draw (3.5,3)+(0.3,0.3) node[shape=rectangle,inner sep=0.8pt,minimum size=1pt,draw,sloped,above,rotate=-45] {\tiny1};       
\draw[ fill =white](4,2.5)circle(2mm);    
\draw (4,2.5)+(0.3,0.3) node[shape=rectangle,inner sep=0.8pt,minimum size=1pt,draw,sloped,above,rotate=-45] {\tiny2};     
\draw[ fill =sussexp](4.5,2)circle(2mm);    
\draw (4.5,2)+(0.3,0.3) node[color=sussexp,inner sep=0.6pt,sloped,above,rotate=-45] {\tiny0};     
\draw[ fill =white](5,1.5)circle(2mm);    
\draw (5,1.5)+(0.3,0.3) node[shape=rectangle,inner sep=0.8pt,minimum size=1pt,draw,sloped,above,rotate=-45] {\tiny3};     
\draw[ fill =white](5.5,1)circle(2mm);    
\draw (5.5,1)+(0.3,0.3) node[shape=rectangle,inner sep=0.8pt,minimum size=1pt,draw,sloped,above,rotate=-45] {\tiny4};     

\end{scope}

\begin{scope}[shift={(30,-7.5)},local bounding box=ff]

\draw[line width=\thick,color=sussexb](3,0)--(6,0);
\draw[line width=\thick,color=sussexb](0,2)--(0,5); 
\draw[fill=mygray](0,0)--(0,2)--(2,2)--(2,1)--(3,1)--(3,0)--cycle;
\draw[line width=\thick,color=sussexb](0,2)--(2,2)--(2,1)--(3,1)--(3,0); 
\draw[line width=1pt](0.5,0.5)--(2.5,0.5);
\draw[line width=1pt](0.5,0.5)--(0.5,1.5);
\draw[line width=1pt](1.5,0.5)--(1.5,1.5);
\draw[fill=red](0.5,0.5)circle(2mm);
\draw[fill=black](1.5,0.5)circle(1.5mm);
\draw[fill=black](0.5,1.5)circle(1.5mm);
\draw[fill=black](2.5,0.5)circle(1.5mm);
\draw[fill=black](1.5,1.5)circle(1.5mm);
\draw[line width=0.5pt](0.5,6)--(6,0.5);
\draw[ fill =sussexp](1,5.5)circle(2mm);
\draw (1,5.5)+(0.3,0.3) node[color=sussexp,inner sep=0.6pt,sloped,above,rotate=-45] {\tiny4};     
\draw[ fill =sussexp](1.5,5)circle(2mm);    
\draw (1.5,5)+(0.3,0.3) node[color=sussexp,inner sep=0.6pt,sloped,above,rotate=-45] {\tiny3};     
\draw[ fill =sussexp](2,4.5)circle(2mm);    
\draw (2,4.5)+(0.3,0.3) node[color=sussexp,inner sep=0.6pt,sloped,above,rotate=-45] {\tiny2};     
\draw[ fill =white](2.5,4)circle(2mm);    
\draw (2.5,4)+(0.3,0.3) node[shape=rectangle,inner sep=0.8pt,minimum size=1pt,draw,sloped,above,rotate=-45] {\tiny0};     
\draw[ fill =white](3,3.5)circle(2mm);    
\draw (3,3.5)+(0.3,0.3) node[shape=rectangle,inner sep=0.8pt,minimum size=1pt,draw,sloped,above,rotate=-45] {\tiny1};       
\draw[ fill =sussexp](3.5,3)circle(2mm);    
\draw (3.5,3)+(0.3,0.3) node[color=sussexp,inner sep=0.6pt,sloped,above,rotate=-45] {\tiny1};     
\draw[ fill =white](4,2.5)circle(2mm);    
\draw (4,2.5)+(0.3,0.3) node[shape=rectangle,inner sep=0.8pt,minimum size=1pt,draw,sloped,above,rotate=-45] {\tiny2};     
\draw[ fill =sussexp](4.5,2)circle(2mm);    
\draw (4.5,2)+(0.3,0.3) node[color=sussexp,inner sep=0.6pt,sloped,above,rotate=-45] {\tiny0};     
\draw[ fill =white](5,1.5)circle(2mm);    
\draw (5,1.5)+(0.3,0.3) node[shape=rectangle,inner sep=0.8pt,minimum size=1pt,draw,sloped,above,rotate=-45] {\tiny3};     
\draw[ fill =white](5.5,1)circle(2mm);    
\draw (5.5,1)+(0.3,0.3) node[shape=rectangle,inner sep=0.8pt,minimum size=1pt,draw,sloped,above,rotate=-45] {\tiny4};     

\end{scope}

\begin{scope}[shift={(40,-7.5)},scale=1.1,local bounding box=hh]

\draw[fill=mygray](0,0)--(0,10)--(1,10)--(1,7)--(2,7)--(2,6)--(3,6)--(3,4)--(5,4)--(5,3)--(6,3)--(6,2)--(8,2)--(8,1)--(10,1)--(10,0)--cycle;
\draw[line width=\thick,color=sussexb](0,11.5)--(0,10)--(1,10)--(1,7)--(2,7)--(2,6)--(3,6)--(3,4)--(5,4)--(5,3)--(6,3)--(6,2)--(8,2)--(8,1)--(10,1)--(10,0)--(11.5,0);

\draw[line width=1pt](9.5,0.5)--(0.5,0.5)--(0.5,9.5);     
\draw[line width=1pt](2.5,0.5)--(2.5,4.5); 
\draw[line width=1pt](1.5,0.5)--(1.5,3.5);
\draw[line width=1pt](0.5,4.5)--(1.5,4.5)--(1.5,5.5)--(2.5,5.5); 
\draw[line width=1pt](0.5,6.5)--(1.5,6.5);
\draw[line width=1pt](2.5,1.5)--(4.5,1.5);
\draw[line width=1pt](5.5,1.5)--(7.5,1.5);
\draw[line width=1pt](2.5,2.5)--(4.5,2.5)--(4.5,3.5); 
\draw[line width=1pt](5.5,0.5)--(5.5,1.5);
\draw[line width=1pt](5.5,1.5)--(5.5,2.5); 
\draw[line width=1pt](3.5,2.5)--(3.5,3.5);


\foreach \x in {1,...,9}{
\draw[ fill =black](\x+0.5,0.5)circle(1.7mm);    
\draw[ fill =black](0.5,\x+0.5)circle(1.7mm);    
} 
\draw[fill=red](0.5,0.5)circle(2.3mm);
\foreach \x in {0,...,4}{
\draw[ fill =black](2.5,\x+0.5)circle(1.7mm);    
} 
\foreach \x in {0,...,3}{
\draw[ fill =black](1.5,\x+0.5)circle(1.7mm);    
} 
\draw[ fill =black](2.5,5.5)circle(1.7mm);    
\draw[ fill =black](1.5,5.5)circle(1.7mm);    
\draw[ fill =black](1.5,6.5)circle(1.7mm);    
\draw[ fill =black](1.5,4.5)circle(1.7mm);    
\draw[ fill =black](4.5,1.5)circle(1.7mm);    
\draw[ fill =black](3.5,1.5)circle(1.7mm);    
\draw[ fill =black](3.5,2.5)circle(1.7mm);    
\draw[ fill =black](3.5,3.5)circle(1.7mm);    
\draw[ fill =black](4.5,3.5)circle(1.7mm);    
\draw[ fill =black](4.5,2.5)circle(1.7mm);    
\draw[ fill =black](5.5,2.5)circle(1.7mm);    
\draw[ fill =black](5.5,1.5)circle(1.7mm);    
\draw[ fill =black](6.5,1.5)circle(1.7mm);    
\draw[ fill =black](7.5,1.5)circle(1.7mm);    

\end{scope}

\end{tikzpicture}

%% file: CIF.tex
\begin{tikzpicture}[scale=0.5]

\draw[line width=1pt](9,0)--(0,0)--(0,9);     
\draw[line width=1pt](2,0)--(2,4)--(5,4)--(5,7)--(7,7)--(7,9)--(9,9);
\draw[line width=1pt](1,0)--(1,3);
\draw[line width=1pt](0,4)--(1,4)--(1,5)--(2,5)--(2,7);
\draw[line width=1pt](2,5)--(4,5)--(4,6);
\draw[line width=1pt](3,5)--(3,6)--(3,8)--(5,8)--(5,9)--(6,9);
\draw[line width=1pt](3,7)--(4,7);
\draw[line width=1pt](6,7)--(6,8);
\draw[line width=1pt](0,6)--(1,6)--(1,8)--(2,8)--(2,9)--(3,9);
\draw[line width=1pt](4,8)--(4,9);
\draw[line width=1pt](0,9)--(1,9);
\draw[line width=1pt](5,4)--(7,4)--(7,5)--(9,5)--(9,8);
\draw[line width=1pt](6,4)--(6,6)--(7,6);
\draw[line width=1pt](8,5)--(8,8);
\draw[line width=1pt](2,1)--(8,1)--(8,2)--(9,2);
\draw[line width=1pt](2,2)--(4,2)--(4,3)--(6,3);
\draw[line width=1pt](9,0)--(9,1);

\draw[line width=1pt](5,0)--(5,1);
\draw[line width=2pt,color=white](4,1)--(5,1);

\draw[line width=1pt](5,1)--(5,2)--(7,2)--(7,3);
\draw[line width=1pt](8,2)--(8,4)--(9,4);
\draw[line width=1pt](3,2)--(3,3);
\draw[line width=1pt](8,3)--(9,3);


\draw[line width=2.0pt](.5,.5)--(0.5,3.5)--(1.5,3.5)--(1.5, 4.5)--(4.5,4.5)--(4.5,7.5)--(5.5,7.5)--(5.5,8.5)--(6.5,8.5)--(6.5, 9.5);   
\foreach \x in {0,...,9}{
             \foreach \y in {0,...,9}{
\draw[ fill =sussexp](\x,\y)circle(1.7mm);    
}}

\foreach \x in{0,...,6}{
		\draw[ fill =sussexg](\x, 9)circle(1.7mm);    
				}

\foreach \x in{0,...,5}{
		\draw[ fill =sussexg](\x, 8)circle(1.7mm);    
				}

\foreach \x in{0,...,4}{
		\foreach \y in {5,6,7}{
		\draw[fill =sussexg](\x, \y)circle(1.7mm);    
				}}

\foreach \x in{0,...,1}{
		\draw[ fill =sussexg](\x, 4)circle(1.7mm);    
				}

\foreach \x in{0,...,0}{
		\foreach \y in {1,2,3}{
		\draw[fill =sussexg](\x, \y)circle(1.7mm);    
				}}

\draw[fill=white](0,0)circle(1.7mm); 
\draw[fill=sussexg](0,1)circle(2.5mm); 
\draw[fill=sussexp](1,0)circle(2.5mm);

\end{tikzpicture}

%% file: shape.tex
\def\thickk{4pt}

\begin{tikzpicture}[>=latex,scale=5]


\draw[fill=mygray](0,0)--(-0.01,0.90)--(0.01,0.90)--(0.01,0.77)--(0.01,0.77)--(0.01,0.70)--(0.03,0.70)--(0.03,0.62)--(0.04,0.62)--(0.04,0.60)--(0.04,0.60)--(0.04,0.59)--(0.06,0.59)--(0.06,0.59)--(0.07,0.59)--(0.07,0.55)--(0.07,0.55)--(0.07,0.54)--(0.09,0.54)--(0.09,0.53)--(0.10,0.53)--(0.10,0.49)--(0.10,0.49)--(0.10,0.47)--(0.12,0.47)--(0.12,0.46)--(0.12,0.46)--(0.12,0.43)--(0.14,0.43)--(0.14,0.40)--(0.14,0.40)--(0.14,0.38)--(0.15,0.38)--(0.15,0.36)--(0.17,0.36)--(0.17,0.30)--(0.17,0.30)--(0.17,0.29)--(0.18,0.29)--(0.18,0.28)--(0.20,0.28)--(0.20,0.28)--(0.20,0.28)--(0.20,0.28)--(0.21,0.28)--(0.21,0.28)--(0.23,0.28)--(0.23,0.28)--(0.23,0.28)--(0.23,0.28)--(0.24,0.28)--(0.24,0.26)--(0.26,0.26)--(0.26,0.24)--(0.27,0.24)--(0.27,0.23)--(0.28,0.23)--(0.28,0.23)--(0.28,0.23)--(0.28,0.23)--(0.29,0.23)--(0.29,0.21)--(0.30,0.21)--(0.30,0.21)--(0.32,0.21)--(0.32,0.21)--(0.33,0.21)--(0.33,0.21)--(0.34,0.21)--(0.34,0.20)--(0.34,0.20)--(0.34,0.20)--(0.35,0.20)--(0.35,0.20)--(0.36,0.20)--(0.36,0.20)--(0.38,0.20)--(0.38,0.20)--(0.39,0.20)--(0.39,0.17)--(0.40,0.17)--(0.40,0.17)--(0.41,0.17)--(0.41,0.17)--(0.41,0.17)--(0.41,0.10)--(0.42,0.10)--(0.42,0.10)--(0.43,0.10)--(0.43,0.10)--(0.45,0.10)--(0.45,0.10)--(0.46,0.10)--(0.46,0.10)--(0.47,0.10)--(0.47,0.09)--(0.47,0.09)--(0.47,0.09)--(0.48,0.09)--(0.48,0.09)--(0.49,0.09)--(0.49,0.09)--(0.51,0.09)--(0.51,0.07)--(0.52,0.07)--(0.52,0.07)--(0.53,0.07)--(0.53,0.07)--(0.54,0.07)--(0.54,0.07)--(0.55,0.07)--(0.55,0.07)--(0.56,0.07)--(0.56,0.07)--(0.56,0.07)--(0.56,0.07)--(0.57,0.07)--(0.57,0.07)--(0.58,0.07)--(0.58,0.07)--(0.59,0.07)--(0.59,0.07)--(0.60,0.07)--(0.60,0.07)--(0.61,0.07)--(0.61,0.04)--(0.62,0.04)--(0.62,0.04)--(0.64,0.04)--(0.64,0.04)--(0.65,0.04)--(0.65,0.04)--(0.66,0.04)--(0.66,0.04)--(0.67,0.04)--(0.67,0.03)--(0.68,0.03)--(0.68,0.01)--(0.69,0.01)--(0.69,0.01)--(0.69,0.01)--(0.69,0.01)--(0.70,0.01)--(0.70,0.01)--(0.71,0.01)--(0.71,0.01)--(0.72,0.01)--(0.72,0.01)--(0.73,0.01)--(0.73,0.01)--(0.74,0.01)--(0.74,0.01)--(0.76,0.01)--(0.76,0.01)--(0.77,0.01)--(0.77,0.01)--(0.78,0.01)--(0.78,0.01)--(0.79,0.01)--(0.79,0.01)--(0.80,0.01)--(0.80,0.01)--(0.81,0.01)--(0.81,0.01)--(0.81,0.01)--(0.81,0.01)--(0.82,0.01)--(0.82,0.01)--(0.83,0.01)--(0.83,0.01)--(0.84,0.01)--(0.84,0.01)--(0.85,0.01)--(0.85,0.01)--(0.86,0.01)--(0.86,-0.01)--(0.88,-0.01)--(0.88,-0.01)--(0.89,-0.01)--(0.89,-0.01)--(0.90,-0.01)--(0.90,-0.01)--(0.91,-0.01)--(0.91,-0.01)--(0.92,-0.01)--(0.92,-0.01)--(0.93,-0.01)--(0.93,-0.01)--(0.94,-0.01)--(0.94,-0.01)--(0.94,-0.01)--(0.94,-0.01)--(0.95,-0.01)--(0.95,-0.01)--(0.96,-0.01)--(0.96,-0.01)--(0.97,-0.01)--(0.97,-0.01)--(0.98,-0.01)--(0.98,-0.01)--(0.99,-0.01)--cycle;
\draw[line width=1.5pt,color=sussexb](-0.01,1)--(-0.01,0.90)--(0.01,0.90)--(0.01,0.77)--(0.01,0.77)--(0.01,0.70)--(0.03,0.70)--(0.03,0.62)--(0.04,0.62)--(0.04,0.60)--(0.04,0.60)--(0.04,0.59)--(0.06,0.59)--(0.06,0.59)--(0.07,0.59)--(0.07,0.55)--(0.07,0.55)--(0.07,0.54)--(0.09,0.54)--(0.09,0.53)--(0.10,0.53)--(0.10,0.49)--(0.10,0.49)--(0.10,0.47)--(0.12,0.47)--(0.12,0.46)--(0.12,0.46)--(0.12,0.43)--(0.14,0.43)--(0.14,0.40)--(0.14,0.40)--(0.14,0.38)--(0.15,0.38)--(0.15,0.36)--(0.17,0.36)--(0.17,0.30)--(0.17,0.30)--(0.17,0.29)--(0.18,0.29)--(0.18,0.28)--(0.20,0.28)--(0.20,0.28)--(0.20,0.28)--(0.20,0.28)--(0.21,0.28)--(0.21,0.28)--(0.23,0.28)--(0.23,0.28)--(0.23,0.28)--(0.23,0.28)--(0.24,0.28)--(0.24,0.26)--(0.26,0.26)--(0.26,0.24)--(0.27,0.24)--(0.27,0.23)--(0.28,0.23)--(0.28,0.23)--(0.28,0.23)--(0.28,0.23)--(0.29,0.23)--(0.29,0.21)--(0.30,0.21)--(0.30,0.21)--(0.32,0.21)--(0.32,0.21)--(0.33,0.21)--(0.33,0.21)--(0.34,0.21)--(0.34,0.20)--(0.34,0.20)--(0.34,0.20)--(0.35,0.20)--(0.35,0.20)--(0.36,0.20)--(0.36,0.20)--(0.38,0.20)--(0.38,0.20)--(0.39,0.20)--(0.39,0.17)--(0.40,0.17)--(0.40,0.17)--(0.41,0.17)--(0.41,0.17)--(0.41,0.17)--(0.41,0.10)--(0.42,0.10)--(0.42,0.10)--(0.43,0.10)--(0.43,0.10)--(0.45,0.10)--(0.45,0.10)--(0.46,0.10)--(0.46,0.10)--(0.47,0.10)--(0.47,0.09)--(0.47,0.09)--(0.47,0.09)--(0.48,0.09)--(0.48,0.09)--(0.49,0.09)--(0.49,0.09)--(0.51,0.09)--(0.51,0.07)--(0.52,0.07)--(0.52,0.07)--(0.53,0.07)--(0.53,0.07)--(0.54,0.07)--(0.54,0.07)--(0.55,0.07)--(0.55,0.07)--(0.56,0.07)--(0.56,0.07)--(0.56,0.07)--(0.56,0.07)--(0.57,0.07)--(0.57,0.07)--(0.58,0.07)--(0.58,0.07)--(0.59,0.07)--(0.59,0.07)--(0.60,0.07)--(0.60,0.07)--(0.61,0.07)--(0.61,0.04)--(0.62,0.04)--(0.62,0.04)--(0.64,0.04)--(0.64,0.04)--(0.65,0.04)--(0.65,0.04)--(0.66,0.04)--(0.66,0.04)--(0.67,0.04)--(0.67,0.03)--(0.68,0.03)--(0.68,0.01)--(0.69,0.01)--(0.69,0.01)--(0.69,0.01)--(0.69,0.01)--(0.70,0.01)--(0.70,0.01)--(0.71,0.01)--(0.71,0.01)--(0.72,0.01)--(0.72,0.01)--(0.73,0.01)--(0.73,0.01)--(0.74,0.01)--(0.74,0.01)--(0.76,0.01)--(0.76,0.01)--(0.77,0.01)--(0.77,0.01)--(0.78,0.01)--(0.78,0.01)--(0.79,0.01)--(0.79,0.01)--(0.80,0.01)--(0.80,0.01)--(0.81,0.01)--(0.81,0.01)--(0.81,0.01)--(0.81,0.01)--(0.82,0.01)--(0.82,0.01)--(0.83,0.01)--(0.83,0.01)--(0.84,0.01)--(0.84,0.01)--(0.85,0.01)--(0.85,0.01)--(0.86,0.01)--(0.86,-0.01)--(0.88,-0.01)--(0.88,-0.01)--(0.89,-0.01)--(0.89,-0.01)--(0.90,-0.01)--(0.90,-0.01)--(0.91,-0.01)--(0.91,-0.01)--(0.92,-0.01)--(0.92,-0.01)--(0.93,-0.01)--(0.93,-0.01)--(0.94,-0.01)--(0.94,-0.01)--(0.94,-0.01)--(0.94,-0.01)--(0.95,-0.01)--(0.95,-0.01)--(0.96,-0.01)--(0.96,-0.01)--(0.97,-0.01)--(0.97,-0.01)--(0.98,-0.01)--(0.98,-0.01)--(0.99,-0.01)--(1,-0.01);

\draw[red, smooth, line width=1pt,domain=0:1] plot (\x, {1+\x-2*sqrt(\x)});


\end{tikzpicture}%

%% file: flat.tex
\def\thickk{4pt}

\begin{tikzpicture}[>=latex,scale=8]


\draw[line width=1pt,color=sussexb](0,0.794545454545454)-- (0.0101010101010101,0.839191919191919)-- (0.0202020202020202,0.866868686868686)-- (0.0303030303030303,0.887474747474746)-- (0.0404040404040404,0.900808080808079)-- (0.0505050505050505,0.91141414141414)-- (0.0606060606060606,0.92151515151515)-- (0.0707070707070707,0.93090909090909)-- (0.0808080808080808,0.937171717171716)-- (0.0909090909090909,0.944242424242423)-- (0.101010101010101,0.95070707070707)-- (0.111111111111111,0.957070707070706)-- (0.121212121212121,0.96090909090909)-- (0.131313131313131,0.966060606060605)-- (0.141414141414141,0.969797979797979)-- (0.151515151515152,0.973535353535353)-- (0.161616161616162,0.976666666666666)-- (0.171717171717172,0.979898989898989)-- (0.181818181818182,0.982424242424242)-- (0.191919191919192,0.985050505050505)-- (0.202020202020202,0.986262626262626)-- (0.212121212121212,0.987676767676767)-- (0.222222222222222,0.989191919191919)-- (0.232323232323232,0.990909090909091)-- (0.242424242424242,0.991212121212121)-- (0.252525252525253,0.991717171717172)-- (0.262626262626263,0.992020202020202)-- (0.272727272727273,0.993131313131313)-- (0.282828282828283,0.993939393939394)-- (0.292929292929293,0.994343434343434)-- (0.303030303030303,0.994646464646464)-- (0.313131313131313,0.995252525252525)-- (0.323232323232323,0.995454545454545)-- (0.333333333333333,0.995151515151515)-- (0.343434343434343,0.995959595959596)-- (0.353535353535354,0.995656565656566)-- (0.363636363636364,0.996060606060606)-- (0.373737373737374,0.996060606060606)-- (0.383838383838384,0.995555555555555)-- (0.393939393939394,0.995757575757576)-- (0.404040404040404,0.995757575757576)-- (0.414141414141414,0.995858585858586)-- (0.424242424242424,0.995757575757576)-- (0.434343434343434,0.995656565656566)-- (0.444444444444444,0.995656565656566)-- (0.454545454545455,0.996262626262626)-- (0.464646464646465,0.996363636363636)-- (0.474747474747475,0.996262626262626)-- (0.484848484848485,0.995757575757576)-- (0.494949494949495,0.996464646464646)-- (0.505050505050505,0.995555555555555)-- (0.515151515151515,0.995858585858586)-- (0.525252525252525,0.996161616161616)-- (0.535353535353535,0.996060606060606)-- (0.545454545454545,0.996161616161616)-- (0.555555555555556,0.995858585858586)-- (0.565656565656566,0.995858585858586)-- (0.575757575757576,0.995555555555555)-- (0.585858585858586,0.996262626262626)-- (0.595959595959596,0.996363636363636)-- (0.606060606060606,0.996161616161616)-- (0.616161616161616,0.996161616161616)-- (0.626262626262626,0.996161616161616)-- (0.636363636363636,0.995757575757576)-- (0.646464646464647,0.995959595959596)-- (0.656565656565657,0.995959595959596)-- (0.666666666666667,0.996060606060606)-- (0.676767676767677,0.995858585858586)-- (0.686868686868687,0.995050505050505)-- (0.696969696969697,0.995252525252525)-- (0.707070707070707,0.994444444444444)-- (0.717171717171717,0.994343434343434)-- (0.727272727272727,0.994141414141414)-- (0.737373737373737,0.993838383838384)-- (0.747474747474748,0.992727272727273)-- (0.757575757575758,0.991919191919192)-- (0.767676767676768,0.99060606060606)-- (0.777777777777778,0.989393939393939)-- (0.787878787878788,0.987575757575757)-- (0.797979797979798,0.985858585858586)-- (0.808080808080808,0.983939393939394)-- (0.818181818181818,0.981919191919192)-- (0.828282828282828,0.978989898989899)-- (0.838383838383838,0.974545454545454)-- (0.848484848484849,0.971313131313131)-- (0.858585858585859,0.967878787878787)-- (0.868686868686869,0.963939393939393)-- (0.878787878787879,0.957575757575757)-- (0.888888888888889,0.954343434343433)-- (0.898989898989899,0.949696969696969)-- (0.909090909090909,0.942222222222221)-- (0.919191919191919,0.933939393939393)-- (0.929292929292929,0.927373737373736)-- (0.939393939393939,0.92111111111111)-- (0.94949494949495,0.910606060606059)-- (0.95959595959596,0.898989898989898)-- (0.96969696969697,0.886363636363635)-- (0.97979797979798,0.864646464646464)-- (0.98989898989899,0.838585858585858)-- (1,0.791414141414141);

\draw[->](0,0.7)--(1.2,0.7);
\draw[->](0,0.7)--(0,1.2);
\draw(-.02,0.8)--(0.02,0.8);
\draw(-0.065,0.8)node{$0.8$};
\draw(-.02,0.98)--(0.02,0.98);
\draw(-0.04,0.99)node{$1$};
\draw(1,0.7-.02)--(1,0.7+0.02);
\draw(1,0.7-0.06)node{$1$};
\draw(0,0.7-0.06)node{$0$};

\draw[dashed](0.24,0.7)--(0.24,0.99);
\draw[dashed](1-0.24,0.7)--(1-0.24,0.99);
\draw(0.24,0.7-0.06)node{$t_0$};
\draw(1-0.24,0.7-0.06)node{$1-t_0$};


\draw[line width=1pt,color=red] (0,0.8)-- (0.001,0.825285569006847)-- (0.002,0.835741292645902)-- (0.003,0.843752028524401)-- (0.004,0.85049514828179)-- (0.005,0.856426943918664)-- (0.006,0.861781550644185)-- (0.007,0.866698125910703)-- (0.008,0.871267383844224)-- (0.009,0.875552365945747)-- (0.01,0.87959899496853)-- (0.011,0.883441955873529)-- (0.012,0.887108208568424)-- (0.013,0.890619203262885)-- (0.014,0.893992340113437)-- (0.015,0.897241966249146)-- (0.016,0.900380077704692)-- (0.017,0.903416826483895)-- (0.018,0.90636089506957)-- (0.019,0.909219778428634)-- (0.02,0.912)-- (0.021,0.914707279629499)-- (0.022,0.917346665909177)-- (0.023,0.919922641732077)-- (0.024,0.922439209406138)-- (0.025,0.924899959967968)-- (0.026,0.927308130141009)-- (0.027,0.929666649528705)-- (0.028,0.931978180014728)-- (0.029,0.934245148888144)-- (0.03,0.936469776873856)-- (0.031,0.938654101994856)-- (0.032,0.9408)-- (0.033,0.942909201943052)-- (0.034,0.944983309384218)-- (0.035,0.947023807595913)-- (0.036,0.949032077084096)-- (0.037,0.951009403680698)-- (0.038,0.952956987418032)-- (0.039,0.95487595036028)-- (0.04,0.956767343538123)-- (0.041,0.958632153109009)-- (0.042,0.960471305846248)-- (0.043,0.962285674044261)-- (0.044,0.964076079914167)-- (0.045,0.965843299533023)-- (0.046,0.967588066400923)-- (0.047,0.969311074652546)-- (0.048,0.971012981963359)-- (0.049,0.972694412185224)-- (0.05,0.974355957741627)-- (0.051,0.97599818180879)-- (0.052,0.977621620305637)-- (0.053,0.979226783712703)-- (0.054,0.980814158737639)-- (0.055,0.982384209842848)-- (0.056,0.983937380648959)-- (0.057,0.985474095226261)-- (0.058,0.986994759284853)-- (0.059,0.988499761273058)-- (0.06,0.989989473392607)-- (0.061,0.99146425253817)-- (0.062,0.992924441168039)-- (0.063,0.994370368112015)-- (0.064,0.995802349321963)-- (0.065,0.997220688569937)-- (0.066,0.998625678098276)-- (0.067,1.00001759922567)-- (0.068,1.00139672291276)-- (0.069,1.00276331029059)-- (0.07,1.00411761315477)-- (0.071,1.00545987442807)-- (0.072,1.00679032859397)-- (0.073,1.00810920210313)-- (0.074,1.00941671375513)-- (0.075,1.01071307505705)-- (0.076,1.01199849056066)-- (0.077,1.01327315817983)-- (0.078,1.01453726948948)-- (0.079,1.01579101000737)-- (0.08,1.01703455946001)-- (0.081,1.01826809203363)-- (0.082,1.01949177661133)-- (0.083,1.02070577699734)-- (0.084,1.0219102521291)-- (0.085,1.02310535627815)-- (0.086,1.02429123924041)-- (0.087,1.02546804651657)-- (0.088,1.02663591948321)-- (0.089,1.02779499555521)-- (0.09,1.02894540834007)-- (0.091,1.03008728778444)-- (0.092,1.0312207603136)-- (0.093,1.03234594896404)-- (0.094,1.03346297350972)-- (0.095,1.03457195058233)-- (0.096,1.03567299378588)-- (0.097,1.03676621380594)-- (0.098,1.03785171851387)-- (0.099,1.03892961306628)-- (0.1,1.04)-- (0.101,1.04106297932283)-- (0.102,1.04211864860023)-- (0.103,1.04316710303822)-- (0.104,1.04420843556274)-- (0.105,1.04524273689551)-- (0.106,1.04627009562673)-- (0.107,1.04729059828469)-- (0.108,1.04830432940245)-- (0.109,1.0493113715818)-- (0.11,1.05031180555459)-- (0.111,1.05130571024153)-- (0.112,1.05229316280867)-- (0.113,1.05327423872159)-- (0.114,1.05424901179749)-- (0.115,1.05521755425519)-- (0.116,1.05617993676321)-- (0.117,1.05713622848599)-- (0.118,1.05808649712839)-- (0.119,1.05903080897839)-- (0.12,1.05996922894835)-- (0.121,1.06090182061458)-- (0.122,1.06182864625552)-- (0.123,1.06274976688857)-- (0.124,1.06366524230547)-- (0.125,1.06457513110646)-- (0.126,1.06547949073328)-- (0.127,1.06637837750088)-- (0.128,1.06727184662811)-- (0.129,1.0681599522673)-- (0.13,1.0690427475328)-- (0.131,1.0699202845286)-- (0.132,1.07079261437491)-- (0.133,1.07165978723396)-- (0.134,1.07252185233482)-- (0.135,1.07337885799747)-- (0.136,1.07423085165605)-- (0.137,1.07507787988132)-- (0.138,1.07591998840244)-- (0.139,1.07675722212799)-- (0.14,1.07758962516636)-- (0.141,1.07841724084546)-- (0.142,1.07924011173182)-- (0.143,1.08005827964908)-- (0.144,1.08087178569589)-- (0.145,1.08168067026333)-- (0.146,1.08248497305167)-- (0.147,1.0832847330867)-- (0.148,1.08407998873557)-- (0.149,1.08487077772211)-- (0.15,1.08565713714171)-- (0.151,1.08643910347577)-- (0.152,1.08721671260566)-- (0.153,1.08798999982638)-- (0.154,1.08875899985974)-- (0.155,1.08952374686716)-- (0.156,1.09028427446212)-- (0.157,1.09104061572227)-- (0.158,1.09179280320118)-- (0.159,1.09254086893971)-- (0.16,1.09328484447717)-- (0.161,1.09402476086207)-- (0.162,1.09476064866261)-- (0.163,1.09549253797685)-- (0.164,1.09622045844269)-- (0.165,1.09694443924748)-- (0.166,1.09766450913738)-- (0.167,1.09838069642656)-- (0.168,1.09909302900603)-- (0.169,1.09980153435231)-- (0.17,1.10050623953589)-- (0.171,1.10120717122937)-- (0.172,1.10190435571551)-- (0.173,1.10259781889498)-- (0.174,1.10328758629393)-- (0.175,1.10397368307141)-- (0.176,1.10465613402654)-- (0.177,1.10533496360555)-- (0.178,1.10601019590857)-- (0.179,1.10668185469636)-- (0.18,1.10734996339678)-- (0.181,1.1080145451111)-- (0.182,1.10867562262025)-- (0.183,1.10933321839078)-- (0.184,1.1099873545808)-- (0.185,1.11063805304566)-- (0.186,1.11128533534364)-- (0.187,1.11192922274131)-- (0.188,1.11256973621898)-- (0.189,1.1132068964758)-- (0.19,1.11384072393493)-- (0.191,1.11447123874847)-- (0.192,1.11509846080233)-- (0.193,1.11572240972094)-- (0.194,1.11634310487191)-- (0.195,1.11696056537052)-- (0.196,1.11757481008418)-- (0.197,1.1181858576367)-- (0.198,1.11879372641255)-- (0.199,1.11939843456097)-- (0.2,1.12)-- (0.201,1.12059844042041)-- (0.202,1.12119377328958)-- (0.203,1.12178601585526)-- (0.204,1.12237518514923)-- (0.205,1.12296129799095)-- (0.206,1.12354437099106)-- (0.207,1.12412442055482)-- (0.208,1.12470146288553)-- (0.209,1.12527551398776)-- (0.21,1.12584658967066)-- (0.211,1.12641470555108)-- (0.212,1.12697987705668)-- (0.213,1.12754211942894)-- (0.214,1.12810144772616)-- (0.215,1.12865787682634)-- (0.216,1.12921142143006)-- (0.217,1.12976209606321)-- (0.218,1.13030991507976)-- (0.219,1.13085489266444)-- (0.22,1.13139704283533)-- (0.221,1.13193637944642)-- (0.222,1.13247291619018)-- (0.223,1.13300666659993)-- (0.224,1.13353764405236)-- (0.225,1.1340658617698)-- (0.226,1.13459133282259)-- (0.227,1.13511407013135)-- (0.228,1.13563408646918)-- (0.229,1.13615139446386)-- (0.23,1.13666600660001)-- (0.231,1.13717793522115)-- (0.232,1.13768719253179)-- (0.233,1.13819379059941)-- (0.234,1.13869774135651)-- (0.235,1.13919905660246)-- (0.236,1.13969774800549)-- (0.237,1.14019382710449)-- (0.238,1.1406873053109)-- (0.239,1.14117819391046)-- (0.24,1.141666504065)-- (0.241,1.14215224681419)-- (0.242,1.1426354330772)-- (0.243,1.14311607365438)-- (0.244,1.14359417922893)-- (0.245,1.14406976036845)-- (0.246,1.14454282752656)-- (0.247,1.14501339104446)-- (0.248,1.1454814611524)-- (0.249,1.14594704797122)-- (0.25,1.14641016151378)-- (0.251,1.14687081168643)-- (0.252,1.14732900829041)-- (0.253,1.14778476102325)-- (0.254,1.14823807948012)-- (0.255,1.14868897315516)-- (0.256,1.14913745144284)-- (0.257,1.1495835236392)-- (0.258,1.15002719894317)-- (0.259,1.15046848645777)-- (0.26,1.15090739519138)-- (0.261,1.15134393405892)-- (0.262,1.15177811188304)-- (0.263,1.1522099373953)-- (0.264,1.15263941923727)-- (0.265,1.15306656596172)-- (0.266,1.15349138603366)-- (0.267,1.15391388783149)-- (0.268,1.15433407964801)-- (0.269,1.1547519696915)-- (0.27,1.15516756608677)-- (0.271,1.15558087687613)-- (0.272,1.15599191002044)-- (0.273,1.15640067340004)-- (0.274,1.15680717481575)-- (0.275,1.15721142198984)-- (0.276,1.15761342256688)-- (0.277,1.15801318411478)-- (0.278,1.15841071412557)-- (0.279,1.15880602001639)-- (0.28,1.1591991091303)-- (0.281,1.15958998873717)-- (0.282,1.15997866603453)-- (0.283,1.16036514814837)-- (0.284,1.16074944213401)-- (0.285,1.16113155497685)-- (0.286,1.16151149359322)-- (0.287,1.16188926483111)-- (0.288,1.16226487547097)-- (0.289,1.16263833222648)-- (0.29,1.16300964174523)-- (0.291,1.16337881060953)-- (0.292,1.1637458453371)-- (0.293,1.16411075238174)-- (0.294,1.16447353813411)-- (0.295,1.16483420892235)-- (0.296,1.16519277101279)-- (0.297,1.1655492306106)-- (0.298,1.16590359386046)-- (0.299,1.16625586684721)-- (0.3,1.16660605559647)-- (0.301,1.16695416607527)-- (0.302,1.1673002041927)-- (0.303,1.16764417580046)-- (0.304,1.16798608669351)-- (0.305,1.16832594261062)-- (0.306,1.16866374923499)-- (0.307,1.1689995121948)-- (0.308,1.16933323706377)-- (0.309,1.16966492936171)-- (0.31,1.16999459455511)-- (0.311,1.17032223805761)-- (0.312,1.1706478652306)-- (0.313,1.17097148138368)-- (0.314,1.17129309177522)-- (0.315,1.17161270161285)-- (0.316,1.17193031605396)-- (0.317,1.1722459402062)-- (0.318,1.17255957912796)-- (0.319,1.17287123782882)-- (0.32,1.1731809212701)-- (0.321,1.17348863436522)-- (0.322,1.17379438198025)-- (0.323,1.17409816893431)-- (0.324,1.1744)-- (0.325,1.1746998799039)-- (0.326,1.17499781332696)-- (0.327,1.1752938049049)-- (0.328,1.1755878592287)-- (0.329,1.17587998084495)-- (0.33,1.17617017425628)-- (0.331,1.17645844392177)-- (0.332,1.17674479425733)-- (0.333,1.17702922963611)-- (0.334,1.17731175438886)-- (0.335,1.17759237280432)-- (0.336,1.17787108912961)-- (0.337,1.17814790757057)-- (0.338,1.17842283229213)-- (0.339,1.1786958674187)-- (0.34,1.17896701703446)-- (0.341,1.17923628518379)-- (0.342,1.17950367587153)-- (0.343,1.17976919306337)-- (0.344,1.18003284068617)-- (0.345,1.1802946226283)-- (0.346,1.18055454273993)-- (0.347,1.1808126048334)-- (0.348,1.18106881268348)-- (0.349,1.18132317002773)-- (0.35,1.18157568056678)-- (0.351,1.18182634796462)-- (0.352,1.18207517584894)-- (0.353,1.18232216781139)-- (0.354,1.18256732740787)-- (0.355,1.18281065815883)-- (0.356,1.18305216354956)-- (0.357,1.18329184703043)-- (0.358,1.1835297120172)-- (0.359,1.18376576189129)-- (0.36,1.184)-- (0.361,1.18423242965684)-- (0.362,1.18446305414175)-- (0.363,1.18469187670134)-- (0.364,1.18491890054919)-- (0.365,1.18514412886606)-- (0.366,1.18536756480015)-- (0.367,1.18558921146733)-- (0.368,1.1858090719514)-- (0.369,1.18602714930429)-- (0.37,1.18624344654635)-- (0.371,1.18645796666649)-- (0.372,1.18667071262251)-- (0.373,1.18688168734123)-- (0.374,1.18709089371878)-- (0.375,1.18729833462074)-- (0.376,1.18750401288245)-- (0.377,1.18770793130912)-- (0.378,1.18791009267613)-- (0.379,1.18811049972914)-- (0.38,1.18830915518437)-- (0.381,1.18850606172877)-- (0.382,1.1887012220202)-- (0.383,1.18889463868765)-- (0.384,1.18908631433141)-- (0.385,1.18927625152326)-- (0.386,1.18946445280667)-- (0.387,1.18965092069697)-- (0.388,1.18983565768154)-- (0.389,1.19001866621996)-- (0.39,1.19019994874423)-- (0.391,1.19037950765889)-- (0.392,1.19055734534124)-- (0.393,1.19073346414148)-- (0.394,1.19090786638286)-- (0.395,1.19108055436189)-- (0.396,1.19125153034844)-- (0.397,1.19142079658598)-- (0.398,1.19158835529163)-- (0.399,1.1917542086564)-- (0.4,1.19191835884531)-- (0.401,1.19208080799754)-- (0.402,1.19224155822656)-- (0.403,1.19240061162032)-- (0.404,1.19255797024134)-- (0.405,1.19271363612689)-- (0.406,1.1928676112891)-- (0.407,1.19301989771512)-- (0.408,1.19317049736724)-- (0.409,1.19331941218302)-- (0.41,1.19346664407545)-- (0.411,1.19361219493303)-- (0.412,1.19375606661993)-- (0.413,1.19389826097611)-- (0.414,1.19403877981742)-- (0.415,1.19417762493576)-- (0.416,1.19431479809918)-- (0.417,1.19445030105198)-- (0.418,1.19458413551485)-- (0.419,1.19471630318496)-- (0.42,1.1948468057361)-- (0.421,1.19497564481877)-- (0.422,1.19510282206028)-- (0.423,1.1952283390649)-- (0.424,1.1953521974139)-- (0.425,1.1954743986657)-- (0.426,1.19559494435597)-- (0.427,1.19571383599768)-- (0.428,1.19583107508128)-- (0.429,1.19594666307471)-- (0.43,1.19606060142357)-- (0.431,1.19617289155115)-- (0.432,1.19628353485857)-- (0.433,1.19639253272482)-- (0.434,1.19649988650692)-- (0.435,1.19660559753992)-- (0.436,1.19670966713706)-- (0.437,1.19681209658981)-- (0.438,1.19691288716795)-- (0.439,1.19701204011969)-- (0.44,1.19710955667171)-- (0.441,1.19720543802924)-- (0.442,1.19729968537617)-- (0.443,1.19739229987507)-- (0.444,1.19748328266733)-- (0.445,1.19757263487318)-- (0.446,1.19766035759175)-- (0.447,1.19774645190121)-- (0.448,1.19783091885875)-- (0.449,1.19791375950072)-- (0.45,1.19799497484265)-- (0.451,1.19807456587931)-- (0.452,1.19815253358481)-- (0.453,1.19822887891262)-- (0.454,1.19830360279566)-- (0.455,1.19837670614633)-- (0.456,1.1984481898566)-- (0.457,1.19851805479802)-- (0.458,1.19858630182183)-- (0.459,1.19865293175894)-- (0.46,1.19871794542007)-- (0.461,1.19878134359571)-- (0.462,1.19884312705624)-- (0.463,1.19890329655193)-- (0.464,1.19896185281302)-- (0.465,1.19901879654974)-- (0.466,1.19907412845235)-- (0.467,1.19912784919121)-- (0.468,1.1991799594168)-- (0.469,1.19923045975977)-- (0.47,1.19927935083097)-- (0.471,1.19932663322148)-- (0.472,1.19937230750266)-- (0.473,1.1994163742262)-- (0.474,1.1994588339241)-- (0.475,1.19949968710876)-- (0.476,1.19953893427299)-- (0.477,1.19957657589003)-- (0.478,1.19961261241357)-- (0.479,1.19964704427782)-- (0.48,1.1996798718975)-- (0.481,1.19971109566786)-- (0.482,1.19974071596474)-- (0.483,1.19976873314455)-- (0.484,1.19979514754434)-- (0.485,1.19981995948176)-- (0.486,1.19984316925515)-- (0.487,1.19986477714347)-- (0.488,1.19988478340642)-- (0.489,1.19990318828436)-- (0.49,1.1999199919984)-- (0.491,1.19993519475035)-- (0.492,1.19994879672278)-- (0.493,1.19996079807901)-- (0.494,1.19997119896313)-- (0.495,1.19997999949998)-- (0.496,1.19998719979519)-- (0.497,1.1999927999352)-- (0.498,1.1999967999872)-- (0.499,1.1999991999992)-- (0.5,1.2)-- (0.501,1.1999991999992)-- (0.502,1.1999967999872)-- (0.503,1.1999927999352)-- (0.504,1.19998719979519)-- (0.505,1.19997999949998)-- (0.506,1.19997119896313)-- (0.507,1.19996079807901)-- (0.508,1.19994879672278)-- (0.509,1.19993519475035)-- (0.51,1.1999199919984)-- (0.511,1.19990318828436)-- (0.512,1.19988478340642)-- (0.513,1.19986477714347)-- (0.514,1.19984316925515)-- (0.515,1.19981995948176)-- (0.516,1.19979514754434)-- (0.517,1.19976873314455)-- (0.518,1.19974071596474)-- (0.519,1.19971109566786)-- (0.52,1.1996798718975)-- (0.521,1.19964704427782)-- (0.522,1.19961261241357)-- (0.523,1.19957657589003)-- (0.524,1.19953893427299)-- (0.525,1.19949968710876)-- (0.526,1.1994588339241)-- (0.527,1.1994163742262)-- (0.528,1.19937230750266)-- (0.529,1.19932663322148)-- (0.53,1.19927935083097)-- (0.531,1.19923045975977)-- (0.532,1.1991799594168)-- (0.533,1.19912784919121)-- (0.534,1.19907412845235)-- (0.535,1.19901879654974)-- (0.536,1.19896185281302)-- (0.537,1.19890329655193)-- (0.538,1.19884312705624)-- (0.539,1.19878134359571)-- (0.54,1.19871794542007)-- (0.541,1.19865293175894)-- (0.542,1.19858630182183)-- (0.543,1.19851805479802)-- (0.544,1.1984481898566)-- (0.545,1.19837670614633)-- (0.546,1.19830360279566)-- (0.547,1.19822887891262)-- (0.548,1.19815253358481)-- (0.549,1.19807456587931)-- (0.55,1.19799497484265)-- (0.551,1.19791375950072)-- (0.552,1.19783091885875)-- (0.553,1.19774645190121)-- (0.554,1.19766035759175)-- (0.555,1.19757263487318)-- (0.556,1.19748328266733)-- (0.557,1.19739229987507)-- (0.558,1.19729968537617)-- (0.559,1.19720543802924)-- (0.56,1.19710955667171)-- (0.561,1.19701204011969)-- (0.562,1.19691288716795)-- (0.563,1.19681209658981)-- (0.564,1.19670966713706)-- (0.565,1.19660559753992)-- (0.566,1.19649988650692)-- (0.567,1.19639253272482)-- (0.568,1.19628353485857)-- (0.569,1.19617289155115)-- (0.57,1.19606060142357)-- (0.571,1.19594666307471)-- (0.572,1.19583107508128)-- (0.573,1.19571383599768)-- (0.574,1.19559494435597)-- (0.575,1.1954743986657)-- (0.576,1.1953521974139)-- (0.577,1.1952283390649)-- (0.578,1.19510282206028)-- (0.579,1.19497564481877)-- (0.58,1.1948468057361)-- (0.581,1.19471630318496)-- (0.582,1.19458413551485)-- (0.583,1.19445030105198)-- (0.584,1.19431479809918)-- (0.585,1.19417762493576)-- (0.586,1.19403877981742)-- (0.587,1.19389826097611)-- (0.588,1.19375606661993)-- (0.589,1.19361219493303)-- (0.59,1.19346664407545)-- (0.591,1.19331941218302)-- (0.592,1.19317049736724)-- (0.593,1.19301989771512)-- (0.594,1.1928676112891)-- (0.595,1.19271363612689)-- (0.596,1.19255797024134)-- (0.597,1.19240061162032)-- (0.598,1.19224155822656)-- (0.599,1.19208080799754)-- (0.6,1.19191835884531)-- (0.601,1.1917542086564)-- (0.602,1.19158835529163)-- (0.603,1.19142079658598)-- (0.604,1.19125153034844)-- (0.605,1.19108055436189)-- (0.606,1.19090786638286)-- (0.607,1.19073346414148)-- (0.608,1.19055734534124)-- (0.609,1.19037950765889)-- (0.61,1.19019994874423)-- (0.611,1.19001866621996)-- (0.612,1.18983565768154)-- (0.613,1.18965092069697)-- (0.614,1.18946445280667)-- (0.615,1.18927625152326)-- (0.616,1.18908631433141)-- (0.617,1.18889463868765)-- (0.618,1.1887012220202)-- (0.619,1.18850606172877)-- (0.62,1.18830915518437)-- (0.621,1.18811049972914)-- (0.622,1.18791009267613)-- (0.623,1.18770793130912)-- (0.624,1.18750401288245)-- (0.625,1.18729833462074)-- (0.626,1.18709089371878)-- (0.627,1.18688168734123)-- (0.628,1.18667071262251)-- (0.629,1.18645796666649)-- (0.63,1.18624344654635)-- (0.631,1.18602714930429)-- (0.632,1.1858090719514)-- (0.633,1.18558921146733)-- (0.634,1.18536756480015)-- (0.635,1.18514412886606)-- (0.636,1.18491890054919)-- (0.637,1.18469187670134)-- (0.638,1.18446305414175)-- (0.639,1.18423242965684)-- (0.64,1.184)-- (0.641,1.18376576189129)-- (0.642,1.1835297120172)-- (0.643,1.18329184703043)-- (0.644,1.18305216354956)-- (0.645,1.18281065815883)-- (0.646,1.18256732740787)-- (0.647,1.18232216781139)-- (0.648,1.18207517584894)-- (0.649,1.18182634796462)-- (0.65,1.18157568056678)-- (0.651,1.18132317002773)-- (0.652,1.18106881268348)-- (0.653,1.1808126048334)-- (0.654,1.18055454273993)-- (0.655,1.1802946226283)-- (0.656,1.18003284068617)-- (0.657,1.17976919306337)-- (0.658,1.17950367587153)-- (0.659,1.17923628518379)-- (0.66,1.17896701703446)-- (0.661,1.1786958674187)-- (0.662,1.17842283229213)-- (0.663,1.17814790757057)-- (0.664,1.17787108912961)-- (0.665,1.17759237280432)-- (0.666,1.17731175438886)-- (0.667,1.17702922963611)-- (0.668,1.17674479425733)-- (0.669,1.17645844392177)-- (0.67,1.17617017425628)-- (0.671,1.17587998084495)-- (0.672,1.1755878592287)-- (0.673,1.1752938049049)-- (0.674,1.17499781332696)-- (0.675,1.1746998799039)-- (0.676,1.1744)-- (0.677,1.17409816893431)-- (0.678,1.17379438198025)-- (0.679,1.17348863436522)-- (0.68,1.1731809212701)-- (0.681,1.17287123782882)-- (0.682,1.17255957912796)-- (0.683,1.1722459402062)-- (0.684,1.17193031605396)-- (0.685,1.17161270161285)-- (0.686,1.17129309177522)-- (0.687,1.17097148138368)-- (0.688,1.1706478652306)-- (0.689,1.17032223805761)-- (0.69,1.16999459455511)-- (0.691,1.16966492936171)-- (0.692,1.16933323706377)-- (0.693,1.1689995121948)-- (0.694,1.16866374923499)-- (0.695,1.16832594261062)-- (0.696,1.16798608669351)-- (0.697,1.16764417580046)-- (0.698,1.1673002041927)-- (0.699,1.16695416607527)-- (0.7,1.16660605559647)-- (0.701,1.16625586684721)-- (0.702,1.16590359386046)-- (0.703,1.1655492306106)-- (0.704,1.16519277101279)-- (0.705,1.16483420892235)-- (0.706,1.16447353813411)-- (0.707,1.16411075238174)-- (0.708,1.1637458453371)-- (0.709,1.16337881060953)-- (0.71,1.16300964174523)-- (0.711,1.16263833222648)-- (0.712,1.16226487547097)-- (0.713,1.16188926483111)-- (0.714,1.16151149359322)-- (0.715,1.16113155497685)-- (0.716,1.16074944213401)-- (0.717,1.16036514814837)-- (0.718,1.15997866603453)-- (0.719,1.15958998873717)-- (0.72,1.1591991091303)-- (0.721,1.15880602001639)-- (0.722,1.15841071412557)-- (0.723,1.15801318411478)-- (0.724,1.15761342256688)-- (0.725,1.15721142198984)-- (0.726,1.15680717481575)-- (0.727,1.15640067340004)-- (0.728,1.15599191002044)-- (0.729,1.15558087687613)-- (0.73,1.15516756608677)-- (0.731,1.1547519696915)-- (0.732,1.15433407964801)-- (0.733,1.15391388783149)-- (0.734,1.15349138603366)-- (0.735,1.15306656596172)-- (0.736,1.15263941923727)-- (0.737,1.1522099373953)-- (0.738,1.15177811188304)-- (0.739,1.15134393405892)-- (0.74,1.15090739519138)-- (0.741,1.15046848645777)-- (0.742,1.15002719894317)-- (0.743,1.1495835236392)-- (0.744,1.14913745144284)-- (0.745,1.14868897315516)-- (0.746,1.14823807948012)-- (0.747,1.14778476102325)-- (0.748,1.14732900829041)-- (0.749,1.14687081168643)-- (0.75,1.14641016151378)-- (0.751,1.14594704797122)-- (0.752,1.1454814611524)-- (0.753,1.14501339104446)-- (0.754,1.14454282752656)-- (0.755,1.14406976036845)-- (0.756,1.14359417922893)-- (0.757,1.14311607365438)-- (0.758,1.1426354330772)-- (0.759,1.14215224681419)-- (0.76,1.141666504065)-- (0.761,1.14117819391046)-- (0.762,1.1406873053109)-- (0.763,1.14019382710449)-- (0.764,1.13969774800549)-- (0.765,1.13919905660246)-- (0.766,1.13869774135651)-- (0.767,1.13819379059941)-- (0.768,1.13768719253179)-- (0.769,1.13717793522115)-- (0.77,1.13666600660001)-- (0.771,1.13615139446386)-- (0.772,1.13563408646918)-- (0.773,1.13511407013135)-- (0.774,1.13459133282259)-- (0.775,1.1340658617698)-- (0.776,1.13353764405236)-- (0.777,1.13300666659993)-- (0.778,1.13247291619018)-- (0.779,1.13193637944642)-- (0.78,1.13139704283533)-- (0.781,1.13085489266444)-- (0.782,1.13030991507976)-- (0.783,1.12976209606321)-- (0.784,1.12921142143006)-- (0.785,1.12865787682634)-- (0.786,1.12810144772616)-- (0.787,1.12754211942894)-- (0.788,1.12697987705668)-- (0.789,1.12641470555108)-- (0.79,1.12584658967066)-- (0.791,1.12527551398776)-- (0.792,1.12470146288552)-- (0.793,1.12412442055482)-- (0.794,1.12354437099106)-- (0.795,1.12296129799095)-- (0.796,1.12237518514923)-- (0.797,1.12178601585526)-- (0.798,1.12119377328958)-- (0.799,1.12059844042041)-- (0.8,1.12)-- (0.801,1.11939843456097)-- (0.802,1.11879372641255)-- (0.803,1.1181858576367)-- (0.804,1.11757481008418)-- (0.805,1.11696056537052)-- (0.806,1.11634310487191)-- (0.807,1.11572240972094)-- (0.808,1.11509846080233)-- (0.809,1.11447123874847)-- (0.81,1.11384072393493)-- (0.811,1.1132068964758)-- (0.812,1.11256973621898)-- (0.813,1.11192922274131)-- (0.814,1.11128533534364)-- (0.815,1.11063805304566)-- (0.816,1.1099873545808)-- (0.817,1.10933321839078)-- (0.818,1.10867562262025)-- (0.819,1.1080145451111)-- (0.82,1.10734996339678)-- (0.821,1.10668185469636)-- (0.822,1.10601019590857)-- (0.823,1.10533496360555)-- (0.824,1.10465613402654)-- (0.825,1.10397368307141)-- (0.826,1.10328758629393)-- (0.827,1.10259781889498)-- (0.828,1.10190435571551)-- (0.829,1.10120717122937)-- (0.83,1.10050623953589)-- (0.831,1.09980153435231)-- (0.832,1.09909302900603)-- (0.833,1.09838069642656)-- (0.834,1.09766450913738)-- (0.835,1.09694443924748)-- (0.836,1.09622045844269)-- (0.837,1.09549253797685)-- (0.838,1.09476064866261)-- (0.839,1.09402476086207)-- (0.84,1.09328484447717)-- (0.841,1.09254086893971)-- (0.842,1.09179280320118)-- (0.843,1.09104061572227)-- (0.844,1.09028427446212)-- (0.845,1.08952374686716)-- (0.846,1.08875899985974)-- (0.847,1.08798999982638)-- (0.848,1.08721671260566)-- (0.849,1.08643910347577)-- (0.85,1.08565713714171)-- (0.851,1.08487077772211)-- (0.852,1.08407998873557)-- (0.853,1.0832847330867)-- (0.854,1.08248497305167)-- (0.855,1.08168067026333)-- (0.856,1.08087178569589)-- (0.857,1.08005827964908)-- (0.858,1.07924011173182)-- (0.859,1.07841724084546)-- (0.86,1.07758962516636)-- (0.861,1.07675722212799)-- (0.862,1.07591998840244)-- (0.863,1.07507787988132)-- (0.864,1.07423085165605)-- (0.865,1.07337885799747)-- (0.866,1.07252185233482)-- (0.867,1.07165978723396)-- (0.868,1.07079261437491)-- (0.869,1.0699202845286)-- (0.87,1.0690427475328)-- (0.871,1.0681599522673)-- (0.872,1.06727184662811)-- (0.873,1.06637837750088)-- (0.874,1.06547949073328)-- (0.875,1.06457513110646)-- (0.876,1.06366524230547)-- (0.877,1.06274976688857)-- (0.878,1.06182864625552)-- (0.879,1.06090182061458)-- (0.88,1.05996922894835)-- (0.881,1.05903080897839)-- (0.882,1.05808649712839)-- (0.883,1.05713622848599)-- (0.884,1.05617993676321)-- (0.885,1.05521755425519)-- (0.886,1.05424901179749)-- (0.887,1.05327423872159)-- (0.888,1.05229316280867)-- (0.889,1.05130571024153)-- (0.89,1.05031180555459)-- (0.891,1.0493113715818)-- (0.892,1.04830432940245)-- (0.893,1.04729059828469)-- (0.894,1.04627009562673)-- (0.895,1.04524273689551)-- (0.896,1.04420843556274)-- (0.897,1.04316710303822)-- (0.898,1.04211864860023)-- (0.899,1.04106297932283)-- (0.9,1.04)-- (0.901,1.03892961306628)-- (0.902,1.03785171851387)-- (0.903,1.03676621380594)-- (0.904,1.03567299378588)-- (0.905,1.03457195058233)-- (0.906,1.03346297350972)-- (0.907,1.03234594896404)-- (0.908,1.0312207603136)-- (0.909,1.03008728778444)-- (0.91,1.02894540834007)-- (0.911,1.02779499555521)-- (0.912,1.02663591948321)-- (0.913,1.02546804651657)-- (0.914,1.02429123924041)-- (0.915,1.02310535627815)-- (0.916,1.0219102521291)-- (0.917,1.02070577699734)-- (0.918,1.01949177661133)-- (0.919,1.01826809203363)-- (0.92,1.01703455946001)-- (0.921,1.01579101000737)-- (0.922,1.01453726948948)-- (0.923,1.01327315817983)-- (0.924,1.01199849056066)-- (0.925,1.01071307505705)-- (0.926,1.00941671375513)-- (0.927,1.00810920210313)-- (0.928,1.00679032859397)-- (0.929,1.00545987442807)-- (0.93,1.00411761315477)-- (0.931,1.00276331029059)-- (0.932,1.00139672291276)-- (0.933,1.00001759922567)-- (0.934,0.998625678098276)-- (0.935,0.997220688569937)-- (0.936,0.995802349321963)-- (0.937,0.994370368112014)-- (0.938,0.992924441168039)-- (0.939,0.99146425253817)-- (0.94,0.989989473392607)-- (0.941,0.988499761273058)-- (0.942,0.986994759284853)-- (0.943,0.985474095226261)-- (0.944,0.983937380648959)-- (0.945,0.982384209842848)-- (0.946,0.980814158737639)-- (0.947,0.979226783712703)-- (0.948,0.977621620305637)-- (0.949,0.97599818180879)-- (0.95,0.974355957741627)-- (0.951,0.972694412185224)-- (0.952,0.971012981963359)-- (0.953,0.969311074652546)-- (0.954,0.967588066400923)-- (0.955,0.965843299533023)-- (0.956,0.964076079914167)-- (0.957,0.962285674044261)-- (0.958,0.960471305846248)-- (0.959,0.958632153109009)-- (0.96,0.956767343538124)-- (0.961,0.95487595036028)-- (0.962,0.952956987418032)-- (0.963,0.951009403680698)-- (0.964,0.949032077084096)-- (0.965,0.947023807595913)-- (0.966,0.944983309384219)-- (0.967,0.942909201943052)-- (0.968,0.9408)-- (0.969,0.938654101994856)-- (0.97,0.936469776873856)-- (0.971,0.934245148888144)-- (0.972,0.931978180014728)-- (0.973,0.929666649528705)-- (0.974,0.927308130141009)-- (0.975,0.924899959967968)-- (0.976,0.922439209406138)-- (0.977,0.919922641732077)-- (0.978,0.917346665909177)-- (0.979,0.914707279629499)-- (0.98,0.912)-- (0.981,0.909219778428635)-- (0.982,0.90636089506957)-- (0.983,0.903416826483895)-- (0.984,0.900380077704692)-- (0.985,0.897241966249146)-- (0.986,0.893992340113437)-- (0.987,0.890619203262885)-- (0.988,0.887108208568424)-- (0.989,0.883441955873529)-- (0.99,0.87959899496853)-- (0.991,0.875552365945747)-- (0.992,0.871267383844224)-- (0.993,0.866698125910703)-- (0.994,0.861781550644185)-- (0.995,0.856426943918664)-- (0.996,0.850495148281791)-- (0.997,0.843752028524401)-- (0.998,0.835741292645902)-- (0.999,0.825285569006847)-- (1,0.8);

\end{tikzpicture}